\documentclass[11pt]{amsart}
\usepackage{graphicx} 

\usepackage[T1]{fontenc}
\usepackage{amssymb,amsmath, amsthm, amsfonts, tikz-cd}
\usepackage{graphicx}
\usepackage{listings}
\usepackage{lstautogobble}
\usepackage{enumerate}
\usepackage[shortlabels]{enumitem}
\usepackage{thmtools}
\usepackage{thm-restate}
\usepackage{amsthm}
\usepackage{verbatim}

\usepackage{mathtools}
\usepackage{physics}
\usepackage{color}
\usepackage{hyperref}
\usepackage[capitalise]{cleveref}
\usepackage[bottom]{footmisc}
\crefformat{equation}{(#2#1#3)}

\usepackage{float}
\restylefloat{table}

\usepackage{mathrsfs}
\setlist{  
  listparindent=\parindent,
  parsep=0pt,
}
\theoremstyle{plain}
\newtheorem{thm}{Theorem}[section]
\newtheorem{prop}[thm]{Proposition}
\newtheorem{lemma}[thm]{Lemma}
\newtheorem{cor}[thm]{Corollary}

\theoremstyle{definition}
\newtheorem{mydef}[thm]{Definition}

\newtheorem{remark}[thm]{Remark}

\newtheorem*{sum-invol}{Summary of result 1 - $\mathcal{H}_{n}$ is involution}
\newtheorem*{sum-GP_ham}{Summary of result 2 - GP Hamiltonian flows}

\numberwithin{equation}{section} 

\newcommand{\p}{{\partial}}

\newcommand{\R}{{\mathbb{R}}}
\newcommand{\C}{{\mathbb{C}}}
\newcommand{\N}{{\mathbb{N}}}

\newcommand{\Z}{{\mathbb{Z}}}

\newcommand{\T}{{\mathbb{T}}}

\newcommand{\Sc}{{\mathcal{S}}}
\newcommand{\A}{{\mathcal{A}}}

\newcommand{\ol}{\overline}

\renewcommand{\Re}{\;\mathrm{Re}\;}

\newcommand{\E}{{\mathbf{E}}}

\def\XXint#1#2#3{{\setbox0=\hbox{$#1{#2#3}{\int}$ }
\vcenter{\hbox{$#2#3$ }}\kern-.6\wd0}}

\let\oldtocsection=\tocsection
 
\let\oldtocsubsection=\tocsubsection
 
\let\oldtocsubsubsection=\tocsubsubsection
 
\renewcommand{\tocsection}[2]{\hspace{0em}\oldtocsection{#1}{#2}}
\renewcommand{\tocsubsection}[2]{\hspace{1em}\oldtocsubsection{#1}{#2}}
\renewcommand{\tocsubsubsection}[2]{\hspace{2em}\oldtocsubsubsection{#1}{#2}}

\title[Invariant measures for complex valued mKdV]{Well-posedness and invariant measures for complex valued modified KdV equation}

\author[Kenig]{Carlos E. Kenig$^1$}
\address{$^1$  Department of Mathematics\\
University of Chicago\\
Eckhart Hall, 5734 S University Ave\\
Chicago, IL 60637}
\email{cek@uchicago.edu}
\thanks{$^1$ The first author was funded in part by NSF DMS-2153794 and DMS-2052710.}

\author[Nahmod]{Andrea R. Nahmod$^2$}
\address{$^2$  
Department of Mathematics \\ University of Massachusetts\\  710 N. Pleasant Street, Amherst MA 01003}
\email{nahmod@umass.edu}
\thanks{$^2$ The second author was funded in part by NSF DMS-2052740, NSF-DMS 2400036 and the Simons Foundation 
Collaboration Grant on Wave Turbulence (Nahmod’s award ID 651469).}

\author[Pavlovi\'c]{Nata\v{s}a Pavlovi\'c$^3$}
\address{$^3$  
Department of Mathematics\\ 
University of Texas at Austin\\ 
2515 Speedway, Stop C1200\\
Austin, TX 78712}
\email{natasa@math.utexas.edu}
\thanks{$^3$ The third author was funded in part by NSF DMS-1840314, DMS-2009549 and DMS-2052789.}

\author[Staffilani]{Gigliola Staffilani$^4$}
\address{$^4$ Department of Mathematics\\
Massachusetts Institute of Technology\\ 
77 Massachusetts Avenue,  Cambridge, MA 02139}
\email{gigliola@math.mit.edu}
\thanks{$^4$ The fourth  author was funded in part by NSF DMS-2052651, 
DMS-2306378 and the Simons Foundation
Collaboration Grant on Wave Turbulence.}

\author[Visciglia]{Nicola Visciglia$^5$}
\address{$^5$ Dipartimento di Matematica\\Universit\`a di Pisa\\
Largo Bruno Pontecorvo, 5, 56127,Italy}
\email{nicola.visciglia@unipi.it}
\thanks{$^5$ The fifth  author was funded in part by  PRIN 2020XB3EFL}

\begin{document}

\maketitle

\begin{abstract} We consider the one dimensional periodic complex valued mKdV, which corresponds to the first equation above cubic NLS
in the associated integrable hierarchy. Our main result is the construction of a sequence of invariant measures supported on Sobolev spaces with increasing regularity.
The fact that we work with complex valued functions makes the analysis of the invariance much harder compared to the real valued case, that can be handled instead
following the ideas used by  Zhidkov \cite{ZhLN}.  
\end{abstract}

\tableofcontents

\section{Introduction}

\subsection{Our results in a nutshell}

As an integrable system the one dimensional periodic cubic nonlinear Schr\"odinger  equation (NLS) 
\begin{equation}\label{NLS}
 i\partial_t u +\partial_x^2 u= 2 |u|^2 u, \quad (t,x)\in   \mathbb R \times\T , \quad u(t,x)\in \mathbb{C},
\end{equation}
where $\T=\R/(2\pi \Z)$, admits an infinite list  of conserved integrals. For simplicity we have chosen the defocusing nonlinearity, however all the results
of this paper extend to the focusing case, after minor changes. 
Following \cite{FaTa2007}, one can define inductively a sequence of  conservation laws satisfied by solutions to \eqref{NLS} as follows: 
\begin{equation} \label{in-In}
E_{n}(u) := \int_{\T}\bar u w_{n}[u] dx, \quad \forall n\in \N
\end{equation}
where the sequence of operators  $\{w_n\}_{n=1}^{\infty}$ is such that $w_{n}:\Sc(\T) \rightarrow \Sc(\T)$ and 
\begin{align*} 
w_{1}[u] & := u \\
w_{n+1}[u] & := -i\p_{x}w_{n}[u] + \bar{u}\sum_{k=1}^{n-1}w_{k}[u]w_{n-k}[u].
\end{align*}
Hence, we get for instance for $n=1,2,3,4,5$ the following explicit conservation laws:
\begin{multline}\label{consexpl}
E_1(u) = \int_\T |u|^2 dx, \quad 
E_2(u) = \Im \int _\T\bar u \partial_x u dx, \quad 
E_3(u) =  \int_{\T} \Big(|\partial_{x}u|^2 + |u|^4\Big )dx,\\
E_4(u) = \Im \int_{\T} \Big(\partial_x u \partial_{xx} \bar u+3 |u|^2u \partial_x \bar u\Big)dx, \\
E_5(u) = \int_\T \Big( |\partial_{xx} u|^2 + 6 |\partial_x  u|^2 |u|^2 + |\partial_x (|u|^2)|^2 + 2 |u|^6 \Big)dx.
\end{multline} 
 We also mention that in the paper  \cite{KT18} other conservation laws at lower regularity have been discovered for solutions to \eqref{NLS}
(see also \cite{KVZ18} for results in the same direction). However in this paper we shall focus only on the aforementioned conservation laws $E_n(u), n\in \N$. 
\begin{remark} 
Notice that for $n=1,3,5$ the quadratic term of the conservation law is positive definite. This fact is true in general, namely it is easy to check that
\begin{equation}\label{energybasic}
E_{2n+1}(u)=\|\partial_x^n u\|_{L^2(\T)}^2 + R_n(u), 
\end{equation}
where $R_{n}(u)$ is the integral of a linear combination of densities involving at most $2n-2$ derivatives, with super quadratic homogeneity (see Section \ref{RNU} for more details on the structure of $R_n(u)$). 
\end{remark}

By introducing the symplectic $L^2$ scalar product
\begin{equation}\label{eq:Hstr}
\omega_{L^2}(f,g) = 2\Im\int_{\T}\ol{f(x)}g(x) dx,
\end{equation}
we can consider  for each of the conserved energies $E_n(u)$ the associated Hamiltonian flow:
\begin{equation}\label{nNLS}
\frac{\partial}{\partial t} u(t) = \grad E_n(u(t)), 
\end{equation}
where $\grad $ denotes the symplectic $L^2$ gradient. 
Inspired by \cite{FaTa2007}, we refer to \eqref{nNLS} as the $n$-th nonlinear Schr\"{o}dinger equation, and we shall denote it by $(nNLS)$.
The {\it nonlinear Schr\"{o}dinger hierarchy}, named by Palais \cite{Pa1997}, refers to the whole family of $nNLS$ as $n\in \N$. 
For $n=3$, \eqref{nNLS}
reduces to \eqref{NLS}. It is well-known that the full set of aforementioned energies $E_n(u)$, are conserved by any single equation in the hierarchy.

The notion of NLS hierarchy mentioned above is closely linked to the fact that the one-dimensional cubic NLS equation \eqref{NLS} is an integrable equation. The notion of {\em  integrable equation} can have different meanings, one of which is that the equation can be exactly solved by the inverse scattering methods. For the case of the 1D cubic NLS equation \eqref{NLS}, this notion was formally carried out by Zakharov and Shabat in \cite{ZaSh71} and mathematically addressed in e.g. 
\cite{BeCo1984, BeCo1985, BeCo1987, TeUh1998, Zh1998, Zh1989, DeZh2003}. We also recall \cite{GK}, where the associated action-angle variables are introduced for  
cubic NLS.

This paper is devoted to the analysis, from several viewpoints, of \eqref{nNLS} for $n=4$, namely
the complex-valued  periodic modified Korteweg De Vries (mKdV): 
\begin{equation}\label{mKdV0}
 \partial_t u+\partial_x^3 u= 6|u|^2\partial_x u, \quad (t,x)\in \R\times \T, \quad u(t,x)\in \C.
 \end{equation}
More precisely, the main focus of this article is a probabilistic analysis of the Cauchy problem associated with \eqref{mKdV0}. Our goal is to construct a sequence of invariant measures (which are supported on spaces with increasing regularity) 
each of which is associated to one of infinitely many conservation laws (starting with $E_5$) for the cubic NLS, which in turn are invariant with respect to the flow associated to \eqref{mKdV0} as well. 
This direction is inspired by and generalizes a pioneering work of Zhidkov \cite{ZhLN, Zh2001}, where 
the author considered the Cauchy problem for \eqref{nNLS} for $n=3$, i.e.  the cubic NLS \eqref{NLS}, and the real valued KdV equation  and constructed an infinite sequence of invariant measures for their  flows. 

Although, as far as we know, it is not explicitly written elsewhere, the argument used by  Zhidkov  should extend to the real valued mKdV equation. In our work we are interested in \eqref{mKdV0}, which is the  complex valued mKdV, and in this case the proof of the existence of infinitely many  invariant measures in the spirit of  Zhidkov is more involved and 
requires new ideas, since we consider complex valued solutions (see comments after Theorem \ref{probBBB} for more details on this point).

On the other hand it is well-known that probabilistic analysis typically relies on a deterministic local Cauchy theory associated with \eqref{mKdV0}.
It is worth mentioning that the Cauchy problem associated with \eqref{mKdV0} has been studied in \cite{B1993}, \cite{Ch21}  and \cite{Ch23}, 
for initial data up to regularity $H^\frac 12(\T)$ as well as in the Fourier-Lebesgue spaces. 
However we present our proof  of local well-posedness, at regularity above $H^\frac 12(\T)$, which is enough for our purposes.
We have decided to include our analysis of the deterministic Cauchy theory for several reasons: \\
First of all, we would like to make this paper self-contained; the second reason is that our argument allows us to obtain, in a relatively elementary way, suitable estimates for solutions to  the Cauchy problem associated to \eqref{mKdV0}, as well as to its  finite dimensional approximation
\eqref{mKdVN}, which are well tailored to applications to  the subsequent probabilistic result; finally, we mention that our proof of well-posedness contains   unconditional uniqueness of strong solutions (see Definition \ref{defsol}  for the definition of strong solution). As far as we know, the unconditional uniqueness of strong solutions for the Cauchy  problem associated with complex-valued mKdV is not proved elsewhere at regularity $s\leq \frac 32$; at regularity $s>\frac 32 $ see \cite{Ka83}, \cite{Ts81}  
and the texts \cite{LP15} and \cite{BS75}. Hence Theorem \ref{Main} below seems to be the first result in this direction, since we obtain unconditional uniqueness of strong solutions up to regularity $s>\frac 43$. 
In this paper we  do not pursue  the issue of unconditional uniqueness of strong solutions at regularity $s\leq \frac 43$,
since it would take us in a different direction than what is the aim of this  paper.\\
\\
The paper has two results: 
\begin{enumerate}
\item[(i)] First we  prove {\it unconditional} well-posedness of  strong  solutions for the initial value problem associated with \eqref{mKdV0}, namely 
\begin{equation}\label{data}
\begin{cases} 
\partial_t u+\partial_x^3 u= 6 |u|^2\partial_x u, \quad (t,x)\in \R\times \T, \quad u(t,x)\in \C\\
u(0)=u_0,
\end{cases}
\end{equation}
for initial data $u_0\in H^s(\T)$, $s>4/3$.  Our argument is based on the strategies developed in \cite{KT03}, \cite{K2004} and \cite{IK2007}. Moreover our result comes with a-priori bounds for  the associated truncated initial value problem \eqref{mKdVN} below, which are uniform in $N$, and which are crucial for the second result of the paper.
\item[(ii)] We show that the local solutions obtained in Theorem \ref{Main} can be extended to global solutions, with a logarithmic upper bound on the growth of the fractional Sobolev norms,
for almost every initial datum with respect  to the Gaussian measure $\mu_{n}$ defined in Subsection \ref{prob}. More importantly we show that a sequence of weighted Gaussian measures, associated with  
the conservation laws $E_{n}(u)$, are invariant along the flow \eqref{data}.
\end{enumerate} 
Precise formulation of these results is presented in Subsections \ref{subsec-wp} and \ref{prob} below.

\subsection{The deterministic result} \label{subsec-wp} 

We are now ready to state our unconditional deterministic local  well-posedness result of strong solutions. 
\begin{thm} \label{Main} Assume $u_0\in H^s(\T)$ for $s>4/3$. Then the Cauchy problem \eqref{data} admits a unique local strong solution
$u\in {\mathcal C}([-T,T],H^s(\T))$, where $T=T(\|u_0\|_{H^s(\T)})>0$.
\end{thm}

One could also prove continuity of the solution mapping with respect to the initial date. We do not do it here because we do not need it, but this can be done following \cite{BS75} and \cite{ABITZ24}. 

Using the conserved energy $E_5(u)$, which allows  uniform control  in time of the $H^2(\T)$ norm of the solution,  we immediately obtain the following corollary.

\begin{cor}\label{COR} The initial value problem \eqref{data} is globally and unconditionally well-posed for strong solutions in  ${\mathcal C}(\R,H^s(\T))$ for $s\geq 2$. Moreover we have an exponential upper bound for  the $H^s(\T)$ norms of solutions.
\end{cor}
\begin{remark} The exponential upper bound mentioned in the corollary can be updated to a uniform in time bound when $s\in \N$. This fact easily  follows from the conserved energies $E_{2k+1}$.  To the best of our knowledge a uniform bound is not known for $s$ fractional. The aforementioned exponential upper bound does not exploit the integrability of the equation but follows from the iteration of the local in time estimates. It is quite interesting to notice that by exploiting the invariance of the
measures associated to the conserved energies one can update almost surely the exponential bound  to a logarithmic one, (see Theorem \ref{probBBB}).
We recall that the deterministic uniform control for every time of the $H^s(\T)$ norm, with $s$ fractional, has been established for the flow associated with the cubic NLS in \cite{KST},
by using the action-angle variables. The same property has proved for the Benjamin-Ono equation in \cite{GK}.
To the  best of our knowledge this property, as well as the action-angle variables,  have not been identified yet  for complex valued mKdV. 
\end{remark}
The proof of Theorem \ref{Main} is inspired by the papers \cite{KT03}, \cite{K2004} and \cite{IK2007}. 

There have been  many contributions  by many authors on the study of the Cauchy  initial value problem for real and complex valued KdV and mKdV on both the line and the torus. Here we limit ourselves to recalling those contributions obtained for the  periodic mKdV. 

Using the classical energy method unconditional well-posedness in $H^s,  s>\frac 32$, for  both real and complex mKdV can be found  in \cite{Ka83, Ts81}. 
 It was Bourgain  \cite{B1993} who, using  techniques involving analytic number theory,  obtained  local well-posedness of analytic type for strong solutions for  \eqref{data} in $H^s, s\geq \frac12$. This result was stated for the  real valued mKdV, but the argument applies in the complex case as well, see also \cite{CKSTT04}. For real valued solutions this well-posedness was extended to global in time in \cite{CKSTT03}. That in the periodic case $s=\frac12$ is the threshold for analytic well-posedness  was proved in  \cite{CCT03}. For the real valued \eqref{data} in \cite{KT05} the authors proved well-posedness with continuity with respect to the initial data all the way to $L^2$ and in \cite{M12} it was proved that this result is sharp. Here the notion of solution is not as strong solution, but as the limit of smooth solutions. It turns out that by {\it renormalizing} the equation in an appropriate manner one can push this type of well-posedness for real valued solutions  even below $L^2$, 
as for example in \cite{KM17}.
More recently in \cite{Ch21,Ch23}, again via renormalization,  local well-posedness of strong solutions was proved for the complex valued mKdV in a larger class of Fourier Lebesgue spaces as well as global well-posedness in a smaller set using the ideas of \cite{OW21} based on work in \cite{KVZ18}.

All the results mentioned above, except for  \cite{KT05} and   \cite{KM17},  feature conditional uniqueness since they are proved using a contraction method. 
 The results  in  \cite{KT05} and   \cite{KM17} that use inverse scattering produce solutions that are limits of any sequence of smooth solutions with data approximating the given initial data.  The uniqueness of such solutions is immediate from their definition.
We finally  recall the papers \cite{KO} and \cite{MPV}, where unconditional uniqueness is established for distributional solutions of the real valued mKdV, up to regularity $H^{1/2}(\T) $ and $H^{1/3}(\T)$  respectively.
  However, to the best of our knowledge, before our Theorem \ref{Main} there seems to be no  results in the literature about unconditional uniqueness of strong solutions below the classical regularity $s>\frac 32$ for the complex valued mKdV \eqref{data} on $\T$.

\subsection{The probabilistic result}\label{prob}
A basic but fundamental question is how does randomness propagate under nonlinear dispersive and wave flows. Such questions are physically and mathematically relevant both in equilibrium statistical mechanics and constructive quantum field theory (QFT), e.g.  invariance of Gibbs measures under the flow of the equation and Poincar\'e recurrence,  as well as in the study of out-of-equilibrium dynamics (e.g. turbulence).  Seminal works by  Bourgain in the '90s \cite{Bou2, Bou3} to prove the invariance of Gibbs measures for certain Hamiltonian PDE introduced novel probabilistic and analytic methods. At the root of these works were ideas by  Lebowitz, Rose and Speer \cite{LRS88}  who introduced the concept of studying the NLS from a statistical mechanics perspective via appropriate ensembles on phase space, in particular via Gibbs measures. Their perspective was a drastic shift from the previous viewpoint of analyzing a particular microscopic dynamical trajectory. 

The Hamiltonian associated with complex valued mKdV does not allow for a construction of the Gibbs measure, with the Gaussian as a reference measure, since its quadratic part is real valued but not signed definite (see $E_4$ in \eqref{consexpl}). Therefore the question of existence of the Gibbs measure in this context is a question that needs further investigation. 

In this paper we focus instead on the construction and the invariance of weighted Gaussian measures associated to the higher conservation laws of the complex valued mKdV \eqref{mKdV0}, whose statistical ensembles correspond to higher and higher regularity spaces. To obtain the invariance we need to understand the global dynamics of the equation on the statistical ensemble of the measure. A deterministic global well-posedness result is established in Corollary 1.2, which yields the usual exponential upper bound in time on the Sobolev norm, in both integer and fractional case.  However, by exploiting the invariance of the associated measures and the deterministic local well-posedness, we are able to establish a better quantitive global dynamics result by proving that almost surely solutions grow at most logarithmically in time,  see Theorem \ref{probBBB}. More importantly we get the invariance of a sequence of weighted Gaussian measures, which are relevant in order to study the dynamical system associated with \eqref{mKdV0} from the viewpoint of statistical mechanics (Poincar\'e recurrence theorem etc.)\\

In order to state our  probabilistic  result 
we first introduce the Gaussian probability measure $\mu_n$.  Given a probability space $(\Omega, \mathcal A, p)$ and $\{g_{j}(\omega)\}_{j\in \Z}$  a sequence of centered, normalized, independent identically distributed (i.i.d.)  complex Gaussian random variables, we introduce for any given $n\in \N$ the random vector
\begin{equation}\label{randomizedgen}
 \varphi(x, \omega)=\sum_{j\in \Z} 
\frac{g_j(\omega)}{\sqrt{2\pi( 1+j^{2n})}} e^{i jx}.
\end{equation}
 The law of this random vector is formally given by Gaussian measures $\mu_n$ supported on $H^s(\T)$ with $s<\frac{2n-1}{2}$
(see Section \ref{poBBGA} for more details about properties of $\mu_n$). We mention that Gaussian measures on infinite dimensional Hilbert spaces are classically constructed (see e.g. \cite{Bog}). \\

Along with the Gaussian measures $\mu_{n}$ we also introduce weighted Gaussian measures as follows. First of all for every $R>0$ we introduce the cut-off functions
$$\chi_R(\cdot)=\chi(\frac \cdot R), \quad \chi\in C^\infty_0(\R,\R) \hbox{ and } \chi(x)=1 \hbox{ for } |x|<1.$$ 
Next we define the  density
\begin{equation}\label{FnR}
F_{n, R}(u)=\prod_{l\in \{0,\dots, n-1\}} \chi_R\Big(E_{2l+1}(u)\Big) \exp  \Big(- R_{n}(u) \Big),
\end{equation}
with $R_n(u)$ introduced in \eqref{energybasic},
and also the associated weighted Gaussian measure
$$d\rho_{n, R}= F_{n, R}(u) d\mu_n.
$$
It is not difficult to check (see Section \ref{poBBGA}) that, thanks to the truncation by the the cut--off function $\chi_R$ in \eqref{FnR}, we get  
 $\sup_{u\in H^{n-1}} |F_{n,R}(u)|<\infty$ and hence, since $\mu_n(H^{n-1})=1$ we have that
$$F_{n, R}(u)\in L^\infty(\mu_n),$$ then the measure $\rho_{n,R}$ defined above is well defined. 
We recall that the idea to truncate the density by using previous conservation laws, as in \eqref{FnR}, goes back to the work of Zhidkov (see \cite{ZhLN}) who constructed a sequence of invariant measures at regularity above  the Gibbs measure, for cubic NLS and for real-valued KdV.
We can now state the main probabilistic results of this paper.
\begin{thm}\label{probBBB} Let $n\geq 2$ be an integer. For every $s\in (\frac 43, \frac{2n-1}2)$, there exists a Borel set $\Sigma^s \subset H^s(\T)$ such that:
\begin{enumerate}
\item[(i)]  $\mu_{n}(\Sigma^s)=1;$
\item[(ii)] for every $u_0\in \Sigma^s$ the unique local solution $u$ to \eqref{data}, provided by Theorem \ref{Main}, can be extended globally in time to a solution
$u\in {\mathcal C}(\R, H^s(\T))$; moreover there exists $C>0$ such that 
$$\|u(t,x)\|_{H^s(\T)}\leq C \sqrt {C+\ln (1+|t|)};$$
\item[(iii)] $\Sigma^s$ is invariant along the nonlinear flow $\Phi(t)$ associated with \eqref{data} established in  $(ii)$; moreover
for every $R>0$ the measure $\rho_{n, R}$ is invariant along the flow $\Phi(t)$ restricted on the invariant set $\Sigma^s$.
\end{enumerate}
\end{thm}

The proof of Theorem \ref{probBBB} will be given in full detail for $n=2$, in order to make more clear and transparent the argument. The general case $n>2$ will be sketched in Section \ref{RNU}.
Some comments about Theorem \ref{probBBB}  are in order: 
\begin{enumerate}
\item with item $(ii)$ we show that local solutions to the Cauchy problem 
\eqref{data} can be globalized almost surely for $s>\frac 43$,
despite the deterministic globalization result in Corollary \ref{COR} which holds for $s\geq 2$;
\item  by item $(ii)$ we obtain almost surely
a logarithmic control on the growth of fractional $H^s(\T)$ Sobolev norms, for $s>\frac 43$ (which improves on the aforementioned deterministic exponential upper bound 
of $H^s(\T)$ for $s\geq 2$ fractional); 
\item more significantly we construct weighted Gaussian invariant measures for the global flow, this is a basic tool to study dynamical systems from the statistical mechanics point of view. 
\end{enumerate}

We mention that Theorem \ref{probBBB} is in the spirit  of  Bourgain's results (see \cite{Bou2}, \cite{Bou3}), that mainly considered the Gibbs measures, namely the measures associated with the Hamiltonian. The main point in our result is that we work
with conservation laws which are above the Hamiltonian. Indeed, as mentioned at the beginning of this section, it is not even clear how to make sense of the Gibbs measure for the complex valued mKdV in the framework of the Gaussian measure.  In general working with measures associated with higher order conservation laws,  makes simpler the needed Cauchy theory.  On the other hand it makes more difficult the problem of the invariance since the higher order measures are no longer  invariant along the truncated flows (see \eqref{mKdVN}). 
We also note that we work with complex-valued mKdV, which makes the proof of the invariance of the measure much more involved compared to the real-valued case. In fact, 
in the real valued case the invariance of the measure $\rho_{n,R}$ can be proved following the same arguments as in Zhidkov (see \cite{ZhLN}). Concerning our proof,
one can see the difference between the real-valued case and the complex-valued case, by looking at the expression 
in \eqref{nahm}, which in the case of real-valued functions is identically equal to zero by an easy integration by parts argument.
However in the complex valued case it is the main term to be estimated. Our argument to prove the invariance of $\rho_{n,R}$ is inspired by a series of papers
devoted to the existence of a sequence of invariant measures for the Benjamin-Ono equation (see \cite{TV1}, \cite{TV2}, \cite{TV3} \cite{DTV}), whose ideas have been exploited in other papers (see \cite{GLV1}, \cite{GLV2}, \cite{CLO}, \cite{ChF} etc.). We also mention the paper \cite{ChF}, where the authors study, for the real-valued mKdV, invariant measures associated with conservation laws which are at level of regularity below the Gibbs measure. We finally quote the paper \cite{ONRS} devoted to the construction of invariant measures for the derivative NLS (DNLS), where  the same issue on the lack of invariance of the conservation law, along the associated finite dimensional approximating equations, appears. The technique used in \cite{ONRS}
to overcome this difficulty,
which is related to a strong deterministic control of the solutions, is different from the approach of this paper, which is based more on  probabilistic considerations.

\subsection{Organization of the paper}

Section \ref{sec-pf-main} is devoted to the analysis of the Cauchy problem \eqref{data} and its finite dimensional approximation \eqref{mKdVN}, in particular the convergence of solutions of \eqref{mKdVN} to solutions of \eqref{data} is established. In Section \ref{poBBGA} we provide basic facts about weighted and unweighted Gaussian measures. Also, the notion of pairing is introduced there. We shall see that it will be rather useful to study the momentum of multilinear Gaussian variables. In  Section \ref{almostinv} we show in detail how to overcome the lack of invariance of the energies $E_{2n+1}(u)$ along the flow \eqref{mKdVN} - we give a detailed proof in the case $n=2$.
In Section \ref{RNU}, we provide the main ideas on how to treat the general case $n>2$.
In Section \ref{PROBarg} we prove Theorem \ref{probBBB} for $n=2$ (indeed the argument of this section extends {\em mutatis mutandis}
to any generic $n\geq 2$). The arguments of Section \ref{PROBarg} are inspired by  Bourgain's original papers and subsequent papers by other authors (in particular \cite{Tz}), with the additional difficulty coming from the fact that we 
work with conservation laws which are above Hamiltonian, and hence are not preserved along \eqref{mKdVN}. This fact makes the relative construction delicate and more involved.

\subsection{Notation}
We introduce some notations useful in the sequel. 
\begin{itemize} 
\item  We shall use the following Fourier multipliers:
\begin{eqnarray*}
\widehat{Q^k(g)}(m)&=&\chi_{[2^{k-1},2^k)}(|m|)\hat g(m), \quad \forall k>1,\\
\widehat{Q^0(g)}(m)&=&\chi_{[0,1)}(|m|)\hat g(m),\\
\widehat{J^s g}(m)&=&(1+|m|^2)^{s/2}\hat g(m), \quad \forall s\in \R.
\end{eqnarray*}
\item We denote by $W(t)$ the propagator associated with the linear KdV
$$
\begin{cases}
\partial_t v+\partial_x^3 v=0\\
v(x,0)=u_0
\end{cases}
$$
namely
\begin{equation} \label{wp-W} 
W(t)u_0(x)=\sum_{k\in \Z}\widehat{u_0}(k)e^{i(kx+tk^3)}.
\end{equation} 

\item To shorten the notations we shall write sometimes $\partial_x=\partial$. 
\item We shall write $\int f=\int_\T f dx$.
\item We also write  $H^s=H^s(\T)$ and endow $H^s$ with the following norm $\|u\|_{H^s}=\|J^s u\|_{L^2}$.  
\item We denote $B_R^s$ the ball of radius $R>0$ in $H^s$. 
\item In  the paper the functions are always assumed to be $\C$-valued. 
\item Often, for $x\in \R$  we will use the notation $\langle x\rangle:=\sqrt{1+x^2}$. 
\item We denote by $\Phi(t)$ and $\Phi_N(t)$ the flows associated with \eqref{data} and \eqref{mKdVN}. 
\item For any topological space $X$ we define ${\mathcal B}(X)$ the Borelian subsets of $X$.
\end{itemize}

\section{Finite dimensional approximation and proof of Theorem \ref{Main}} \label{sec-pf-main} 
To prove  Theorem \ref{Main} we first introduce a frequency truncated version of the initial value problem \eqref{data}, we get estimates independent of the truncating parameter, and then we pass to the limit. We recall that the Cauchy problems \eqref{mKdVN} are crucial along the proof of Theorem \ref{probBBB}.
We consider the following finite dimensional approximation of \eqref{data}:
\begin{equation}\label{mKdVN}
\begin{cases}
\partial_t u + \partial_x^3 u - 6\Pi_N (|\Pi_N u|^2 \partial_x \Pi_N u)=0, \quad (t,x)\in \T\times \R, \quad u(t,x)\in \C, \quad N\in \N\\
u(x,0)=u_0
\end{cases}
\end{equation}
where $\Pi_N$ denotes the Dirichlet projection
\begin{equation}\label{pin}
\Pi_N \Big (\sum_{n\in \mathbb Z} u_n e^{inx}\Big)= \sum_{|n|\leq N}  u_n e^{inx}.
\end{equation}
Along with $\Pi_N$ we also introduce 
$$\Pi_{>N}=I-\Pi_N$$ where $I$ is the identity operator, namely
\begin{equation}\label{pi>n}
\Pi_{>N} \Big (\sum_{n\in \mathbb Z} u_n e^{inx}\Big)= \sum_{|n|>N}  u_n e^{inx}.
\end{equation}

In what follows we shall denote by $\Phi(t)$ the flow associated with \eqref{data} and by $\Phi_N(t)$ the flow associated with \eqref{mKdVN}.
It is worth mentioning that for any given $N\in \N$, the Cauchy problem \eqref{mKdVN} is globally well-posed in $H^s$ for any $s\geq 0$.
Indeed notice that the solution $u$ to \eqref{mKdVN} splits as
follows:
$$u=\Pi_N u+ \Pi_{>N} u,$$
where the second term $ \Pi_{>N} u$ evolves with the linear KdV flow (and hence $\| \Pi_{>N} u\|_{H^s}= \|\Pi_{>N} u_0\|_{H^s}$), while the first one $\Pi_N u$ 
evolves according with a system of $2N+1$ nonlinear ODEs, hence the solution
exists locally in time. Moreover it can be extended globally in time since one can check directly that $\|\Pi_N u(t)\|_{L^2}=\|\Pi_N u(t)\|_{L^2}$ for all times $t$.
We conclude the globalization argument since for every $s\geq 0$ and $N\in \N$, we have 
that $\|\Pi_N u(t)\|_{H^s}$ is equivalent to $\|\Pi_N u(t)\|_{L^2}$, with constants in the equivalence that blow-up when $N\rightarrow \infty$.

Next we show some uniform (with respect to  $N$) energy estimates, locally in time, for the flows $\Phi_N(t)$.
 
\subsection{Energy and dispersive estimate for the flow $\Phi_N(t)$}
The main result of this section is the following proposition. Recall that $B_S^s$ denotes the ball of radius $S$ in the space $H^s$. 
\begin{prop}\label{cauchytheory}
Let $\Phi_N(t)$ be the flow associated with \eqref{mKdVN}. Let $s>\frac 43$ be fixed, then there exist $c,\beta>0$ such that for every $S>0$ 
we have:
\begin{equation}\label{LRG}\sup_{\substack{u_0\in B_S^s, N\in \N\\|t|\leq c\langle S\rangle ^{-\beta}}}
\|\Phi_N(t)u_0\|_{H^s}\leq S+\frac 1S.\end{equation}
\end{prop}

\begin{remark}
The bound $S+\frac 1S$ on the  right hand side of \eqref{LRG} is inspired by \cite{btt}. In particular it is rather useful along the proof of 
items $(ii), (iii)$ of Theorem \ref{probBBB}. See in particular the proof of Lemma \ref{i+1}.
\end{remark}

We shall need some lemmas in order to prove Proposition \ref{cauchytheory}.

\begin{lemma} \label{lemma-energy} 
Let $s>\frac 12$ be fixed, then there exists a universal constant $C>0$ such that for every $u_0\in H^s$ and $N\in \N$, we have the following bound:
\begin{equation}\label{energy}
\left|\frac{d}{dt}\|J^s u_N(t)\|_{L^2}^2\right| \leq C\|J^s \Pi_N u_N(t)\|_{L^2}^3\|\Pi_N u_N (t)\|_{W^{1,\infty}},
\end{equation}
where $u_N(t)=\Phi_N(t) u_0$.
\end{lemma}
\begin{proof} Since $\Pi_{>N} u$ evolves with the linear KdV flow, we have 
$\frac d{dt} \|\Pi_{>N} (J^s u_N)\|_{L^2}^2=0$. Hence we shall focus on  the estimate of  $\frac d{dt} \|\Pi_N (J^s u_N)\|^2_{L^2}$.
For simplicity we shall write $v_N=\Pi_N u_N$.
We have, by using the properties $\Pi_N=\Pi_N^2$ and $\Pi_N=\Pi_N^*$ the following chain of identities:
\begin{multline*}\frac{d}{dt}\int J^s v_N \overline{J^s v_N} = 12 \Re \int \Pi_N J^s(|v_N|^2\partial v_N)\overline{J^s v_N}  dx
=12 \Re \int J^s(|v_N|^2\partial v_N)\overline{ J^s v_N}
\\= 12 \Re \int [J^s(|v_N|^2\partial v_N)-|v_N|^2\partial J^s v_N]\overline{J^s v_N} 
+12 \Re \int |v_N|^2\partial J^s u\overline{J^s v_N} 
\\= 12\Re \int [J^s(|v_N|^2\partial v_N)-|v_N|^2\partial J^s u]\overline{J^s v_N} 
-6\int \partial (|v_N|^2)|J^s v_N|^2 
\end{multline*}
where in the last step we have integrated by parts.
By using the Kato-Ponce estimate in \cite{IK2007},
\begin{equation}\label{KP}\|J^s(fg)-fJ^s g\|_{L^2}\leq \|J^sf\|_{L^2}\|g\|_{L^\infty}+
(\|f\|_{L^\infty}+\|\partial f\|_{L^\infty})\|J^{s-1}g\|_{L^2}\end{equation}
with $g=\partial v_N$ and $f=|v_N|^2$, in conjunction with the H\"older inequality we get
\begin{multline*}
\left|\frac{d}{dt} \|J^s v_N\|_{L^2}^2\right|
\leq C \|J^s(|v_N|^2)\|_{L^2}\|\partial v_N\|_{L^\infty} \|J^s v_N\|_{L^2}\\+ C (\||v_N|^2\|_{L^\infty}+\|\partial (|v_N|^2)\|_{L^\infty})\|J^{s-1}\partial v_N\|_{L^2}\|J^s v_N\|_{L^2}
+ C\|\partial (|v_N|^2)\|_{L^\infty}  \|J^s v_N\|_{L^2}^2.
\end{multline*}
We conclude by combining 
the Sobolev embedding $H^s\subset L^\infty$ for $s>\frac 12$ and the fact that $H^s$ is an algebra for $s>\frac 12$.
\end{proof}

The next lemma  is an adaptation of Lemma 1.7 in \cite{K2004}  to our context, see also \cite{KT03}.

\begin{lemma} \label{infinity} Let $s>4/3$ be fixed. Then there exists $C=C(s)>0$ such that 
for all  $T>0$ and  for all $u_0\in H^s$ we have:
\begin{equation}\label{est2}
 \int_0^T\|\Pi_N u_N(t)\|_{W^{1,\infty}} dt\leq C  \langle T\rangle T^{\frac{5}6} \Big (\sup_{t\in [0,T]}\|J^s \Pi_N u_N(t)\|_{L^2}\Big) 
 +C  \langle T\rangle T^{\frac{5}6} \|J^s \Pi_N u_N(t)\|_{L^3([0,T],L^2)}^3,
\end{equation}
where $u_N=\Phi_N(t) u_0$.
\end{lemma}
\begin{proof}
In order to prove this lemma we use the bound from
(8.37) in \cite{B1993}. For every $\varepsilon>0$ there exists $C_\varepsilon$ such that
\begin{equation}\label{bour}
\Big\|\sum_{|n|\leq N}\widehat{\phi}e^{i(kx+tk^3)}\Big\|_{L^6([0,T], L^6)}\leq C_\epsilon \langle T\rangle  N^\epsilon\|\phi\|_{L^2}, \quad \forall T>0,
\end{equation} 
which in turn implies
\begin{equation}\label{bour1}
\|W(t)Q^j \phi\|_{L^6([0,T], L^6)}\leq C_\epsilon \langle T \rangle 2^{j\epsilon}\|Q^j \phi\|_{L^2},
\end{equation} 
where $W(t)$ is the propagator associated with the linear KdV, and the operator $Q^j$ localizes at frequencies of order $2^j$.
Next we use as usual the notation $v_N=\Pi_N (u_N)$ and we immediately note that  by the Sobolev embedding $H^s\subset L^\infty$ for $s>\frac 12,$ we get the bound
$$ \int_0^T\|v_N(t)\|_{L^\infty} dt\leq C T \Big(\sup_{t\in [0,T]}\|J^s v_N(t)\|_{L^2}\Big),$$
hence we shall focus on  the estimate of $\int_0^T\|\partial v_N(t)\|_{L^\infty} dt$.
Now  for any interval $I\subset [0,T]$, we get by the Sobolev embedding $W^{\frac 16+\varepsilon}\subset L^\infty$, H\"older inequality in time and 
\eqref{bour1}, the following bound
\begin{multline}\label{NKVPS}
\|W(t)Q^j \phi\|_{L^1(I,L^\infty)}
\leq |I|^{5/6} \|W(t)Q^j \phi\|_{L^6 (I,L^\infty)}
\\\leq C_\varepsilon |I|^{5/6} 
2^{j(1/6+\epsilon)} \|W(t)Q^j\phi\|_{L^6(I,L^6)}
\leq C_\varepsilon |I|^{5/6} \langle T \rangle
2^{j(1/6+2 \epsilon)}\|Q^j \phi\|_{L^2}
\end{multline}
where $C_\varepsilon$ changes from line to line.
Now subdivide the interval $[0,T]$ into $2^j$ subintervals $I_l$ such that  $|I_l|=2^{-j}T$
and notice that for $w_N=\partial v_N$ we have the new 
 equation
$$\partial_t w_N+\partial_x^3 w_N=\partial_x f_N.$$
where $f_N= 6\pi_N(|v_N|^2\partial v_N)$.
Let $I_l=[a_l,a_{l+1}]$ and use the Duhamel principle to write, for $t\in I_l$,
$$w_N(t)=W(t-a_l)(w_N(a_l))+\int_{a_l}^tW(t-s)\partial f_N(s)ds$$
so that 
\begin{equation}\label{DuHaMeL}Q^j(w_N(t))=W(t-a_l)(Q^j(w_N(a_l))+\int_{a_l}^tW(t-s)\partial Q^j f_N(s)ds.\end{equation}
By using the Minkowski inequality, 
\eqref{DuHaMeL} and \eqref{NKVPS} we get
\begin{multline*}\|Q^j w_N(t)\|_{L^1([0,T], L^\infty)}\leq \sum_{l=1}^{2^j} \|Q^j w_N(t)\|_{L^1(I_l,L^\infty)}
\\\leq C_\varepsilon \langle T \rangle \sum_{l=1}^{2^j} 
 |I_l|^{5/6} 
2^{j(1/6+2\epsilon)}\sup_{t\in [0,T]}\|Q^j w_N(t)\|_{L^2}
\\+C_\varepsilon \langle T \rangle |I_l|^{5/6} 
2^{j(1/6+2\epsilon)} 2^j \int_{0}^{T}\|Q^j f_N(s)\|_{L^2}ds,
\end{multline*}
so that for any given $j$ we get
\begin{multline*}
\|Q^j w_N(t)\|_{L^1([0,T], L^\infty)}\leq C_\epsilon \langle T\rangle T^{\frac{5}6}2^{j(1/3 +2\varepsilon) } \Big(\sup_{t\in [0,T]}\|Q^jw_N(t)\|_{L^2}\Big)
\\+C_\epsilon  \langle T\rangle T^{\frac{5}6}2^{j(1/3+2\varepsilon) }\int_{0}^{T}\|Q^j f_N(s)\|_{L^2}ds.\end{multline*}
Adding in $j$ we now obtain 
\begin{equation}\label{est6}\|w_N(t)\|_{L^1([0,T], L^\infty)}\leq C_\epsilon  \langle T\rangle T^{\frac{5}6} \sup_{t\in [0,T]}\|J^{\frac 13+3\varepsilon} w_N(t)\|_{L^2}
+C_\epsilon  \langle T\rangle T^{\frac{5}6} \int_{0}^{T}\|J^{\frac 13 + 3 \varepsilon}f_N(s)\|_{L^2}ds.\end{equation}
We now recall that $w_N=\partial v_N$ and  $f_N=6\Pi_N(|v_N|^2\partial v_N)$, 
we deduce from \eqref{est6} the bound
\begin{multline}\label{est7}\|\partial v_N(t)\|_{L^1([0,T], L^\infty)}\leq C_\epsilon  \langle T\rangle T^{\frac{5}6}
\big( \sup_{t\in [0,T]}\|J^{\frac 43 + 3 \varepsilon} v_N(t)\|_{L^2}\big)\\
+C_\epsilon  \langle T\rangle T^{\frac{5}6} \int_{0}^{T}\|J^{\frac 13 + 3 \varepsilon}(|v_N|^2\partial v_N)(s)\|_{L^2}ds.\end{multline}
Now we use  the  following Kato-Ponce type estimate
from  \cite{IK2007}:
\begin{equation}\label{KPE}\|J^{\sigma}(hg)-hJ^{\sigma}g\|_{L^2}\leq C\|J^{\sigma}h\|_{L^2}\|g\|_{L^\infty},\end{equation}
where $0<\sigma<1$ and $C>0$  depends on $\sigma$. By using this estimate,
where we take $\sigma=\frac 13 +3\varepsilon$, $h=\partial v_N$ and $g=|v_N|^2$, we get
\begin{multline}\label{est8}\|J^{\frac 13 + 3 \varepsilon}(|v_N|^2\partial v_N)\|_{L^2}\leq C \| \partial v_N J^{\frac 13 + 3 \varepsilon}(|v_N|^2) \|_{L^2}+ C
 \|J^{\frac 13 + 3 \varepsilon} (\partial v_N) \|_{L^2}\| v_N\|^2_{L^\infty}\\
 \leq C \| \partial v_N\|_{L^2} \|J^{\frac 13 + 3 \varepsilon}(|v_N|^2) \|_{L^\infty}+ C
 \|J^{\frac 13 + 3 \varepsilon} (\partial v_N) \|_{L^2} \| v_N \|^2_{L^\infty}\\
 \leq C \|J^{\frac 43} v_N\|_{L^2} \|J^{\frac 43}(|v_N|^2) \|_{L^2}+ C
 \|J^{\frac 43 + 3 \varepsilon} v_N \|_{L^2} \|J^{\frac 43} v_N\|_{L^2}^2\leq C  \|J^{\frac 43 + 3 \varepsilon} v_N \|_{L^2}^3,
 \end{multline}
 where we used the Sobolev embedding $H^\sigma\subset L^\infty$ for $\sigma>\frac 12$ and the fact that $H^\sigma$ is an algebra for $\sigma>\frac 12$.
 We conclude by combining \eqref{est7} with \eqref{est8}.
\end{proof}

\begin{proof}[Proof of Prop. \ref{cauchytheory}] As usual we denote $u_N=\Phi_N(t)u_0$.
By combining \eqref{energy} and \eqref{est2} we get
\begin{multline}\label{energ+disp}
\sup_{t\in [0,T]} \|J^s u_N(t)\|_{L^2}^2 \leq \|J^s u(0)\|_{L^2}^2  +
C  \langle T\rangle T^{\frac{5}6} \Big (\sup_{t\in [0,T]}\|J^s \Pi_N u_N(t)\|_{L^2}^4\Big) 
\\+C  \langle T\rangle T^{\frac{5}6} \Big (\sup_{t\in [0,T]}\|J^s \Pi_N u_N(t)\|_{L^2}^3\Big) \|J^s \Pi_N u_N(t)\|_{L^3([0,T],L^2)}^3
\\
\leq  \|J^s u(0)\|_{L^2}^2  +  C\langle T\rangle T^{\frac{5}6}
 \Big (\sup_{t\in [0,T]}\|J^s \Pi_N u_N(t)\|_{L^2}^4\Big) 
+C  \langle T\rangle T^{\frac{11}6} \Big (\sup_{t\in [0,T]}\|J^s \Pi_N u_N(t)\|_{L^2}^6\Big).
\end{multline}
We split the proof of \eqref{LRG} in two steps. We shall show the existence of $S^*\in (0,1)$  such that:
\begin{equation}\label{216}\sup_{\substack{u_0\in B_S^s, N\in \N\\|t|\leq 1}}
\|\Phi_N(t)u_0\|_{H^s}\leq 2S, \quad \forall S<S^*;\end{equation}
\begin{equation}\label{LRGbig}
\exists c,\beta>0 \hbox{ s.t. }  \sup_{\substack{u_0\in B_S^s, N\in \N\\|t|\leq c\langle S\rangle ^{-\beta}}}
\|\Phi_N(t)u_0\|_{H^s}\leq S+\frac 1S, \quad \forall S\geq S^*.\end{equation}
Notice that if we choose $c$ small enough then we get the inclusion
$\{|t|<c\langle S\rangle ^{-\beta}\}\subset [-1,1]$. Moreover for any $S\in (0,1)$ we have $2S<S+\frac 1S$. By combining these facts 
we conclude from \eqref{216} that 
\begin{equation}\label{sanpancr}\sup_{\substack{u_0\in B_S^s, N\in \N\\|t|\leq c\langle S\rangle ^{-\beta}}}
\|\Phi_N(t)u_0\|_{H^s}\leq S+\frac 1S, \quad \forall S<S^*.\end{equation}
Hence we get \eqref{LRG} by combining \eqref{sanpancr} with \eqref{LRGbig}.
In order to establish \eqref{216} we first introduce for any $S>0$ the set
\begin{equation*}
{\mathcal I}_{S}=\Big \{T \hbox{ s.t. } C \langle T\rangle |T|^{\frac{5}6}  (2S)^4+ C  \langle T\rangle |T|^{\frac{11}6}
(2S)^6\leq S^2\Big \}.
\end{equation*}
We claim that
\begin{equation}\label{Claim}\sup_{t\in {\mathcal I}_S} \|J^s u_N(t)\|_{L^2}<2S, \quad \forall u_0 \hbox{ s.t. } \|J^s u_0\|_{L^2}<S.
\end{equation}
Assume by contradiction  that there exist $u_0\in H^s$ such that $\|J^s u_0\|_{L^2}<S$ and $t^*\in {\mathcal I}_{S}$ with the property 
\begin{equation}\label{trivNPbasic}\|J^s u_N(t^*)\|_{L^2}=2S=\sup_{t\in [0, t^*]} \|J^s u_N(t)\|_{L^2}\end{equation} 
(to fulfill the conditions in \eqref{trivNPbasic} it is sufficient to select $t^*$ to be the first time that 
the value $2S$ is achieved by $\|J^s u_N(t)\|_{L^2}$).
Then by using \eqref{energ+disp} (with $T=t^*$) in conjunction with the definition of ${\mathcal I}_{S}$ we get
$$\|J^s u_N(t^*)\|_{L^2}^2 \leq S^2 + S^2$$
which is in contradiction with \eqref{trivNPbasic}.
We conclude \eqref{216} from \eqref{Claim},
since it it is easy to check the following inclusion
$${\mathcal I}_S\supseteq [-1,1], \quad \forall S\in [0, S^*],$$
provided that we choose $S^*>0$ small enough.
\\
Next, in order to prove \eqref{LRGbig}  we 
introduce the following set:
\begin{equation*}
{\mathcal K}_S=\Big \{T \hbox{ s.t. } C \langle T\rangle |T|^{\frac{5}6}  (S+ S^{-1})^4\\+ C  \langle T\rangle |T|^{\frac{11}6}
(S+ S^{-1})^6\leq S^{-2}\Big \}.
\end{equation*}
We claim that
\begin{equation}\label{Claim2}\sup_{t\in {\mathcal K}_{S}} \|J^s u_N(t)\|_{L^2}<S+S^{-1},  \quad \forall u_0 \hbox{ s.t. } \|J^s u_0\|_{L^2}<S.
\end{equation}
Arguing as above we assume by contradiction that there exists $u_0\in H^s$ with $\|J^s u_0\|_{L^2}<S$ and $t^*\in {\mathcal K}_{S}$ such that 
\begin{equation}\label{trivNP}\|J^s u_N(t^*)\|_{L^2}=S+S^{-1}=\sup_{t\in [0, t^*]} \|J^s u_N(t)\|_{L^2},\end{equation} 
then 
by using \eqref{energ+disp} in conjunction with the definition of ${\mathcal K}_S$ we get
$$\|J^s u_N(t^*)\|_{L^2}^2 \leq S^2 + S^{-2}<(S+S^{-1})^2$$
which is in contradiction with \eqref{trivNP}.
We conclude  \eqref{LRGbig} from \eqref{Claim2} if we show  that
$$\exists c, \beta>0 \hbox{ s.t. } [0, c\langle S \rangle^{-\beta}] \subset {\mathcal K}_{S}, \hbox{ for } S\geq S^*$$
where $S^*$ is already given in \eqref{216}.
In fact notice that by elementary considerations we have the inclusion
\begin{equation*}
\Big \{|T|<1 \hbox{ s.t. } 2\sqrt 2 C |T|^\frac 56
(S+ S^{-1})^6\leq S^{-2}\Big\}\subset
{\mathcal K}_{S}.
\end{equation*}
Hence we get 
$$\Big \{T \hbox{ s.t. } |T|\leq \min\{1, (2\sqrt 2 C)^{-\frac 65} S^{-\frac {12}5}(S+ S^{-1})^{-\frac{36}5}\} \Big \}\subset {\mathcal K}_{S}$$
and it is sufficient to notice that for suitable $c, \beta>0$ we have
$$c \langle S \rangle^{-\beta}\leq \min\Big \{1, (2\sqrt 2 C)^{-\frac 65} S^{-\frac {12}5}(S+ S^{-1})^{-\frac{36}5} \Big \}, \quad \forall S\geq S^*.$$

\end{proof}
 \subsection{Proof of Theorem \ref{Main}}
 
 We now address the question of existence and uniqueness of a local strong solution to \eqref{data}.  For convenience of the reader we are making this section self contained and we start by the definition of strong solution. 
 \begin{mydef}\label{defsol} We say that $u\in {\mathcal C}([-T,T], H^s), \, s>4/3$ is a strong solution of \eqref{data}  if
 \begin{equation}\label{duha}
 u(t)=W(t)u_0+ 6\int_0^t W(t-t')|u|^2\partial u(t')dt'
 \end{equation}
 for each $t\in[-T,T]$, where the equality holds in the sense of distributions.
 \end{mydef}

 We now  state a key lemma.
 \begin{lemma}\label{N Cauchy} 
Let $\{u_{N,0}\}$ be a sequence in $H^s$, with $s>\frac 43$, such that 
\begin{equation}\label{roscar}u_{N,0}\overset{N\rightarrow \infty} \longrightarrow u_0 \hbox{ in } L^2 \hbox{ and }
\sup_N \|u_{N,0}\|_{H^s}=S<\infty.\end{equation}
Then the sequence $\{\Phi_N(t)u_{N,0}\}$, of solutions to \eqref{mKdVN} with initial condition $u_{N,0}$, is a Cauchy sequence in 
${\mathcal C}([-T_0,T_0],L^2)$, where $T_0= c \langle S \rangle^{-\beta}$ and $c, \beta>0$ are as in Proposition \ref{cauchytheory}.
\end{lemma} 

\begin{proof} 
We have the following splitting:
$$\Phi_N(t)u_{N,0}=\Pi_N \Phi_N(t)u_{N,0}+\Pi_{>N} W(t)u_{N,0},$$ where $W(t)$ is the propagator associated with the linear KdV.
Next assume $N\leq M$, then we get
\begin{multline}\label{piNM}\|\Pi_{>M} W(t)u_{M,0}-\Pi_{>N} W(t)u_{N,0}\|_{L^2}\leq
\|\Pi_{>M} W(t)u_{M,0}-\Pi_{>M} W(t)u_{N,0}\|_{L^2} \\+ 
\|\Pi_{>M} W(t)u_{N,0}-\Pi_{>N} W(t)u_{N,0}\|_{L^2} \leq \|u_{M,0}-u_{N,0}\|_{L^2} + 2
\|\Pi_{>N} u_{N,0}\|_{L^2}\\\leq \|u_{M,0}-u_{N,0}\|_{L^2} +
2 N^{-s}\|\Pi_{>N} u_{N,0}\|_{H^s} .\end{multline}
By combining this estimate with the assumptions \eqref{roscar}, we get
$\Pi_{>N} W(t)u_{N,0}$ is a Cauchy sequence in ${\mathcal C} ([-T_0,  T_0], L^2)$.
Hence it is sufficient to show that $v_N(t)=\Pi_N \Phi_N(t)u_{N,0}$ is a Cauchy sequence in  
${\mathcal C} ([-T_0,  T_0], L^2)$. Notice that $v_N(t)$ solves the Cauchy problem \eqref{mKdVN} with initial datum $\Pi_N u_{N,0}$. 
Then from the equation solved by $u_N-u_M$  we calculate:
\begin{multline}\label{I123} 
\frac{d}{dt} \|v_N - v_M \|_{L^2}^2 
= 12\Re \int \big[ \Pi_N \left( |v_N|^2 \partial v_N \right)  - \Pi_M \left( |v_M|^2 \partial  v_M \right)  \big] 
\overline{( v_N - v_M)} \\
=  12\Re \int \big[ \left( \Pi_N - \Pi_M \right)  \left( |v_N|^2 \partial v_N \right) (\overline{ v_N - v_M})\\
+ 12\Re \int \Pi_M  \left( |v_N|^2 \partial v_N - |v_M|^2 \partial v_M \right) (\Pi_N - \Pi_M)(\overline{v_N - v_M}) \\
+ 12\Re \int \Pi_M  \left( |v_N|^2 \partial v_N - |v_M|^2 \partial v_M \right) \Pi_M (\overline{v_N - v_M})\
=  I+ II +III.
\end{multline} 
 Due to the Sobolev embedding $H^s\subset L^\infty$, we have the following bound for the term $I$:
\begin{multline}\label{estimateI}
|I|\leq C \|v_N\|_{L^\infty}^2 \|\partial v_N\|_{L^2} 
\|\left( \Pi_N - \Pi_M \right)(v_N-v_M)\|_{L^2}
\\\leq C 
\|v_N\|_{H^s}^2 \|\partial v_N\|_{L^2} N^{-s}  (\|v_N\|_{H^s}+\|v_M\|_{H^s}).
\end{multline}
The term $II$ satisfies a similar estimate
\begin{equation}\label{estimateII}
|II|\leq C 
(\|v_N\|_{H^s}^2 \|\partial v_N\|_{L^2} +
\|v_M\|_{H^s}^2 \|\partial v_M\|_{L^2} )N^{-s}  (\|v_N\|_{H^s}+\|v_M\|_{H^s}).
\end{equation}
Concerning $III$ we have
\begin{multline*} 
III  = 12 \Re \int  \big( \left( |v_N|^2 - |v_M|^2 \right) \partial v_N   +  |v_M|^2 \partial (v_N - v_M) \big )\Pi_M (\overline{v_N - v_M})\\
= 12\Re \int  \big( \big(v_N(\overline{v_N - v_M}) + (v_N - v_M) \overline{v_M}  \big) \partial v_N \big ) \Pi_M (\overline{v_N - v_M})\\
+ 12\Re \int  \big( |v_M|^2 \partial (v_N - v_M) \big ) \Pi_M (\overline{v_N - v_M}).
\end{multline*}
and since we have $(v_N - v_M) =\Pi_M (v_N - v_M)$, we get by integration by parts on the second term on the right hand side the formula
\begin{multline*} 
III  = 12 \Re \int  \big( \big(v_N(\overline{v_N - v_M}) + (v_N - v_M) \overline{v_M}  \big) \partial v_N \big ) \Pi_M (\overline{v_N - v_M})
- 6 \int \partial (|v_M|^2) |\Pi_M (v_N - v_M)|^2.
\end{multline*} 
Hence we get 
\begin{multline}\label{treIII}
|III|\leq  C\|v_N - v_M\|_{L^2}^2 \left( \|v_N\|_{L^{\infty}} +  \|v_M\|_{L^{\infty}} \right) (\|\partial v_N\|_{L^{\infty}}
+\|\partial v_M\|_{L^{\infty}})\\\leq C\|v_N - v_M\|_{L^2}^2 \left( \|v_N\|_{H^{s}} +  \|v_M\|_{H^{s}} \right) (\|\partial v_N\|_{L^{\infty}}
+\|\partial v_M\|_{L^{\infty}}),\end{multline}where we used again the embedding $H^s\subset L^\infty$.
Notice that by using Proposition \ref{cauchytheory} we deduce
\begin{equation}\label{SpiuS-1}\sup_{\substack{N\in \N\\|t|<c\langle S \rangle^{-\beta}}} \|\Phi_N(t)u_{N,0}\|_{H^s}\leq S+S^{-1}.\end{equation}
Next we fix any time interval $[a,b]$  contained in $\{|t|\leq c\langle S\rangle^{-\beta}\}$, we integrate 
\eqref{I123} on $[a,b]$ and by using \eqref{estimateI}, \eqref{estimateII} and \eqref{treIII}, we get:
\begin{multline*}
\sup_{t\in [a,b]} \|v_N(t) - v_M(t) \|_{L^2}^2 \leq \|v_N(a)-v_M(a)\|_{L^2}^2+C(b-a) N^{-s} (S +S^{-1})^4\\
+C  (S+S^{-1}) \sup_{t\in [a,b]}  \|v_N(t)-v_M(t)\|_{L^2}^2 \int_a^{b} (\|\partial v_N(\tau)\|_{L^{\infty}} + \|\partial v_M(\tau)\|_{L^{\infty}})d\tau,\\
\end{multline*}
where we used \eqref{SpiuS-1}. Next by using Lemma \ref{infinity} and \eqref{SpiuS-1}, we get 
\begin{multline}\label{tragaci}
\sup_{t\in [a,b]} \|v_N(t) - v_M(t) \|_{L^2}^2 \leq \|v_N(a)-v_M(a)\|_{L^2}^2+C(b-a) N^{-s} (S +S^{-1})^4\\
C  \Big(\langle b-a\rangle |b-a|^{\frac{5}6} (S+S^{-1})
 +  \langle b-a\rangle |b-a|^{\frac{11}6} (S+S^{-1})^3\Big)\sup_{t\in [a,b]} \|v_N(t) - v_M(t) \|_{L^2}^2.
\end{multline}
We choose now $[a,b]=[0,c]$ with $c>0$ small in such a way that we can absorb in \eqref{tragaci} all the terms in the right hand side  (except the first  and second one) in the left hand side and we get
\begin{equation}\label{gigan}\frac 12 \sup_{t\in [0,c]} \|v_N(t) - v_M(t) \|_{L^2}^2 \leq \|v_N(0)-v_M(0)\|_{L^2}^2+C N^{-s} (S +S^{-1})^4.
\end{equation}
By using \eqref{roscar} we get that $v_N(0)=\Pi_N u_{N,0}$ is a Cauchy sequence in $L^2$, hence by using \eqref{gigan} we obtain
$v_N$ is a Cauchy sequence in ${\mathcal C}([0,c], L^2)$.
In particular $v_N(c)$ converges in $L^2$. Hence we can iterate the argument above on the interval $[c,2c]$
(recall that we used along our computations \eqref{SpiuS-1}, which is satisfied on the whole interval $[-T_0, T_0]$
and in particular on $[c, 2c]$). Hence we get $v_N$ is a Cauchy sequence in ${\mathcal C}([c,2c], L^2)$. The same argument can be repeated
$[\frac{T_0}c]+1$ times and we get that 
 $v_N$ is a Cauchy sequence in ${\mathcal C}([0,T_0], L^2)$. By a similar argument we get
 that  $v_N$ is a Cauchy sequence in ${\mathcal C}([-T_0,0], L^2)$ and we conclude.

\end{proof} 

We are now ready to prove the existence of a local strong solution to \eqref{data}.
\begin{prop}[Existence] \label{esistsss}
Let $s>\frac 43$ be given. Then for every $u_0\in H^s$ there exists at least one strong solution 
$u\in {\mathcal C}([-c\langle \|u_0\|_{H^s}\rangle^{-\beta}, c\langle \|u_0\|_{H^s}\rangle^{-\beta}],H^s)$ to \eqref{data},
where $c,\beta>0$ are the same constants of Proposition \ref{cauchytheory}.
\end{prop}

\begin{proof}
We consider the sequence $u_N$ of solutions to \eqref{mKdVN}, with initial condition $u_0$ independent of $N$.
Notice that by Proposition \ref{cauchytheory} we get 
\begin{equation}\label{hsbounded}\sup_{|t|\leq c\langle \|u_0\|_{H^s} \rangle^{-\beta}}  \|u_N(t)\|_{H^s}\leq \|u_0\|_{H^s}+\|u_0\|_{H^s}^{-1}
\end{equation}
and by Lemma \ref{N Cauchy}, if we set $T= c\langle \|u_0\|_{H^s}\rangle^{-\beta}$ we have that 
there exists $u\in {\mathcal C}([-T, T], L^2)$ such that
\begin{equation}\label{l2converg}
u_N\overset{N\rightarrow \infty} \longrightarrow u \hbox{ in } {\mathcal C}([-T, T],L^2).\end{equation}
By combining \eqref{hsbounded} and \eqref{l2converg} with an interpolation argument, we get
\begin{equation}\label{hsprimeconverg}
u_N\overset{N\rightarrow \infty} \longrightarrow u \hbox{ in } {\mathcal C}([-T, T],H^{s'}),
\quad \forall \, 0 \leq s'<s.\end{equation}
Now pick $4/3<s'<s$. The integral equation like \eqref{duha} for $u_N$ reads
\begin{equation}\label{duHAme}u_N(t)=W(t)u_0+6 \int_0^tW(t-t')\Pi_N(|\Pi_N u_N|^2\partial \Pi_N u_N(t'))dt',\end{equation}
and using \eqref{hsprimeconverg} it is easy to see that, since $s'>4/3$, we obtain 
$$\Pi_N(|\Pi_N u_N|^2\partial \Pi_N u_N\overset{N\rightarrow \infty} \longrightarrow  |u|^2\partial u 
\hbox{ in } {\mathcal C}([-T, T],L^2).$$
By passing to the limit in \eqref{duHAme} as $N\rightarrow \infty$, we get 
\eqref{duha} for each $t$, in $L^2$, and hence $u$ is a strong solution. 
Once we have this statement we can use the standard Bona-Smith argument in \cite{BS75} to conclude 
that $u\in {\mathcal C}([-T, T],H^s)$. But as promised at the beginning of the section we want to keep the presentation self contained, so we will prove directly the continuity in $H^s$ proving three claims.

\bigskip

{\bf Claim  1:} We have $u\in {\mathcal C}_w([-T, T],H^s)$, namely it $t_n \overset{n\rightarrow \infty} \longrightarrow  \bar t$ then 
$u(t_n) \overset{n\rightarrow \infty} \rightharpoonup u(\bar t)$ in $H^s$. 
\\

To prove this claim we first observe that, as a consequence of \eqref{hsprimeconverg} we have $u\in {\mathcal C}([-T, T],H^{s'})$, hence we have $u(t_n) \overset{n\rightarrow \infty} \longrightarrow u(\bar t)$ strongly in $H^{s'},\, \, 4/3<s'<s$. Moreover thanks to \eqref{hsbounded} we have that $\sup_n \|u(t_n)\|_{H^s}<\infty$. As a consequence there is a subsequence $t_{n_k}$ such that $u(t_{n_k}) \overset{k\rightarrow \infty} \rightharpoonup \bar u$ in $H^s$, and by uniqueness of the limit in the sense of distribution we have $u(\bar t)=\bar u$. Claim 1 is then proved.

\bigskip

{\bf Claim  2:} $u(t)$ is continuous at $t=0$ in the strong $H^s$ topology.
\\

To prove this claim we let $t_n\overset{n\rightarrow \infty} \longrightarrow 0$ and we want to show, after invoking Claim 1,  that
$$\|u(t_n)\|_{H^s}\overset{n\rightarrow \infty} \longrightarrow \|u(0)\|_{H^s}.$$
By Claim 1 and semicontinuity of the $H^s$ norm we obtain
$$\|u(0)\|_{H^s}\leq \liminf_{n\rightarrow \infty} \|u(t_n)\|_{H^s}.$$
For the other direction we use the energy estimate, from  Lemma \ref{lemma-energy}:
$$\left|\frac{d}{dt} \|u_N(t)\|_{H^s}^2\right|\leq C\|u_N(t)\|_{H^s}^3\|\Pi_N u_N(t)\|_{W^{1,\infty}},$$
and after integration in time
$$\Big| \|u_N(t_n)\|_{H^s}^2-\|u_N(0)\|_{H^s}^2\Big|\leq C  \|u_N(t)\|_{L^\infty([-T,T];H^s)}^3\int_0^{t_n}  
\|\Pi_N u_N(t)\|_{W^{1,\infty}}\,dt$$
so that from \eqref{hsbounded} 
$$\Big| \|u_N(t_n)\|_{H^s}^2-\|u_N(0)\|_{H^s}^2\Big |\leq C\int_0^{t_n}  \|\Pi_N u_N(t)\|_{W^{1,\infty}}\,dt.$$
Hence  we have 
\begin{equation*}
\|u_N(t_n)\|_{H^s}^2\leq \|u(0)\|_{H^s}^2+C\int_0^{t_n}  \|\Pi_N u_N(t)\|_{W^{1,\infty}}\,dt\\
\end{equation*}
and using Lemma \ref{infinity} and \eqref{hsbounded}
\begin{equation}\label{good1}
\|u_N(t_n)\|_{H^s}^2\leq \|u(0)\|_{H^s}^2+C|t_{n}|^\gamma,
\end{equation} 
for some $\gamma>0$. We now observe that since $u_N(t)$ converges strongly to $u(t)$ in $H^{s'}, \, s'<s,$ then $u_N(t)$ converges weakly  to $u(t)$ in $H^s$, and so, thanks to    to \eqref{good1}, we have 
\begin{eqnarray*}\|u(t_n)\|_{H^s}^2&\leq& \liminf_{N\rightarrow \infty} \|u_N(t_n)\|_{H^s}^2\\
&\leq &\liminf_{N\rightarrow \infty} (C|t_n|^\gamma +\|u(0)\|_{H^s}^2)= C|t_n|^\gamma +\|u(0)\|_{H^s}^2,
\end{eqnarray*}
and from here 
$$\limsup_{n\rightarrow \infty}\|u(t_n)\|_{H^s}^2\leq \|u(0)\|_{H^s}^2.$$

\bigskip

{\bf Claim  3:}  $u(t)$ is continuous in the $H^s$ topology for any $t\in [-T,T]$. \\

To show this claim one replace the initial data of the truncated initial value problem \eqref{mKdVN} with
$u_N(\bar t)=u(\bar t)$. We call $w_N$ the solution of this new Cauchy problem.  By repeating 
for  $w_N$ and $u$  at time $\bar t$, what we did for $u_N$ and $u$ at time zero, we obtain continuity for any $\bar t \in
[-T,T]$, and this proves the claim.

\end{proof}

Next we show the uniqueness.
 \begin{prop}[Uniqueness] \label{prop-uniq} Let $s>\frac 43$, $T>0$ and $u, \, v\in {\mathcal C}([-T,T], H^s)$ be two strong solutions of \eqref{data}.
 Then $u(t)=v(t)$ for every $t\in [-T, T]$.
 \end{prop}  
 \begin{proof}  
 Let $u$ be a strong solution of  the initial value problem \eqref{data}. Let $\theta$ be a smooth and compactly supported function on $\R$, then we define on $\T$ for $\varepsilon\in (0,1)$ the functions $\theta_\varepsilon$ such that 
 $\hat {\theta_\epsilon}(n)=\hat\theta(\epsilon n)$ and $\hat \theta(0)=1$. We abuse notation and call $\theta_\epsilon$ the 
 associated multiplier operator. Then if we define $u_\epsilon(t):=\theta_\epsilon u(t) $ we can write
 $$ u_\epsilon(t)=W(t)\theta_\epsilon u_0+6 \int_0^t W(t-t')\theta_\epsilon (|u|^2\partial u)(t')dt'$$
 in the pointwise sense.  Hence we have
 \begin{equation}\label{dataep}
\begin{cases} 
\partial_t u_\epsilon+\partial_x^3 u_\epsilon= 6\theta_\epsilon(|u|^2\partial_x u), \quad (t,x)\in \R\times \T, \quad u(t,x)\in \C\\
u_\epsilon(0)=\theta_\epsilon u_0,
\end{cases}
\end{equation}
 By repeating the proof of Lemma \ref{infinity} for the initial value problem \eqref{dataep}, we now have that 
 \begin{equation}\label{bound2}
\sup_{\varepsilon}  \|\partial u_\epsilon\|_{L^1([-T,T],L^\infty)}=C<\infty,
 \end{equation} 
where $C$ is independent of $\epsilon$, if $T$ is small enough.  
 We also note that $u_\epsilon$ converges to $u$ in $ {\mathcal C}([-T, T],H^{s})$ as $\epsilon\rightarrow 0$. We claim that 
\begin{equation}\label{bound3}
 \|\partial u\|_{L^1([-T,T],L^\infty)}\leq C,
 \end{equation} 
 as well. To prove this claim we fix  $p\geq 1$ and $g$ such that $\|g\|_{L^\infty([-T,T],L^{p'}_x)}\leq 1$, where $p'$ is the conjugate of $p$.  Then 
 \begin{eqnarray*}
 \left|\int_{-T}^T\int \partial u\bar g\,dx\,dt\right|&\leq & \int_{-T}^T\int |\partial u||g|\,dx\,dt\\
 &\leq&\liminf_{\epsilon\rightarrow 0}\int_{-T}^T\int |\partial u_\epsilon||g|\,dx\,dt\\
 &\leq&\int_{-T}^T\|\partial u_\epsilon\|_{L^p}\|g\|_{L^{p'}}\,dt\leq C\int_{-T}^T\|\partial u_\epsilon\|_{L^\infty}\leq C,
 \end{eqnarray*}
 where we have used Fatou's theorem and  \eqref{bound2}.
 As a consequence we deduce that for any $p\geq 1$ we have 
 \begin{equation}\label{bound4}
 \int_{-T}^T\left(\int|\partial u|^p\,dx\right)^{1/p} dt \leq C.
  \end{equation}
 We now recall that  for any  function $f$ we have 
 $$\|f\|_{L^\infty}=\lim_{p\rightarrow \infty}\left(\frac{1}{2\pi}\int|f|^p\,dx\right)^{1/p},$$
 and taking a sequence of indices $p_j$ such that $p_j\rightarrow \infty$ as $j\rightarrow \infty$ in \eqref{bound4} and using Fatou's theorem again, we finally have \eqref{bound3}.

 Now assume that $v$ is another strong solution of \eqref{data}. 
  We define in a similar manner $v_\epsilon$ and also for $v$ \eqref{bound3} holds. Next we compute 
$$\frac{d}{dt}\|u_\epsilon-v_\epsilon\|_{L^2}^2=12\Re\int\theta_\epsilon 
 (|u|^2\partial u-|v|^2\partial v)\overline{\theta_\epsilon (u-v)}dx,$$
 we integrate in $t$ to get 
 $$\|u_\epsilon(t)-v_\epsilon(t)\|_{L^2}^2=12\Re\int_0^t\int\theta_\epsilon 
 (|u|^2\partial u-|v|^2\partial v)\overline{\theta_\epsilon (u-v)}dx\,dt'. $$
 We now let $\epsilon\rightarrow 0$, to get
 \begin{equation}\label{limitep}
 \|u(t)-v(t)\|_{L^2}^2=12\Re\int_0^t\int 
 (|u|^2\partial u-|v|^2\partial v)\overline{ (u-v)}dx\,dt'. 
 \end{equation}
 We can write the right hand side of \eqref{limitep} as 
 \begin{multline}\label{uniqueness}12\Re\int_0^t\int 
 (|u|^2-|v|^2)\overline{ (u-v)}\partial u \,dx\,dt'
+12\Re\int_0^t\int |v|^2\partial (u-v)\overline{(u-v)}dx\,dt'\\
 =12\Re\int_0^t\int (|u|^2-|v|^2)\overline{(u-v)}\partial u \, dx\,dt'\
- 6 \int_0^t \int\partial (|v|^2) |u-v|^2dx\,dt'
 \end{multline}
 where we have used integration by parts.
  From the identity 
 $$|u|^2-|v|^2=u(\overline{ u-v})+\bar v(u-v)$$ 
 we get
  \begin{multline}\label{uniqueness}
   \|u(t)-v(t)\|_{L^2}^2= 12 \Re\int^t_0\int u(\overline{ u-v})^2 \partial u  \, dx\,dt'
  \\+12 \Re\int_0^t\int \bar v|u-v|^2\partial u \, dx\,dt'-6
\int_0^t\int\partial (|v|^2) |u-v|^2  \, dx\,dt'.
 \end{multline}
 Next we introduce $$M=\sup_{t\in [-T, T)]}\Big \{\|u(t)\|_{H^s}, \|v(t)\|_{H^s}\Big \},$$
 which is finite due to the fact that $u$ and $v$ are strong solutions. 
 From \eqref{uniqueness} we see that, for $0<\delta<T$,  we have 
 $$\sup_{0\leq t\leq \delta}\|u(t)-v(t)\|_{L^2}^2\leq CM\sup_{0\leq t\leq \delta}\|u(t)-v(t)\|_{L^2}^2 \int_0^\delta (\|\partial u\|_{L^\infty} + \|\partial v\|_{L^\infty})dt'.$$
 Since $\partial u,\, \partial v\in L^1([-T,T],\, L^\infty)$, there exists $\delta>0$ such that 
we can absorb the r.h.s. on the l.h.s. and we get 
\begin{equation}\label{giaq}
\sup_{0\leq t\leq \delta}\|u(t)-v(t)\|_{L^2}=0.
 \end{equation}
 Hence if we define
$$\bar t=\sup\{t\in [0,T] \hbox{ s.t. } u(t)=v(t)\}$$
then we have proved that $\bar t>0$. Moreover by repeating the argument above starting from time $\bar t$, we deduce that necessarily $\bar t=T$.

 \end{proof}
 
 \subsection{Long time convergence  of $\Phi_N(t)$ to $\Phi(t)$}

Next proposition will be rather useful in the sequel.

\begin{prop}\label{prop2}
Let $s> s'>\frac 43$ be given and $\{u_{k,0}\}$ be a sequence in $H^s$ 
such that
$$u_{k,0}\overset{k\rightarrow \infty} \longrightarrow u_0 \hbox{ in } H^s.$$
Assume moreover that for some $T>0$ we have
\begin{equation}\label{initialdata}\sup_{\substack{k\in \N\\ t\in [-T, T]}} \|\Phi_{N_k}(t) u_{k,0}\|_{H^s}<\infty,\end{equation}
for a suitable sequence $N_k \overset{k\rightarrow \infty} \longrightarrow \infty$ .
Then there exists 
$u\in {\mathcal C}([-T, T], H^s)$ 
such that
\begin{equation}\label{interpolSS}\sup_{t\in [-T, T]} \|\Phi_{N_k}(t) u_{k,0}-u(t)\|_{H^{s'}}
\overset{k\rightarrow \infty} \longrightarrow 0\end{equation}
and $u$ is the unique solution to \eqref{data}.
Moreover
\begin{equation}\label{convspri}\sup_{t\in [-T, T]} \|u(t)\|_{H^{s}}\leq \sup_{\substack{k\in \N\\ t\in [-T, T]}} \|\Phi_{N_k}(t) u_{k,0}\|_{H^s}.
\end{equation}
\end{prop}

\begin{proof} By adapting the proof of Lemma \ref{N Cauchy}  one can show that 
$\Phi_{N_k}(t) u_{k,0}$ has a limit $u$ in the space ${\mathcal C}([-T, T], L^2)$ (notice that we  work on a subsequence of frequencies $N_k$ and not on the full set of frequencies $N$ as in  Lemma \ref{N Cauchy},  but the proof of Lemma  \ref{N Cauchy} is unchanged in this framework). Notice also that Lemma \ref{N Cauchy} holds on a short time, but it can be iterated up to fixed time $T$, since the assumption \eqref{initialdata} allows for this iteration.
Moreover by looking at the proof of Proposition \ref{esistsss} one can show that the limit function $u$ belongs to ${\mathcal C}([-T, T], H^s)$
and moreover $u$
is the unique solution to the Cauchy problem \eqref{data} (notice that along the proof of Proposition \ref{esistsss} we assume that the initial condition
does not change along with the spectral parameter $N$, however this change of initial conditions is allowed in Lemma \ref{N Cauchy}, which is the key to prove
Proposition \ref{esistsss}).
The proof of \eqref{interpolSS} follows by interpolation, since we have convergence in $\mathcal C([-T,T],L^2)$ and  by assumption we have boundedness of  $\Phi_{N_k}(t) u_{k,0}$
in ${\mathcal C}([-T, T], H^s)$.
The proof of \eqref{convspri} follows from the uniform bound \eqref{initialdata}, which along with 
\eqref{interpolSS} implies $\Phi_{N_k}(t) u_{k,0}\overset{H^s} \rightharpoonup u(t)$ and hence we conclude by semicontinuity of the $H^s$ norm for weakly converging sequences.

\end{proof}

\section{Weighted Gaussian measures, pairing and multilinear gaussian estimates}\label{poBBGA}

In this section we first introduce the Gaussian measures, both unweighted and weighted (see subsections \ref{gaussprMN} and  \ref{sub-wGaussm}).  
Then in Subsection \ref{PAIR} we introduce the notion of pairing, which is a useful tool to estimate the second order momentum of linear combinations of multilinear products of Gaussian variables.\\

We start with some elementary considerations about the structure of the conservation laws $E_{2n+1}(u)$ introduced  in Section 1. We have
$$E_{2n+1}(u)=\|\partial^n u\|_{L^2}^2 +  R_{n}(u)$$
where $R_n(u)$ is a lower order correction which makes the energy $E_{2n+1}$ an exact conservation law for solutions to the complex-valued mKdV
(see Section \ref{RNU} for more details).
It is easy to check by induction that $R_n(u)$ is a linear combination of integrals of densities of the following type:
$$\prod_{i=1}^k \partial^{\alpha_i} v, \quad \max_{i=1,\cdots k} \alpha_i\leq n-1, \quad \sum_{i=1}^k \alpha_i\leq 2n-2, \quad \mbox{with } v\in \{u, \bar u\}.$$
We can assume that $\alpha_i\geq \alpha_{i+1}$ and hence necessarily $\alpha_3\leq n-2$. By the H\"older inequality and the Sobolev embedding 
$H^1\subset L^\infty$ we get
$$\Big|\int \prod_{i=1}^k \partial^{\alpha_i} v\Big|\leq \|u\|_{H^{n-1}}^2 \times \prod_{i=3}^k \|u\|_{W^{\alpha_i,\infty}}
\leq C  \|u\|_{H^{n-1}}^k, \quad \forall n\geq 2.$$
Hence we obtain  the  bounds
\begin{equation}
\label{upperboundM}
\left| R_n(u)\right|\leq p_n(\|u\|_{H^{n-1}})
\end{equation}
and
\begin{equation}\label{ENL2}\|u\|_{H^n}^2\leq |E_{2n+1}(u)| + p_n(\|u\|_{H^{n-1}}), \quad \forall n\geq 2
\end{equation}
where $p_n$ is a polynomial.

\subsection{Gaussian measures and basic properties}\label{gaussprMN}

For every $n\geq 2$ we recall that $\mu_n$, introduced in subsection \ref{prob} in the introduction,  denotes the Gaussian probability measure defined as the law of the random vector \eqref{randomizedgen}.
We have the following basic properties  of the probability measure $\mu_n$:
\begin{itemize}
\item[(1)] $\mu_n(H^s)=1, \quad \forall s<\frac{2n-1}2$;
\item[(2)] $\mu_n (H^{\frac{2n-1}2})=0$;
\item [(3)]$\forall s<\frac{2n-1}2 \quad \exists k, K>0 \hbox{ s.t. } \mu_n\big \{u\in H^s \hbox{ s.t. } \|u\|_{H^s}>M\big \}\leq Ke^{-kM^2}$.
\end{itemize}
\begin{remark}\label{rem1}
From the computation below it will be clear that the constants $K$ and $k$ in (3) can be chosen uniformly with respect to $s$ provided that $s<\bar s<  \frac{2n-1}2$ for some $\bar s$. 
\end{remark}

For completeness we sketch the proof of (1)-(3) above. Item  (1) follows from 
\begin{equation*}
\E\Big(\Big \|\sum_{j\in \Z} 
\frac{g_j(\omega)}{\sqrt{2\pi( 1+j^{2n})}} e^{i jx}\Big \|_{H^s}^2\Big)
\sim \E \Big(\sum_{j\in \Z}\langle j \rangle^{2s-2n} |g_j(\omega)|^2\Big)
\sim \sum_{j\in \Z}\langle j \rangle^{2s-2n}<+\infty\end{equation*}
where we used in the last step that
$s<n-\frac12.$
In order to prove item  (2) we get by independence and Gaussianity
\begin{multline*}
\E \Big(e^{-\Big \|\sum_{j\in \Z} 
\frac{g_j(\omega)}{\sqrt{2\pi( 1+j^{2n})}} e^{i jx}\Big \|_{H^{n-\frac 12}}^2}\Big)\sim \E \Big(e^{- \sum_{j\in \Z}\langle j \rangle^{-1} |g_j(\omega)|^2}\Big)\\
\sim \prod_{j\in \Z} \E \Big(e^{-\langle j \rangle^{-1} |g_j(\omega)|^2}\Big)
\sim \prod_{j\in \Z} \frac 1 \pi \int_\C e^{-\langle j \rangle^{-1} |z|^2} e^{-|z|^2} =\prod_{j\in \Z} \frac{\langle j \rangle}{1+\langle j \rangle}=0
\end{multline*}
and the conclusion follows.
Finally we prove item (3). We first bound the high order momenta:
\begin{multline*}
\E\Big(\Big \|\sum_{j\in \Z} 
\frac{g_j(\omega)}{\sqrt{2\pi( 1+j^{2n})}} e^{i jx}\Big \|_{H^s}^p\Big)
\sim \E \Big(\Big \|\sum_{j\in \Z}\langle j \rangle^{s-n} g_j(\omega) e^{ijx}\Big \|_{L^2_x}^p\Big)
\\\leq \Big\|  \E \Big (\Big|\underbrace{\sum_{j\in \Z}\langle j \rangle^{s-n} g_j(\omega) e^{ijx}}_{\mathcal N(0, \sum_{j\in \Z}\langle j \rangle^{2s-2n})} \Big|^p\Big )   \Big \|_{L^\frac 2p_x}
\leq C p^\frac p2, \end{multline*}
where we used the Minkowski inequality along with well known bounds on the momenta of complex valued Gaussians.
Then we conclude the Gaussian bound (3) by the following elementary fact (see for instance \cite{TV13}):
\begin{lemma}
Let $F:\Omega\rightarrow \R$ be a random variable on the probability space $(\Omega, \mathcal A, p)$
such that:
$$\E (|F(\omega)|^p)\leq C p^\frac p2$$
then we have
$$p\{\omega\in \Omega| |F(\omega)|>\lambda\}\leq \exp(-K\lambda^2)$$
where $K=K(C)>0$.

\end{lemma}

\subsection{Weighted Gaussian measures} \label{sub-wGaussm} 
Next we introduce, for every $N\in \N$ the functionals
$$F_{n,R,N}(u)=F_{n,R}\big(\Pi_N(u)\big)$$
where as usual $\Pi_N$ is the Dirichlet projector \eqref{pin} and $F_{n,R}$ is defined in \eqref{FnR}.
We also introduce the following family of measures:
$$d\rho_{n,R,N}= F_{n,R, N}(u) d\mu_n, \quad \forall R>0, \quad N\in \N.
$$
\\
Next we collect some relevant properties of $\rho_{n,R,N}$.
\begin{prop}\label{ProPMN}
Fix $n\geq 2$ integer and $R>0$. Then the following holds:
\begin{equation}\label{unif} \sup_{N\in \N, \hbox{ } u\in H^{n-1}}\Big |\chi_{\Omega_{n,R,N}}(u)\exp\Big(-R_{n}(\Pi_Nu)\Big)\Big|<\infty, \quad \sup_{u\in H^{n-1}}
\Big|\chi_{\Omega_{n,R}(u)}\exp\Big(-R_n(u)\Big)\Big |<\infty, \end{equation}
where
\begin{multline*}
\Omega_{n,R,N}=\Big \{u\in H^{n-1} \hbox{ s.t. }  \prod_{l\in \{0,\dots, n-1\}} \chi_R\Big(E_{2l+1}(\Pi_N u)\Big)>0\Big\},\\
\Omega_{n,R}=\Big \{u\in H^{n-1} \hbox{ s.t. }  \prod_{l\in \{0,\dots, n-1\}} \chi_R\Big(E_{2l+1}(u)\Big)>0\Big\}.
\end{multline*}
Moreover we have 
\begin{equation}\label{munlinf}\sup_N \|F_{n,R,N}\|_{L^\infty(\mu_n)}<\infty, \quad \|F_{n,R}\|_{L^\infty(\mu_n)}<\infty;\end{equation}
\begin{equation}\label{convrhoNR}
\sup_{\substack{A
\in {\mathcal B}(H^s)\\ s<\frac {2n-1}2}} \big|\rho_{n,R, N}(A)-\rho_{n, R}(A)\big| \overset{N\rightarrow \infty} \longrightarrow 0
\end{equation}
where ${\mathcal B}(H^s)$ are  Borel subsets in $H^s$;

\begin{equation}\label{gaussianbound} \forall s<\frac{2n-1}2 \quad \exists k,K \hbox{ s.t. }
\rho_{n,R,N} \Big\{u\in H^{s} \hbox{ s.t. } \|u\|_{H^s}>M\Big\}\leq Ke^{-kM^2}, \quad \forall N,M.
\end{equation}
\end{prop}

\begin{proof}
By elementary considerations based on \eqref{ENL2}, we have the following inclusion: 
$$\Omega_{n,R,N}\subset \Big \{u\in H^{n-1} \hbox{ s.t. }  \|\Pi_N(u)\|_{H^{n-1}}\leq C(R)\Big \},$$
for some $C(R)>0$.
Hence by \eqref{upperboundM} we deduce the uniform boundedness of $\chi_{\Omega_{n,R,N}}(u)\exp(-R_{n}(\Pi_Nu))$ on $H^{n-1}$. Same argument works for 
the uniform boundedness of $\chi_{\Omega_{n,R}}(u)\exp(-R_{n}(u))$, and hence \eqref{unif} follows. The proof of \eqref{munlinf} comes from \eqref{unif} along with 
the fact that $\mu_n(H^{n-1})=1$.\\
In order to prove \eqref{convrhoNR} notice that in \eqref{convrhoNR} we can consider the sup 
over $N\in \N$ and  $A\subset {\mathcal B}(H^{n-1})$.   In fact, since $\mu_n(H^{n-1})=1$,  for any $A\subset H^s$  we have that $\mu_n(A\cap H^{n-1})=\mu_n(A)$
and hence $\rho_{n,R, N}(A)=\rho_{n,R, N}(A\cap H^{n-1}), \quad \rho_{n, R}(A)= \rho_{n, R}(A\cap H^{n-1})$. Based on 
elementary considerations and basic properties on the Dirichlet projectors $\Pi_N$, we have the following pointwise convergence
$F_{n,N,R}(u)\overset{N\rightarrow \infty}\longrightarrow F_{n, R}(u), \quad \forall u\in H^{n-1}.$
Next we can apply the Egoroff theorem which implies that for every $\varepsilon>0$ there exists a subset ${\mathcal R}_{\varepsilon}\subset H^{n-1}$ 
such that $\mu_n({\mathcal R}_{\varepsilon})<\varepsilon$ and the convergence 
$F_{n,R, N}(u)\overset{N\rightarrow \infty}\longrightarrow F_{n, R}(u)$ is uniform on $H^{n-1}\setminus {\mathcal R}_{\varepsilon}$.
Hence we get
\begin{multline}\label{EGO}
\sup_{A\in {\mathcal B}(H^{n-1})} \Big |\rho_{n,R,N}(A)-\rho_{n,R}(A)\Big|\leq \sup_{A\in {\mathcal B}(H^{n-1})}  \Big|\rho_{n,R,N}(A\cap {\mathcal R}_{\varepsilon} )
-\rho_{n,R}(A\cap {\mathcal R}_{\varepsilon} )\Big|\\
+
\sup_{A\in{ \mathcal B}(H^{n-1})} \Big |\rho_{n,R,N}(A\setminus {\mathcal R}_{\varepsilon} ) -\rho_{n,R}(A\setminus {\mathcal R}_{\varepsilon})\Big|.
\end{multline}
The first term on the right hand side can be controlled, up to a constant related to the bound \eqref{unif}, by $\mu_n({\mathcal R}_{\varepsilon})<\varepsilon$.
The second term on the right hand side  can be estimated by the uniform convergence of $F_{n,R,N}$ to $F_{n,R}$ in $H^{n-1}\setminus {\mathcal R}_{\varepsilon}$. More precisely, there exists $N_\varepsilon$ large enough such that if 
 $N\geq N_\varepsilon$ the second term on the right hand side  of 
\eqref{EGO} is smaller than  $\varepsilon$.
The last property \eqref{gaussianbound} is a consequence of \eqref{munlinf} together  with 
Gaussian bounds for the measure $\mu_n$ recalled above, at the beginning of this subsection.
\end{proof}

We shall also need the following property.
\begin{prop}\label{luccidi}
Let $A_j\subset H^s$ Borel sets, with $s\in (0, \frac{2n-1}2)$, such that $\rho_{n,j}(H^s\setminus A_j)=0$ for every $j\in \N$. Then necessarily 
$\mu_n\big (H^s\setminus \bigcup_{j\in \N}A_{j} \big )=0$.
\end{prop}

\begin{proof}
Arguing as in Proposition \ref{ProPMN} we can assume $s=n-1$.
We introduce the following sets:
$$\Omega_j=\Big \{u\in H^{n-1} \hbox{ s.t. }  \prod_{l\in \{0,\dots, n-1\}} \chi_j\Big(E_{2l+1}(u)\Big)>0\Big\},$$ 
then we get
\begin{equation}\label{nullmeas}
\mu_n\Big ((H^{n-1} \setminus A_j )\cap\Omega_j\Big)=0.
\end{equation}
In fact by the assumption on $A_{j}$
we have
$$0=\int_{H^{n-1} \setminus A_{j}} F_{n,j} d\mu_n = \int_{(H^{n-1} \setminus A_{j})\cap \Omega_j} F_{n,j} d\mu_n $$
and hence, since $F_{n,j}(u)>0$ for every $u\in \Omega_j$, we conclude \eqref{nullmeas}.
Next we notice that 
\begin{equation}\label{omegan} 
\mu_n(\Omega_j)\overset{j\rightarrow \infty} \longrightarrow 1.
\end{equation}
In fact for every $u\in H^{n-1}$ we have $|E_{2l+1}(u)|<\infty$ for every  $l=1,\dots, n-1$, and hence
we have that $u\in \Omega_j$ for every $j\geq \bar j$, where $\bar j$ depends on $u$. Namely $\chi_{\Omega_j}(u)\overset{j\rightarrow \infty}\longrightarrow 1$ pointwise
on $H^{n-1}$ and hence, since $\mu_n$ is a probability measure, we conclude \eqref{omegan} by the dominated convergence theorem.
To conclude the proof notice that for every $\bar j\in \N$ we get
\begin{equation*}
\mu_n\Big((H^{n-1}\setminus \bigcup_{j\in \N} A_{j}) \bigcap \Omega_{\bar j}\Big)=\mu_n \Big (\bigcap_{j\in \N} (H^{n-1}\setminus A_{j}) \bigcap \Omega_{\bar j}\Big )
\leq \mu_n \Big((H^{n-1}\setminus A_{\bar j}) \bigcap \Omega_{\bar j}\Big )=0
\end{equation*}
where we have used \eqref{nullmeas}.
We complete the proof by passing to the limit in the estimate above as $\bar j\rightarrow \infty$ and by recalling \eqref{omegan}.

\end{proof}

\subsection{Pairing and multilinear products of Gaussians}\label{PAIR}
In  this subsection we follow \cite{DNY1} and \cite{DNY2}. 
Let us fix $n\in \N$ and ${\it i}_k\in \{\pm 1\}$ for $k=1,\dots, n$.
For any vector $(k_1,\dots, k_n)\in \Z^n$ and for every set $\mathcal I\subset \{1,\dots, n\}$ we denote
by $(k_1,\dots, k_n)_{\mathcal I}\in \Z^{n-\# {\mathcal I}}$ the vector obtained removing from $(k_1,\dots,k_n)$ the entries $k_j$ with $j\in \mathcal I$. Also we denote $\N_{\leq n}=\N\cap\{1,\dots, n\}$.
We give the following definition.
 \begin{mydef}\label{pairing}
Given $(k_1,\dots, k_n)\in \Z^n$ we say that we have: 
\begin{enumerate}
\item $0$-pairing when
$k_l\neq k_m$ for every $l,m\in \{1,\dots, n\}$ with $\it i_l\neq \it i_m$;
\item $1$-pairing, and we write  $(k_l, k_m)$,  when 
$k_l=k_m$ for some $l,m\in \{1,\dots, n\}$ with $\it i_l\neq \it i_m$ and
$k_m\neq k_h$ for every $h,m\in \{1,\dots, n\}\setminus \{l, m\}$ such that $i_h\neq i_m$.
\item $r$-pairing for $r\in \{1,\dots, [\frac n2]\}$ provided that there exist $X=(l_1,\dots,l_r)\in \N_{\leq n}^r, Y=(j_1, \dots, j_r)\in \N_{\leq n}^r$, with $\{l_1,\dots, l_r\}\cap\{j_1,\dots, j_r\}=\emptyset$
and $\#\{l_1, \dots, l_r\}=\#\{j_1, \dots, j_r\}=r$, such that
$k_{l_m}=k_{j_m}$ for every $m=1,\dots, r$ and ${\it i}_{j_m}\neq {\it i}_{l_m}$, and moreover $k_m\neq k_h$ 
for every $h,m\in \{1,\dots, n\} \setminus \{l_1,\dots, l_r, j_1, \dots, j_r\}$ 
such that ${\it i}_h\neq {\it i}_m$. In this case we  shall write that $(k_1,\dots, k_n)$ are $(X,Y)$-pairing.
\end{enumerate}
\end{mydef}
Next we consider linear combinations of multilinear Gaussians
$g_{\vec k}(\omega)=\prod_{j=1}^n g_{k_j}^{\it i_j}(\omega)$ where $g_l^1=g_l$ and $g_l^{-1}=\bar g_l$.
Let $a_{k_1,\dots, k_n}\in \C$ then we have the following bound, which is a special case of the more general large deviation results proved in \cite{DNY1} and \cite{DNY2}.
\begin{prop}\label{annal} We have the following bounds for a suitable constant $C>0$:
\begin{multline*}
\Big \|\sum_{\vec k\in \Z^n}  a_{\vec k} g_{\vec k}(\omega)  \Big \|_{L^2_\omega}^2 \leq C\sum_{\substack{r\in \{1,\dots, [\frac n2]\}\\ X=(l_1,\dots,l_r)\in \N_{\leq n}^r, Y=(j_1, \dots, j_r)\in \N_{\leq n}^r\\ \{l_1,\dots, l_r\}\cap\{j_1,\dots, j_r\}=\emptyset\\\#\{l_1, \dots, l_r\}=\#\{j_1, \dots, j_r\}=r
}}\sum_{\vec h\in \Z^{n-2r}} \Big (\sum_{\substack{\vec k\in \Z^n \\ \hbox{ s.t. } \vec k \hbox{ is } (X,Y)-\hbox{pairing}\\ 
\\\vec k_{\{l_1,\dots, l_r, j_1,\dots, j_r\}}=\vec h}} |a_{\vec k}|\Big )^2.
\end{multline*}
\end{prop}

\section{Almost invariance of $\rho_{2,N,R}$ along the flow $\Phi_N(t)$}\label{almostinv}

Due to the lack of invariance of the energy $E_{2j+1}(u)$, for $j\in\{1,\dots,n\}$ along the flow $\Phi_N(t)$, we have that the measures $\rho_{n, N,R}$
are not invariant along this flow. However the 
next proposition shows that those measures are almost invariant as long as $N\rightarrow \infty$
\begin{prop}\label{impo}
Let $n\geq 2$ be an integer and $R,T>0$ be given, then
\begin{equation}
\label{almostconv}
\sup_{\substack{t\in [0, T]\\ A\in \mathcal B(H^{\sigma}), \hbox{ } \sigma<\frac {2n-1}2} }\Big |\rho_{n,R,N} (\Phi_N(t) A)
- \rho_{n,R,N}(A)\Big|
\overset{N\rightarrow \infty} \longrightarrow 0
\end{equation}
where $\mathcal B(H^{\sigma})$ denotes the Borel sets in $H^\sigma$.
\end{prop}
We shall give the proof of Proposition \ref{impo} in the specific case $n=2$, in Section \ref{RNU} we give a sketch of the modifications to be done to deal with the general case $n>2$. We decided to present the details in the case $n=2$ to make the argument  transparent without dealing with heavy notation. The general argument for  $n>2$  follows the same ideas with only extra minor technical issues that need to be addressed. 

To address the case $n=2$ we shall work with the energies:
$$E_1(u)=\int |u|^2,\quad
E_3(u) = \int |\partial u|^2 + |u|^4,\quad 
E_5(u)= \int  |\partial^2 u|^2 + 6 |\partial  u|^2 |u|^2 + |\partial (|u|^2)|^2 + 2 |u|^6$$
and we deal with the measures
$$d \rho_{2,R,N}= \chi_R\Big(E_1(\Pi_N(u)\Big) \times \chi_R\Big(E_3(\Pi_N u)\Big) \times \exp \Big(-\int 6 |\partial  u|^2 |u|^2 + |\partial (|u|^2)|^2 + 2 |u|^6\Big)d\mu_2$$
Let us first introduce two functionals:
$$E_{j,N}^*: H^{1}\rightarrow \R, \quad j=3,5$$
defined as follows
\begin{equation}\label{EjN}E_{j,N}^*(u)=\frac d{dt} E_j\Big (\Pi_N \Phi_N(t)u\Big )_{t=0}, \quad j=3,5, \quad u\in H^{1}.\end{equation}
\begin{prop}\label{E1star}
For every $R>0$ 
we have:
$$\Big \|\chi_R\Big(E_1(\Pi_Nu)\Big)\times \chi_R' \Big(E_3(\Pi_N u)\Big) \times E_{3,N}^*(u)\Big\|_{L^2(\mu_2)}
\overset{N\rightarrow \infty} \longrightarrow 0.$$ 
\end{prop}
\begin{proof}
First of all let us compute a useful representation of the functional
$E_{3,N}^*(u)$.
Indeed we have from the expression
of $E_3(u)$ the following identity:
\begin{equation}\label{first*}
E_{3,N}^*(u)=2\Re \int \Big(\partial \partial_t  u_N(t)\partial \bar u_N(t)
+ 2 \partial_t u_N(t) \bar u_N(t) |u_N(t)|^2 \Big)_{t=0}\end{equation}
where we used the notation $u_N(t)=\Pi_N(\Phi_N(t)u)$.
Notice that $u_N(t)$ solves the following equation
$$\partial_t u_N +\partial_x^3 u_N =6 |u_N|^2 \partial_x u_N-6\Pi_{>N} (|u_N|^2 \partial_x u_N),$$
namely $u_N(t)$ is an exact solution for \eqref{mKdV0} with the extra the second term 
on the right hand side. This implies that, since $E_3(u)$ is exactly conserved along the flow of \eqref{mKdV0}, we can replace $\partial_t u_N$ in \eqref{first*}
with  the extra term $-6\Pi_{>N} (|u_N|^2 \partial_x u_N)$.
Then we get the following expression: 
\begin{multline}\label{first**}
E_{3,N}^*(u)=-12\Re \int \big(\partial (\Pi_{>N} (|u_N |^2 \partial u_N ) ) \partial (\bar u_N )+ 2\Pi_{>N} (|u_N |^2 \partial u_N ) \bar u_N  |u_N |^2)\big)
\\
=-24\Re \int  \Pi_{>N} (|u_N|^2 \partial u_N ) \bar u_N  |u_N |^2 
\end{multline}
where we used the following orthogonality argument in the last step
$$ \int \partial (\Pi_{>N} (|u_N |^2 \partial u_N ) ) \partial (\bar u_N )
=\int \Pi_{>N} f \Pi_{N}\bar g =0$$
where 
$f=\partial(|u_N |^2 \partial u_N),  \quad g=  \partial (u_N )$.
Recalling the definition of $u_N$ and using  \eqref{first**} we have 
$$E_{3,N}^*(u)=-24\Re \int  \Pi_{>N} (|\Pi_N u|^2 \partial \Pi_N u )  \Pi_N \bar u  |\Pi_N u |^2, $$
then the proposition will follow from proving that 
\begin{equation}\label{cuTTO}\Big \|\chi_R\Big(E_1(\Pi_N u)\Big)\times \chi_R' \Big(E_3(\Pi_N u)\Big)
\times \Big(\Re \int  \Pi_{>N} (|\Pi_Nu |^2 \partial \Pi_N u) \Pi_N \bar u |\Pi_N u|^2\Big)  \Big\|_{L^2(\mu_2)}
\overset{N\rightarrow \infty} \longrightarrow 0.\end{equation}
Thanks  to the cut-off  $\chi_R\Big(E_1(\Pi_N u)\Big)\times \chi_R' \Big(E_3(\Pi_N u)\Big)$ we can argue as in  the proof of Proposition \ref{ProPMN} and  obtain that  the function in \eqref{cuTTO} is supported in the set
$$\Big \{u\in H^1\hbox{ s.t. } \|\Pi_N u\|_{H^1}\leq C(R)\Big\}$$
for a suitable $C(R)>0$.
Hence it is enough to estimate the modulus of the function on its support, namely 
\begin{multline*}
\Big |\Re \int  \Pi_{>N} (|\Pi_Nu |^2 \partial \Pi_N u) \Pi_N \bar u |\Pi_N u|^2\Big |
\\\leq \|(|\Pi_Nu |^2 \partial \Pi_N u)\|_{L^2} \| \Pi_{>N} (\Pi_N \bar u |\Pi_N u|^2|)\|_{L^2}
\\\leq C \|\Pi_N u\|_{H^1} \|\Pi_N u\|_{L^\infty}^2 N^{-1} \| \Pi_N \bar u |\Pi_N u|^2|\|_{H^1}
\leq C N^{-1}\|\Pi_N u\|_{H^1}^6 \leq C(R) N^{-1},
\end{multline*}
where we used the embedding $H^1\subset L^\infty$ and the fact that $H^1$ is an algebra.

\end{proof}

We shall need the following proposition as well. 
\begin{prop}\label{E2star}
For every $R>0$
we have:
\begin{equation}\label{DIVO}\Big \|\chi_R\Big(E_1(\Pi_N u)\Big)\times \chi_R\Big (E_3(\Pi_N u)\Big) \times E_{5,N}^*(u) \Big\|_{L^2(\mu_2)}
\overset{N\rightarrow \infty} \longrightarrow 0.\end{equation}
\end{prop}
The proof of Proposition \ref{E2star} requires more work than
the proof of Proposition \ref{E1star} since the functional 
$E_{5,N}^*(u) $ involves derivatives of higher order compared to 
$E_{3,N}^*(u)$, and hence we cannot proceed by deterministic estimates as done at the end of the proof of Proposition \ref{E1star} above. Indeed we have to rely on a finer probabilistic argument and robust cancelations that allow to control the higher order derivatives
involved in the expression of $E_{5,N}^*(u) $. 

We postpone the proof of Proposition \ref{E2star}, and  we first show how  Proposition \ref{impo} follows from Proposition \ref{E1star} and Proposition \ref{E2star}.

\begin{proof}[Prop. \ref{E1star} and \ref{E2star} $\Rightarrow$ Prop. \ref{impo}]
We shall prove for any given $R>0$ the following 
\begin{equation}\label{tneq0}
\sup_{\substack{t\in \R\\A\subset {\mathcal B}(H^{1})}} \Big |\frac d{dt} \rho_{2, R, N}(\Phi_N(t)A)\Big |
\overset{N\rightarrow \infty} \longrightarrow 0\end{equation}
and hence \eqref{almostconv} follows by integration in  time. Recall that $\mu_2(H^1)=1$, hence once we establish \eqref{tneq0}, then
\eqref{almostconv} follows for every $A\in {\mathcal B}(H^\sigma)$ with $\sigma<\frac 32$.\\
Next we shall use the splitting 
$$d\mu_2= \gamma_N \exp \Big(-\|\Pi_N u\|_{L^2}^2-\|\partial^2 \Pi_N u\|_{L^2}^2\Big) d{\mathcal L}_N  \otimes d \mu_{2,N}^\perp$$
where ${\mathcal L}_N=\prod_{i=-N}^N du_i$ is the classical Lebesgue measure 
on $\C^{2N+1}$,
$\gamma_N$ is a renormalization constant such that $\gamma_N e^{-\|\Pi_N u\|_{L^2}^2-\|\partial_x^2 \Pi_N u\|_{L^2}^2} d{\mathcal L}_N$ is a probability measure on $\C^{2N+1}$,
$\mu_{2,N}^\perp$ is the pushforward measure
of the following high frequency projection of the vector \eqref{randomizedgen} for $n=2$:
\begin{equation}\label{randomizedtrunc}
\varphi(x, \omega)=\frac 1{\sqrt{2\pi}}\sum_{j\in \Z_{>N}} 
\frac{g_j(\omega)}{\sqrt{1+j^4}} e^{inx}
\end{equation}
where $\Z_{>N}=\Z\cap \{z||z|>N\}$.
Next we compute
\begin{multline*}
\gamma_N^{-1}\frac d{dt}  \Big(\rho_{2,R,N}(\Phi_N(t)A)\Big)_{t=0}\\=\frac d{dt} \Big (
\int_{\Phi_N(t)A} \chi_R\Big(E_1(\Pi_N u)\Big) \times \chi_R \Big(E_3(\Pi_N u)\Big)
 \exp(-E_5(\Pi_N u)) d{\mathcal L}_N\otimes d \mu_{2,N}^\perp\Big)_{t=0}.
 \end{multline*}
Since the finite dimensional flow $\Pi_N \Phi_N(t)$ is Hamiltonian, we can rely on the Liouville Theorem and by the change of variable formula,
in conjunction with the fact that\footnote{Recall that $E_1(u)=\int |u|^2$ which is easy to show to be conserved along $\Phi_N(t)$.}  $\frac d{dt} E_1(\Pi_N \Phi_N(t)(u))=0$, we can continue above with 
\begin{multline*}
= \frac d{dt} \Big (\int_A \chi_R(E_1\Big(\Pi_N u)\Big) \times \chi_R \Big(E_3(\Pi_N
(\Phi_N(t) u)\Big)
\exp \Big(-E_5(\Pi_N (\Phi_N(t) u))\Big) 
d{\mathcal L}_N\otimes d \mu_{2,N}^\perp\Big )_{t=0}
\\=\int_A \chi_R\Big(E_1(\Pi_N u)\Big) \times \chi_R' \Big (E_3(\Pi_N
u)\Big ) E_{3,N}^*(u) 
\exp \Big(-E_{5}(\Pi_N u)\Big) d{\mathcal L}_N\otimes d \mu_{2,N}^\perp
\\+\int_A \chi_R\Big(E_1(\Pi_N u)\Big) \times \chi_R \Big (E_3(\Pi_Nu)\Big ) E_{5,N}^*(u)
\exp \Big(-E_{5}(\Pi_N u)\Big) d{\mathcal L}_N \otimes d\mu_{2,N}^\perp
\end{multline*}
where $E_{j,N}^*$ are introduced in \eqref{EjN}.
Summarizing we get
\begin{multline}\label{finalLLL}\frac d{dt}  \Big (\rho_{2,R,N}(\Phi_N(t)A)\Big )_{t=0}
=\int_A \chi_R\Big (E_1(\Pi_N u)\Big) \times \chi_R' \Big(E_3(\Pi_N
u)\Big) E_{3,N}^*(u) 
\exp \Big(
-R_5(\Pi_N u) \Big) d\mu_2 
\\+\int_A \chi_R\Big(E_1(\Pi_N u)\Big) \times \chi_R \Big (E_3(\Pi_Nu)\Big) E_{5,N}^*(u) \exp \Big(
-R_5(\Pi_N u) \Big)
d\mu_2
\end{multline}
where we recall
$$R_5(u)=
\int (\partial |u|^2)^2 + 6 |u|^2|\partial u|^2 + 2 |u|^6.
$$
Concerning the first term in the right hand side of \eqref{finalLLL}, we notice that in fact 
the integral can be taken on the set
$$\Big\{ u\in H^1 \hbox{ s.t. } \chi_R \Big (E_1(\Pi_N u)\Big) \times \chi_R' \Big(E_3(\Pi_N
u)\Big)\neq 0\Big\}$$
which, following the proof of \eqref{unif}, is contained in 
$$\Big \{u\in H^{1} \hbox{ s.t. }  \|\Pi_N(u)\|_{H^{1}}\leq C(R)\Big \},$$
for some $C(R)>0$.
Then we can estimate the first integral in the  right hand side of  \eqref{finalLLL} as follows by using the Cauchy-Schwartz inequality in $L^2(\mu_2)$:
\begin{multline*}\sup_{A\in {\mathcal B}(H^1)} \Big|\int_A \chi_R\Big (E_1(\Pi_N u)\Big) \times \chi_R' \Big(E_3(\Pi_N
u)\Big) E_{3,N}^*(u) 
\exp \Big(
-R_5(\Pi_N u) \Big) d\mu_2\Big | \\\leq \sup_{A\in {\mathcal B}(H^1)} \Big(\sqrt{\mu_2(A)}\Big)  \Big \|\chi_R\Big (E_1(\Pi_N u)\Big) \times \chi_R' \Big(E_3(\Pi_N
u)\Big) E_{3,N}^*(u) \Big \|_{L^2(\mu_2)}\sup_{u\in B_{C(R)}H^1} \Big[\exp \Big(
-R_5(\Pi_N u) \Big)\Big]\overset{N\rightarrow \infty} \longrightarrow 0,\end{multline*}
where we have used Proposition \ref{E1star}, Proposition \ref{ProPMN} for the boundedness of $\exp (
-R_5(\Pi_N u))$ on bounded sets of $H^1$ and the property $\mu_2(A)\leq 1$.
Arguing as above, by using Proposition \ref{E2star} instead of Proposition \ref{E1star} we conclude that 
also for  the second term of the right hand side of  \eqref{finalLLL} we have
$$\sup_{A\in {\mathcal B}(H^1)}  \Big |\int_A \chi_R\Big(E_1(\Pi_N u)\Big) \times \chi_R \Big (E_3(\Pi_Nu)\Big) E_{5,N}^*(u) \exp \Big(
-R_5(\Pi_N u) \Big)
d\mu_2\Big |\overset{N\rightarrow \infty} \longrightarrow 0.$$
Summarizing we have obtained
\begin{equation}\label{t=0}
\sup_{A\subset {\mathcal B}(H^{1})} \Big |\frac d{dt} \Big( \rho_{2, R,N}(\Phi_N(t)A)\Big)_{t=0}\Big |
\overset{N\rightarrow \infty} \longrightarrow 0.\end{equation}
Next we show \eqref{tneq0} by computing the time derivative at any given time $t=t_0\neq 0$.
By using the fact that $\Phi_N(t)$ is a flow we get:
 $$\frac d{dt}  \Big (\rho_{2,R,N}(\Phi_N(t) A)\Big )_{t=t_0}=\frac d{dt} \Big (\int_{\Phi_N(t) (\Phi_N(t_0)A)} d\rho_{2,R,N}\Big)_{t=0}$$
and hence we are reduced to the case of derivative at time $t=0$ provided that we switch from the set $A$ to the new Borel set $\Phi_N(t_0)A$.
In any case since the estimate \eqref{t=0} is uniform with respect  to the Borel set $A$, we conclude.
\end{proof}

\subsection{Proof of Prop. \ref{E2star}} 

Using a computation similar  to that  done to prove  Proposition \ref{E1star} we obtain 
the following expression for $E_{5,N}^*(u)$:
\begin{multline*}
E_{5,N}^*(u)=-12\Re \int \partial^{2} (\Pi_{>N} (|\Pi_N u|^2 \partial \Pi_N u ))
\partial^2 (\Pi_N \bar u)\\
 -72 \Re \int \Pi_{>N} (|\Pi_N u|^2 \partial \Pi_N u ) \Pi_N \bar u |\partial \Pi_N u|^2
\\- 72 \Re \int |\Pi_N u|^2 \partial \Pi_N \bar u \partial \Pi_{>N} (|\Pi_N u|^2 \partial \Pi_N u ) \\
- 24\Re  \int \partial (|\Pi_N u|^2) \partial \Big (\Pi_{>N} (|\Pi_N u|^2 \partial \Pi_N u)
 \Pi_N \bar u\Big )
\\- 72 \Re \int |\Pi_N u|^4 \Pi_N \bar u \Pi_{>N} (|\Pi_N u|^2 \partial \Pi_N u ) .
\end{multline*}
By orthogonality (see the proof of Proposition \ref{E1star}) 
the first term on the right hand side  is zero
and by developing in the fourth term the derivative of a product we obtain
\begin{multline*}
E_{5,N}^*(u)=
 -72 \Re \int \Pi_{>N} (|\Pi_N u|^2 \partial \Pi_N u ) \Pi_N \bar u |\partial \Pi_N u|^2
\\- 72 \Re \int |\Pi_N u|^2 \partial \Pi_N \bar u \partial \Pi_{>N} (|\Pi_N u|^2 \partial \Pi_N u ) \\
- 24\Re  \int \partial (|\Pi_N u|^2)  \partial \Pi_N \bar u \Pi_{>N} (|\Pi_N u|^2 \partial \Pi_N u)
\\
- 24\Re  \int \partial (|\Pi_N u|^2)  \Pi_N \bar u \partial \Pi_{>N} (|\Pi_N u|^2 \partial \Pi_N u)
\\- 72 \Re \int |\Pi_N u|^4 \Pi_N \bar u \Pi_{>N} (|\Pi_N u|^2 \partial \Pi_N u ) .
\end{multline*}
By elementary manipulation on the fourth term we get
\begin{multline*}
E_{5,N}^*(u)=
 -72 \Re \int \Pi_{>N} (|\Pi_N u|^2 \partial \Pi_N u ) \Pi_N \bar u |\partial \Pi_N u|^2
\\- 72 \Re \int |\Pi_N u|^2 \partial \Pi_N \bar u \partial \Pi_{>N} (|\Pi_N u|^2 \partial \Pi_N u ) \\
- 24\Re  \int \partial (|\Pi_N u|^2)  \partial \Pi_N \bar u \Pi_{>N} (|\Pi_N u|^2 \partial \Pi_N u)
\\
- 24\Re  \int \partial (|\Pi_N u|^2  \Pi_N \bar u) \partial \Pi_{>N} (|\Pi_N u|^2 \partial \Pi_N u)
\\
+ 24\Re  \int |\Pi_N u|^2  \partial \Pi_N \bar u \partial \Pi_{>N} (|\Pi_N u|^2 \partial \Pi_N u)
\\- 72 \Re \int |\Pi_N u|^4 \Pi_N \bar u \Pi_{>N} (|\Pi_N u|^2 \partial \Pi_N u ).
\end{multline*}
Notice that
\begin{equation*}
2\Re  \int |\Pi_N u|^2 \partial \Pi_N \bar u \partial (\Pi_{>N} (|\Pi_N u|^2 \partial \Pi_N u )=\int\partial \big (|\Pi_{>N} (|\Pi_N u|^2 \partial \Pi_N u)|^2)=0
\end{equation*}
then the second and fifth terms are zero, hence we get
\begin{multline}\label{expE2*}
E_{5,N}^*(u)=
 -72 \Re \int \Pi_{>N} (|\Pi_N u|^2 \partial \Pi_N u ) \Pi_N \bar u |\partial \Pi_N u|^2
\\
- 24\Re  \int \partial (|\Pi_N u|^2)  \partial \Pi_N \bar u \Pi_{>N} (|\Pi_N u|^2 \partial \Pi_N u)
\\
- 24\Re  \int \partial (|\Pi_N u|^2  \Pi_N \bar u) \partial \Pi_{>N} (|\Pi_N u|^2 \partial \Pi_N u)
\\- 72 \Re \int |\Pi_N u|^4 \Pi_N \bar u \Pi_{>N} (|\Pi_N u|^2 \partial \Pi_N u ).
\end{multline}
By following the same argument
as in Proposition \ref{E1star}, one can check easily, since there is only one term with a derivative, that
\begin{equation*}\Big \|\chi_R\Big(E_1(\Pi_N u)\Big)\times \chi_R \Big(E_3(\Pi_N u)\Big)
\times \Re \int |\Pi_N u|^4 \Pi_{>N} (|\Pi_N u|^2 \partial \Pi_N u ) \Pi_N \bar u\Big\|_{L^2(\mu)}
\overset{N\rightarrow \infty} \longrightarrow 0.
\end{equation*}
Hence we have to deal with  the first, second and third term on the right hand side
of \eqref{expE2*}. We notice that the second term in \eqref{expE2*}, up to a multiplicative constant, can be written as follows
\begin{equation*}
\Re  \int  \Pi_{>N} (|\Pi_N u|^2  \partial \Pi_N u ) \partial \Pi_N \bar u \Pi_N u
\partial \Pi_N \bar u\\+  \Re  \int  \Pi_{>N} (|\Pi_N u|^2  \partial \Pi_N u ) \partial \Pi_N \bar u \partial \Pi_N u \Pi_N \bar u
\end{equation*}
and the second term is exactly (up to a multiplicative constant) the first term in \eqref{expE2*}.
Hence it is sufficient to prove  the following facts:
\begin{equation}\label{nahm}
\Big\|\Re  \int \partial (|\Pi_N u|^2  \Pi_N \bar u) \partial ((\Pi_{>N} (|\Pi_N u|^2 \partial \Pi_N u ))
) \Big \|_{L^2(\mu_2)} \overset{N\rightarrow \infty} \longrightarrow 0,
\end{equation}
\begin{equation}\label{nahm1}\Big \|\Re \int \Pi_{>N} (|\Pi_N u|^2 \partial \Pi_N u ) \Pi_N \bar u |\partial \Pi_N u|^2\Big\|_{L^2(\mu_2)} \overset{N\rightarrow \infty} \longrightarrow 0
\end{equation}
\begin{equation}\label{nahmo2}
\Big \|\Re  \int  \Pi_{>N} (|\Pi_N u|^2  \partial \Pi_N u ) (\partial \Pi_N \bar u)^2 \Pi_N u \Big\|_{L^2(\mu_2)} \overset{N\rightarrow \infty} \longrightarrow 0,
\end{equation}
which  will imply \eqref{DIVO}, since the cut--off functions $\chi_R$ are bounded.
We shall prove first the more complicated estimate \eqref{nahm}, the other ones \eqref{nahm1} and \eqref{nahmo2} are easier and their proof will be sketched at the end of the section.
We replace in \eqref{nahm} the function $u$ by the random vector \eqref{randomizedgen} for $n=2$ 
and we are reduced to showing 
\begin{equation}\label{lucccc}
\Big \| \Im  \sum_{\substack{j_1,j_2,j_3,j_4,j_5,j_6\in {\Z_{\leq N}}\\ 
j_1-j_2-j_3+ j_4-j_5+j_6=0\\
|j_4-j_5+j_6|>N}} \frac{(j_1-j_2-j_3)(j_4-j_5+j_6) j_6}{\langle j_1^2\rangle\langle j_2^2\rangle
\langle j_3^2\rangle\langle j_4^2\rangle\langle j_5^2\rangle\langle j_6^2\rangle}g_{\vec j}
(\omega)\Big \|_{L^2_\omega}\overset{N\rightarrow \infty} \longrightarrow 0 \end{equation}
where $$\Z_{\leq N}=\Z\cap [-N, N] \hbox{ and } g_{\vec j}(\omega)=
g_{j_1}(\omega)\bar g_{j_2}(\omega)\bar g_{j_3}(\omega)g_{j_4}(\omega) \bar g_{j_5}(\omega)
g_{j_6}(\omega).$$
It is important  to note that we pass from taking the real part in \eqref{nahm} to taking the imaginary part in \eqref{lucccc} because taking the Fourier transform of three derivatives brings down $i^3=-i$. The fact that we take the imaginary part in \eqref{lucccc} is crucial since, as we will see below,  it reveals cancelations of terms that otherwise would be infinities. 
We introduce some notations to simplify the presentation:
\begin{equation}\label{tildeINrh}{\mathcal I}_N=\Big \{\vec j= (j_1, j_2, j_3, j_4, j_5, j_6)\in \Z_{\leq N}^6 \hbox{ s.t. } {\mathcal L} (\vec j)=0, \quad |\mathcal P(\vec j)|>N
\Big \}.
\end{equation}
where 
\begin{equation*}{\mathcal L}(\vec j)= j_1-j_2-j_3+j_4-j_5+j_6, \quad
{\mathcal P}(\vec j)=j_1-j_2-j_3.\end{equation*}
Notice that ${\mathcal I}_N$ is the set of indices where we are considering the sum in \eqref{lucccc}.
Next we split ${\mathcal I}_N$ in several subsets that will help us is applying Proposition \ref{annal}. In order to do that we use
Definition \ref{pairing} for $n=6$ and $i_1=i_4=i_5=1$, $i_2=i_3=i_5=-1$ and we split the set of indices according with the case that we have $0$-pairing, $1$-pairing ,
$2$-pairing, $3$-pairing .
\begin{remark}
If we have $\vec j\in \mathcal I_N$ with $2$-pairing or $3$-pairing then 
its contribution in \eqref{lucccc} is zero.
In fact if we have $2$-pairing in $(j_1,j_2,j_3,j_4,j_5,j_6)$ then, 
by the condition  $j_1-j_2-j_3+j_4-j_5+j_6=0$, we have necessarily $3$-pairing and then the associated 
Gaussian $g_{\vec j}(\omega)$ is real valued. Since we consider the imaginary part in \eqref{lucccc}, we get a trivial contribution to the sum.
For this reason it's enough to restrict to the sum over $\vec j\in {\mathcal I}_N$ which have $0$-pairing or $1$-pairing.
\end{remark}
Next we introduce the following sets:
$${\mathcal I}^{0}_N=\Big \{\vec j\in {\mathcal I}_N \hbox{ s.t. } \, \,  \vec j \,\,  \hbox{$0$-pairing}\Big \},$$
and for any $k\in\{1,4,6\}, l\in \{2,3,5\}$
$${\mathcal I}^{j_k,j_l}_N=\Big \{\vec j\in {\mathcal I}_N \hbox{ s.t. } (j_k, j_l) \hbox{ $1$-pairing}\Big \}.$$
We also introduce the splitting $${\mathcal I}^{j_k,j_l}_N= \tilde{{\mathcal I}}^{j_k,j_l}_N \cup \hat{{\mathcal I}}^{j_k,j_l}_N$$
where 
$$\tilde{\mathcal I}_N^{j_k,j_l}=\{\vec j\in {\mathcal I}_N \hbox{ s.t. }  (k, l) \hbox{ is $1$-pairing for } \vec j \hbox{ and }  \#\{j_1,j_2, j_3, j_4,j_5, j_6\}=5\}, \quad
\hat {{\mathcal I}}_N^{j_k,j_l}={\mathcal I}_N^{j_k,j_l}\setminus \tilde{\mathcal I}_N^{j_k,j_l}$$
We also introduce for every $\vec j=(j_1,j_2,j_3,j_4,j_5,j_6)$ the real number
$$a(\vec j)=\frac{(j_1-j_2-j_3)(j_4-j_5+j_6) j_6}{\langle j_1^2\rangle\langle j_2^2\rangle\langle j_3^2\rangle\langle j_4^2\rangle\langle j_5^2\rangle\langle j_6^2\rangle}$$
which is the coefficient that appears in \eqref{lucccc} in front of $g_{\vec j}(\omega)$.\\
Next we split the cases of $1$-pairing in two sub-cases: either the pairing occurs between indices $j_l,j_m$ with $l,m$ in the same triplet $\{1, 2, 3\}$ or 
$\{4, 5, 6\}$, or $l$ and $m$ belong to two different triplet.  
All the pairings of the first type are treated in Lemma \ref{equaltriplet}, while all the pairing of the second type are treated in Lemma \ref{differenttriplet} except
the pairing ${\mathcal I}_N^{j_2,j_6}$ and ${\mathcal I}_N^{j_3,j_6}$. In fact one can easily check that if we roughly apply the proof of Lemma \ref{differenttriplet} to 
${\mathcal I}_N^{j_2,j_6}$ and ${\mathcal I}_N^{j_3,j_6}$ then we get a divergent quantity.
For this reason we have to split 
$${\mathcal I}^{j_2,j_6}_N= \tilde{{\mathcal I}}^{j_2,j_6}_N \cup \hat{{\mathcal I}}^{j_2,j_6}_N, \quad {\mathcal I}^{j_3,j_6}_N= \tilde{{\mathcal I}}^{j_3,j_6}_N \cup \hat{{\mathcal I}}^{j_3,j_6}_N$$
and we first show in Lemma \ref{ZNNN} that by a symmetry argument the contribution to \eqref{lucccc} given by 
$\tilde{{\mathcal I}}^{j_2,j_6}_N $, $\tilde{{\mathcal I}}^{j_3,j_6}_N $ is equal to zero. Then the contribution of $\hat{{\mathcal I}}^{j_2,j_6}_N$, $\hat{{\mathcal I}}^{j_3,j_6}_N$
is considered in Lemma \ref{ZN}.

\begin{lemma}\label{ZNNN} We have the following identity:
\begin{equation}\label{luccccpart1new}
\Im  \sum_{\vec j\in  \Z^6_{\leq N}} 
a(\vec j) {\bf 1}_{ \tilde{\mathcal I}^{j_2,j_6}_N \cup  \tilde{\mathcal I}^{j_3,j_6}_N}(\vec j)
g_{\vec j}
(\omega)=0.\end{equation}
\end{lemma}
\begin{proof}
Notice that $\tilde{\mathcal I}^{j_2,j_6}_N$ and $ \tilde{\mathcal I}^{j_3,j_6}_N$ are disjoint, then we treat them separately.
We show that the contribution to the sum given by vectors $\vec j$ belonging to  $\tilde{\mathcal I}^{j_2,j_6}$ is zero, by a symmetry argument. The same argument 
works for the disjoint set
$\tilde{\mathcal I}^{j_3,j_6}$.
Notice that both vectors 
\begin{equation}\label{markbas}(j_1, j, j_3, j_4, j_5,j), \quad (j_5, j, j_4, j_3, j_1, j),\end{equation} 
belong
to $\tilde {\mathcal I}^{j_2,j_6}_N$ for any $j, j_1, j_3, j_4, j_5\in \Z_{\leq N}$, provided that 
\begin{equation}\label{identL}
j_1-j-j_3+j_4-j_5+j=0, \quad |j_1-j-j_3|>N, \#\{j_1, j_3, j_4, j_5, j\}=5 .
\end{equation} 
Moreover one can show that the vectors in \eqref{markbas} are different, since we have by assumption that $\#\{j_1, j_3, j_4, j_5\}=4$ and hence under any permutation
different from the identity (as done in \eqref{markbas} for the first, third, fourth, fifth entries), we get  different vectors.
The contribution of the first vector in \eqref{markbas} to the sum is given by:
\begin{equation*}\Im  \frac{(j_1-j-j_3)(j_4-j_5+j) j}{\langle j^2\rangle^2 \langle j_1^2\rangle \langle j_3^2\rangle \langle j_4^2\rangle \langle j_5^2\rangle}|g_j|^2 
g_{j_1} \bar g_{j_3} g_{j_4}\bar g_{j_5}
=-\Im  \frac{(j_1-j-j_3)^2 j}{\langle j^2\rangle^2 \langle j_1^2\rangle \langle j_3^2\rangle \langle j_4^2\rangle \langle j_5^2\rangle}|g_j|^2
g_{j_1} \bar g_{j_3} g_{j_4}\bar g_{j_5}
\end{equation*} where we used the first identity in \eqref{identL}. 
For the second vector in \eqref{markbas} we get by a similar argument the contribution
\begin{equation*}
\Im   \frac{(j_5-j-j_4)(j_3-j_1+j) j}{{\langle j^2\rangle^2 \langle j_1^2\rangle \langle j_3^2\rangle \langle j_4^2\rangle \langle j_5^2\rangle}|} |g_j|^2
g_{j_5} \bar g_{j_4} g_{j_3}\bar g_{j_1} = - \Im \frac{(j_3-j_1+j)^2 j}{\langle j^2\rangle^2 \langle j_1^2\rangle \langle j_3^2\rangle \langle j_4^2\rangle \langle j_5^2\rangle}|
|g_j|^2 g_{j_5} \bar g_{j_4} g_{j_3}\bar g_{j_1}.\end{equation*}
Hence the sum of the contributions given by the two  vectors 
\eqref{markbas} 
is given by  
$$- \frac{(j_1-j-j_3)^2 j}{\langle j^2\rangle^2 \langle j_1^2\rangle \langle j_3^2\rangle \langle j_4^2\rangle \langle j_5^2\rangle}\Im (|g_j|^2 g_{j_1} \bar g_{j_3} g_{j_4}\bar g_{j_5}+
|g_j|^2 g_{j_5} \bar g_{j_4} g_{j_3}\bar g_{j_1})=0.$$
\end{proof}

\begin{lemma}\label{ZN} For ${\mathcal Z}_N=\hat{\mathcal I}^{j_2,j_6}_N,  \hat {\mathcal I}^{j_3,j_6}_N$ we have
\begin{equation}\label{luccccpart1new}
\Big \|  \sum_{\vec j\in  \Z^6_{\leq N}} 
a(\vec j) {\bf 1}_{{\mathcal Z}_N}(\vec j)
g_{\vec j}
(\omega)\Big\|_{L^2_\omega}\overset{N\rightarrow \infty}\longrightarrow 0.\end{equation}
\end{lemma}

\begin{proof}

Let's consider the case ${\mathcal Z}_N=\hat{\mathcal I}^{j_2,j_6}_N$, the case ${\mathcal Z}_N=\hat{\mathcal I}^{j_3,j_6}_N$ is similar. 
We have to consider vectors $\vec j=(j_1,j_2,j_3,j_4,j_5,j_6)$ such that
$j_2=j_6=j_l=j$ for some $l\in\{1,3,4,5\}$. 
Next notice that by the property $|{\mathcal P}(\vec j)|>N$ we get
$$\max\{|j|, |j_{l_1}|, |j_{l_2}|, |j_{l_3}|\}\geq \frac N3$$ where $\{l_1, l_2, l_3\}=\{1,2,3,4,5,6\}\setminus \{2,6,l\}$.
Moreover in this case we have the crude bound
$$|a(\vec j)|\leq \frac{C}{\langle j_{l_1}\rangle \langle j_{l_2}\rangle \langle j_{l_3}\rangle \langle j\rangle^3 }.$$
Hence by Proposition \ref{annal} we easily get
that  the desired estimate for ${\mathcal Z}_N=\hat {\mathcal I}_N^{j_2,j_6}$, is reduced to
\begin{equation*}
\sum_{\substack{j,j_{l_1},j_{l_2}, j_{l_3}\in \Z_{\leq N}\\
\max\{|j|,|j_{l_1}|, |j_{l_2}|, |j_{l_3}|\}\geq \frac N3}} 
\frac{ 1}{\langle j_{l_1}\rangle^2 \langle j_{l_2}\rangle^2 \langle j_{l_3}^2\rangle \langle j\rangle^6}
=O(N^{-1}).
\end{equation*}

\end{proof}

Next we deal with all the remaining cases, namely $0$-paring and $1$-pairing except the ones considered in Lemma \ref{ZNNN} and Lemma \ref{ZN}.

\begin{lemma}\label{IN6}
We have the following limit
\begin{equation}\label{luccccpart1}
\Big \|  \sum_{\vec j\in \Z^6}
a(\vec j) {\bf 1}_{{\mathcal I}_N^0}(\vec j) g_{\vec j}
(\omega)\Big \|_{L^2_\omega}\overset{N\rightarrow \infty} \longrightarrow 0.\end{equation}
\end{lemma}

\begin{proof} 
We use the following crude bound
$$|a(\vec j)|\leq \frac C{\langle j_1\rangle\langle j_2\rangle
\langle j_3\rangle\langle j_4\rangle\langle j_5\rangle}.$$
Hence by Proposition \ref{annal}
it is sufficient to prove:
\begin{equation}\label{sumsquares}
\sum_{\vec j\in {\mathcal I}_N^0} 
\frac{1}{\langle j_1\rangle^2\langle j_2\rangle^2
\langle j_3\rangle^2\langle j_4\rangle^2\langle j_5\rangle^2}
\overset{N\rightarrow \infty} \longrightarrow 0.
\end{equation}
From  the constraint $|\mathcal P(\vec j)|>N$ we have
$|j_1-j_2-j_3|>N$ and hence necessarily
$\max\{|j_1|, |j_2|, |j_3|\}\geq \frac N3$, and also by the constraint ${\mathcal L}(\vec j)=0$ we obtain that 
the sum is reduced to the five indices $j_1,j_2,j_3,j_4,j_5$, since $j_6=-j_1+j_2+j_3-j_4+j_5$. 
Hence we conclude from the following bound
$$\sum_{\substack{j_1, j_2,j_3,j_4,j_5\in \Z_{\leq N}\\\max\{|j_1|, |j_2|, |j_3|\}>\frac N3}} 
\frac{1}{\langle j_1\rangle^2\langle j_2\rangle^2
\langle j_3\rangle^2\langle j_4\rangle^2\langle j_5\rangle^2}
=O(N^{-1}).$$

\end{proof}

\begin{lemma}\label{differenttriplet}
We have the following limit for every ${\mathcal Z}_N={\mathcal I}^{j_1,j_5}_N, {\mathcal I}^{j_2,j_4}_N, {\mathcal I}^{j_3,j_4}_N$:
\begin{equation}\label{luccccpart1new}
\Big\| \sum_{\vec j\in \Z_{\leq N}^6} a(\vec j) {\bf 1}_{{\mathcal Z}_N}(\vec j) g_{\vec j}
(\omega)\Big\|_{L^2_\omega}\overset{N\rightarrow \infty} \longrightarrow 0.\end{equation}
\end{lemma}

\begin{proof} 
We focus on the case ${\mathcal Z}_N={\mathcal I}^{j_1,j_5}_N$ (the remaining cases are similar).
Hence we deal with vectors of the type
$$\vec j=(j,j_2,j_3,j_4,j,j_6)\in {\mathcal I}^{j_1,j_5}_N.$$
First of all we have the following crude bound 
$$a(\vec j)\leq \frac{|j-j_2-j_3||j_4-j+j_6| |j_6|}{\langle j^2\rangle^2\langle j_2^2\rangle\langle j_3^2\rangle\langle j_4^2\rangle\langle j_6^2\rangle}
\leq \frac{C}{\langle j\rangle^2\langle j_2\rangle
\langle j_3\rangle\langle j_4\rangle}.$$ 
Notice that the constraint $|\mathcal P(\vec j)|>N$ implies $\max\{|j|, |j_2|, |j_3|\}>\frac N3$. Moreover by the constraint 
$\mathcal L(\vec j)=0$ we have that in fact the sum depends on four indexes $j, j_2, j_3, j_4$, since $j_6=-j+j_2+j_3-j_4+j$.
Hence we can combine the information above with Proposition \ref{annal} and we conclude by the following bound:
\begin{multline*}
\sum_{\substack{j_2,j_3,j_4\in \Z_{\leq N}\\|j_2|>\frac N3}} \Big( \sum_{j\in \Z_{\leq N}}  \frac{1}{\langle j\rangle^2\langle j_2\rangle
\langle j_3\rangle\langle j_4\rangle} \Big)^2+C \sum_{\substack{j_2,j_3,j_4\in \Z_{\leq N}\\|j_3|>\frac N3}} \Big( \sum_{j\in \Z_{\leq N}}  \frac{1}{\langle j\rangle^2\langle j_2\rangle
\langle j_3\rangle\langle j_4\rangle} \Big)^2 \\
+ C \sum_{j_2,j_3,j_4\in \Z_{\leq N}} \Big( \sum_{\substack{j\in \Z_{\leq N}\\|j|>\frac N3}}  \frac{1}{\langle j\rangle^2\langle j_2\rangle
\langle j_3\rangle\langle j_4\rangle} \Big)^2=O(N^{-1}).
\end{multline*}

\end{proof}

\begin{lemma}\label{equaltriplet} We have the following limit for every ${\mathcal Z}_N={\mathcal I}^{j_1,j_2}_N,  {\mathcal I}^{j_1,j_3}_N, {\mathcal I}^{j_4,j_5}_N, {\mathcal I}^{j_5,j_6}_N$:
\begin{equation}\label{luccccpart1newnew}
\Big \| \sum_{\vec j\in \Z^6_{\leq N}} a(\vec j) {\bf 1}_{{\mathcal Z}_N}(\vec j) g_{\vec j}
(\omega)\Big \|_{L^2_\omega}\overset{N\rightarrow \infty}\longrightarrow 0.\end{equation}
\end{lemma}
\begin{proof}
We treat first the case ${\mathcal Z}_N={\mathcal I}^{j_5,j_6}_N$. 
Notice that since we are assuming $j_5=j_6$, then by the constraint $|\mathcal P(\vec j)|>N$
we get $|j_4|>N$. Moreover if we denote $j_5=j_6=j$, then we get the following bound 
for every vector $\vec j=(j_1,j_2,j_3,j_4,j,j)\in {\mathcal I}^{j_5,j_6}_N$:
$$|a(\vec j)|=\frac{|j_1-j_2-j_3||j_4| |j|}{\langle j_1^2\rangle\langle j_2^2\rangle
\langle j_3^2\rangle\langle j_4^2\rangle\langle j^2\rangle^2}\leq \frac{C}{\langle j_1\rangle \langle j_2\rangle
\langle j_3\rangle\langle j_4\rangle\langle j\rangle^3}$$
and hence by using Proposition \ref{pairing} the desired bound
 is reduced to the following estimate:
\begin{equation*}
 \sum_{\substack{j_1, j_2, j_3, j_4 \in \Z_{\leq N}\\ 
|j_4|>N}} \Big (\sum_{j\in \Z_{\leq N}} \frac{1}{\langle j_1\rangle\langle j_2\rangle
\langle j_3\rangle\langle j_4\rangle\langle j \rangle^3}\Big )^2
\leq C  \sum_{\substack{j_1, j_2, j_3, j_4 \in \Z_{\leq N}\\ 
|j_4|>N}} \frac{1}{\langle j_1\rangle^2\langle j_2\rangle^2
\langle j_3^2\rangle\langle j_4\rangle^2} =O(N^{-1}).
\end{equation*}
In the case 
${\mathcal Z}_N={\mathcal I}^{j_4,j_5}_N$ arguing as above, and by denoting $j_4=j_5=j$ we get
the following bound 
$$|a(\vec j)|\leq \frac{|j_1-j_2-j_3||j_6|^2}{\langle j_1\rangle^2\langle j_2\rangle^2
\langle j_3\rangle^2\langle j \rangle^4\langle j_6\rangle^2}\leq \frac C{\langle j_1\rangle\langle j_2\rangle
\langle j_3\rangle\langle j \rangle^4}.$$
We also notice that by the constraint $|j_1-j_2-j_3|>N$ we get $\max\{|j_1|,|j_2|, |j_3|\}>\frac N3$, and also the fact that the sum can be reduced to five indices
$j_1,j_2,j_3, j$, since $j_6=-j_1+j_2+j_3$.
Hence by using Proposition \ref{pairing} the desired bound
 is reduced to the following estimate:
 \begin{equation*}
 \sum_{\substack{j_1, j_2, j_3, \in \Z_{\leq N}\\ 
\max\{|j_1|,|j_2|, |j_3|\}>\frac N3}} \Big( \sum_{j\in \Z_{\leq N}}\frac{1}{\langle j_1\rangle\langle j_2\rangle
\langle j_3\rangle  \langle j \rangle^4}\Big)^2\leq C  \sum_{\substack{j_1, j_2, j_3\in \Z_{\leq N}\\ 
\max\{|j_1|,|j_2|, |j_3|\}>\frac N3}} \frac{1}{\langle j_1\rangle^2\langle j_2\rangle^2
\langle j_3\rangle^2}=O(N^{-1}).
\end{equation*}
We conclude with the case ${\mathcal Z}_N={\mathcal I}^{j_1,j_2}_N$ (the remaining case ${\mathcal Z}_N={\mathcal I}^{j_1,j_3}_N$ is similar). 
In this case we deal with vectors $\vec j=(j,j,j_3,j_4,j_5,j_6)\in {\mathcal I}^{j_1,j_2}_N$ and hence 
$$|a(\vec j)|\leq \frac{|j_3||j_4-j_5+j_6||j_6|}{\langle j\rangle^4
\langle j_3\rangle^2\langle j_4 \rangle^2 \langle j_5 \rangle^2\langle j_6\rangle^2}
\leq \frac{C}{\langle j\rangle^4
\langle j_3\rangle\langle j_4 \rangle \langle j_5 \rangle}.$$
By the constraints we get, arguing as in the previous case, that  $|j_3|>N$ and also the sum can be reduced to the four indexes 
$(j,j_3, j_4, j_5)$, then we conclude, by using Proposition \ref{pairing}, from the following bound 
\begin{equation*}
 \sum_{\substack{j_3, j_4, j_5\in \Z_{\leq N}\\ 
|j_3|>N}} \Big(\sum_{j\in \Z_{\leq N}}\frac{1}{\langle j\rangle^4
\langle j_3\rangle\langle j_4\rangle\langle j_5 \rangle}\Big)^2\leq C  \sum_{\substack{j_3, j_4, j_5\in \Z_{\leq N}\\ 
|j_3|>N}} \frac{1}{
\langle j_3\rangle^2\langle j_4\rangle^2\langle j_5 \rangle^2} =O(N^{-1}).
\end{equation*}

\end{proof}

At this point the proof of \eqref{nahm} is concluded. Next we prove \eqref{nahm1}. It is easy to see that a similar  proof holds 
for \eqref{nahmo2}, and hence we  shall skip it.
In order to prove \eqref{nahm1} we replace in the expression
\eqref{nahm1} the function $u$ by the random vector \eqref{randomizedgen} for $n=2$ and we are reduced to show:
\begin{equation}\label{.............}
\Big \| \Im  \sum_{\substack{(j_1,j_2,j_3,j_4,j_5,j_6)\in {\Z_{\leq N}^6}\\ 
j_1-j_2-j_3+ j_4-j_5+j_6=0\\
|j_4-j_5+j_6|>N}} \frac{ j_1 j_2 j_6}{\langle j_1^2\rangle\langle j_2^2\rangle
\langle j_3^2\rangle\langle j_4^2\rangle\langle j_5^2\rangle\langle j_6^2\rangle}g_{\vec j}
(\omega)\Big \|_{L^2_\omega}\overset{N\rightarrow \infty} \longrightarrow 0.\end{equation}
Next we denote $$b(\vec j)=\frac{ j_1 j_2 j_6}{\langle j_1^2\rangle\langle j_2^2\rangle
\langle j_3^2\rangle\langle j_4^2\rangle\langle j_5^2\rangle\langle j_6^2\rangle}$$
and hence
\begin{equation}\label{vecjPN}|b(\vec j)|\leq \frac{C}{\langle j_1\rangle\langle j_2\rangle
\langle j_3\rangle^2\langle j_4\rangle^2\langle j_5\rangle^2\langle j_6\rangle}.\end{equation}
Next we argue as along the proof of \eqref{lucccc}, in particular we have to consider only the $0$-pairing and $1$-pairing. Hence the proof follows from the following Lemmas in conjunction with Proposition \ref{annal}.
\begin{lemma}\label{IN60}
We have the following limit:
\begin{equation}\label{luccccpart1}
\Big \|  \sum_{\vec j\in \Z^6}
b(\vec j) {\bf 1}_{{\mathcal I}_N^0}(\vec j) g_{\vec j}
(\omega)\Big \|_{L^2_\omega}\overset{N\rightarrow \infty} \longrightarrow 0.\end{equation}
\end{lemma}
\begin{proof}
From the constraint $|\mathcal P(\vec j)|>N$ we have
$|j_1-j_2-j_3|>N$ and hence necessarily
$\max\{|j_1|, |j_2|, |j_3|\}\geq \frac N3$.
Hence we conclude from the following bound
$$\sum_{\substack{j_1, j_2,j_3,j_4,j_5\in \Z_{\leq N}\\\max\{|j_1|, |j_2|, |j_3|\}>\frac N3}} 
\frac{1}{\langle j_1\rangle^2\langle j_2\rangle^2
\langle j_3\rangle^4\langle j_4\rangle^4\langle j_5\rangle^4\langle j_6\rangle^2}
=O(N^{-1}).$$

\end{proof}

\begin{lemma}
We have the following limit for every ${\mathcal Z}_N={\mathcal I}^{j_1,j_2}_N, {\mathcal I}^{j_1,j_3}_N, {\mathcal I}^{j_4,j_5}_N, {\mathcal I}^{j_5,j_6}_N$:
\begin{equation}\label{luccccpart1new}
\big\|  \sum_{\vec j\in \Z_{\leq N}^6} a(\vec j) {\bf 1}_{{\mathcal Z}_N}(\vec j) g_{\vec j}
(\omega)\big\|_{L^2_\omega}\overset{N\rightarrow \infty} \longrightarrow 0.\end{equation}
\end{lemma}

\begin{proof}

Let's treat first the case $\mathcal Z_N={\mathcal I}^{j_1,j_2}_N$. In this case we have to deal with vectors
$$\vec j=(j,j,j_3,j_4,j_5,j_6)$$
and hence by \eqref{vecjPN} we get
$$|b(\vec j)|\leq \frac{C}{\langle j\rangle^2
\langle j_3\rangle^2\langle j_4\rangle^2\langle j_5\rangle^2\langle j_6\rangle}.$$
By using the constraint $|\mathcal P(\vec j)|>N$ and the fact that now $j_3=j_4-j_5+j_6$, we obtain
$|j_3|>N$ and hence,
by Proposition \ref{annal}, we conclude that 
$$\sum_{\substack{j_3,j_4,j_5,j_6\in \Z_{\leq N}\\ |j_3|>N}} \Big(\sum_{j\in \Z_{\leq N}} \frac{ 1}{\langle j\rangle^2
\langle j_3\rangle^2\langle j_4\rangle^2\langle j_5\rangle^2\langle j_6\rangle} \Big)^2=O(N^{-3}).$$
In the case ${\mathcal Z}_N={\mathcal I}^{j_1,j_3}_N$ we get $|j_2|>N$ and by recalling \eqref{vecjPN},
we are reduced to the estimate 
$$\sum_{\substack{j_2,j_4,j_5,j_6\in \Z_{\leq N}\\ |j_2|>N}} \Big(\sum_{j\in \Z_{\leq N}} \frac{ 1}{\langle j\rangle^3
\langle j_2\rangle \langle j_4\rangle^2\langle j_5\rangle^2\langle j_6\rangle} \Big)^2=O(N^{-1}).$$
In the case 
${\mathcal Z}_N={\mathcal I}^{j_4,j_5}_N$, we have vectors
$$\vec j=(j_1, j_2,j_3, j,j, j_6)$$
and hence we get  $|j_6|>N$ due to the condition 
$|{\mathcal P}(\vec j)|>N$. By \eqref{vecjPN} and Proposition \ref{annal} we conclude by 
$$\sum_{\substack{j_1,j_2,j_3,j_6\in \Z_{\leq N}\\ |j_6|>N}} \Big(\sum_{j\in \Z_{\leq N}} \frac{ 1}{\langle j\rangle^4
\langle j_1\rangle\langle j_2\rangle\langle j_3\rangle^2\langle j_6\rangle} \Big)^2=O(N^{-1}).$$
In the case 
${\mathcal Z}_N={\mathcal I}^{j_5,j_6}_N$, we have vectors
$$\vec j=(j_1, j_2,j_3, j_4,j, j)$$
and hence we get  $|j_4|>N$ due to the condition 
$|{\mathcal P}(\vec j)|>N$, hence by \eqref{vecjPN} and Proposition \ref{annal} we conclude by 
$$\sum_{\substack{j_1,j_2,j_3,j_4\in \Z_{\leq N}\\ |j_4|>N}} \Big(\sum_{j\in \Z_{\leq N}} \frac{ 1}{\langle j\rangle^3
\langle j_1\rangle\langle j_2\rangle\langle j_3\rangle^2\langle j_4\rangle^2} \Big)^2=O(N^{-3}).$$

\end{proof}

\begin{lemma}
We have the following limit for every ${\mathcal Z}_N={\mathcal I}^{j_1,j_5}_N, {\mathcal I}^{j_2,j_4}_N, {\mathcal I}^{j_3,j_4}_N, {\mathcal I}^{j_2,j_6}_N,
{\mathcal I}^{j_3,j_6}_N:$
\begin{equation}\label{luccccpart1new}
\Big\| \sum_{\vec j\in \Z_{\leq N}^6} a(\vec j) {\bf 1}_{{\mathcal Z}_N}(\vec j) g_{\vec j}
(\omega)\Big\|_{L^2_\omega}\overset{N\rightarrow \infty} \longrightarrow 0.\end{equation}
\end{lemma}
\begin{proof}
In the case 
${\mathcal Z}_N={\mathcal I}^{j_1,j_5}_N$, we have vectors
$$\vec j=(j, j_2,j_3, j_4,j, j_6)$$
and hence we get  $\max \{|j|, |j_2|, |j_3|\}>\frac N3$ due to the condition 
$|{\mathcal P}(\vec j)|>N$, hence by  \eqref{vecjPN} and Proposition \ref{annal}, we conclude by  the following estimate
\begin{multline*}
\sum_{\substack{j_2,j_3,j_4,j_6\in \Z_{\leq N}\\ |j_2|>\frac N 3}} \Big(\sum_{j\in \Z_{\leq N}} \frac{1}{\langle j\rangle^3
\langle j_2\rangle \langle j_3\rangle^2\langle j_4\rangle^2\langle j_6\rangle} \Big)^2
+\sum_{\substack{j_2,j_3,j_4,j_6\in \Z_{\leq N}\\ |j_3|>\frac N 3}} \Big(\sum_{j\in \Z_{\leq N}} \frac{1}{\langle j\rangle^3
\langle j_2\rangle \langle j_3\rangle^2\langle j_4\rangle^2\langle j_6\rangle} \Big)^2
\\+ \sum_{j_2,j_3,j_4,j_6\in \Z_{\leq N}} \Big(\sum_{\substack{j\in \Z_{\leq N} \\|j|>\frac N3}} \frac{1}{\langle j\rangle^3
\langle j_2\rangle \langle j_3\rangle^2\langle j_4\rangle^2\langle j_6\rangle} \Big)^2=O(N^{-1}).\end{multline*}
In the case 
${\mathcal Z}_N={\mathcal I}^{j_2,j_4}_N$, we have vectors
$$\vec j=(j_1, j,j_3, j,j_5, j_6)$$
and hence we get  $\max \{|j|, |j_1|, |j_3|\}>\frac N3$ due to the condition 
$|{\mathcal P}(\vec j)|>N$, hence by  \eqref{vecjPN} and Proposition \ref{annal}, we conclude by 
\begin{multline*}
\sum_{\substack{j_1,j_3, j_5,j_6\in \Z_{\leq N}\\ |j_1|>\frac N 3}} \Big(\sum_{j\in \Z_{\leq N}} \frac{1}{\langle j\rangle^3
\langle j_1\rangle \langle j_3\rangle^2\langle j_5\rangle^2\langle j_6\rangle} \Big)^2
+\sum_{\substack{j_1,j_3,j_5, j_6\in \Z_{\leq N}\\ |j_3|>\frac N 3}} \Big(\sum_{j\in \Z_{\leq N}} \frac{1}{\langle j\rangle^3
\langle j_1\rangle \langle j_3\rangle^2\langle j_5\rangle^2\langle j_6\rangle} \Big)^2
\\+ \sum_{j_1,j_3,j_5,j_6\in \Z_{\leq N}} \Big(\sum_{\substack{j\in \Z_{\leq N} \\|j|>\frac N3}} \frac{1}{\langle j\rangle^3
\langle j_1\rangle \langle j_3\rangle^2\langle j_5\rangle^2\langle j_6\rangle} \Big)^2=O(N^{-1}).\end{multline*}
In the case 
${\mathcal Z}_N={\mathcal I}^{j_3,j_4}_N$, we have vectors
$$\vec j=(j_1, j_2,j, j,j_5, j_6)$$
and hence we get  $\max \{|j|, |j_1|, |j_2|\}>\frac N3$ due to the condition 
$|{\mathcal P}(\vec j)|>N$, hence we conclude by  \eqref{vecjPN} and Proposition \ref{annal}, by the following
\begin{multline*}
\sum_{\substack{j_1,j_2, j_5,j_6\in \Z_{\leq N}\\ |j_1|>\frac N 3}} \Big(\sum_{j\in \Z_{\leq N}} \frac{1}{\langle j\rangle^4
\langle j_1\rangle \langle j_2\rangle\langle j_5\rangle^2\langle j_6\rangle} \Big)^2
+\sum_{\substack{j_1,j_2,j_5, j_6\in \Z_{\leq N}\\ |j_2|>\frac N 3}} \Big(\sum_{j\in \Z_{\leq N}} \frac{1}{\langle j\rangle^4
\langle j_1\rangle \langle j_2\rangle\langle j_5\rangle^2\langle j_6\rangle} \Big)^2
\\+ \sum_{j_1,j_2,j_5,j_6\in \Z_{\leq N}} \Big(\sum_{\substack{j\in \Z_{\leq N} \\|j|>\frac N3}} \frac{1}{\langle j\rangle^4
\langle j_1\rangle \langle j_2\rangle \langle j_5\rangle^2\langle j_6\rangle} \Big)^2=O(N^{-1}).\end{multline*}
In the case 
${\mathcal Z}_N={\mathcal I}^{j_2,j_6}_N$, we have vectors
$$\vec j=(j_1, j,j_3, j_4,j_5, j)$$
and hence we get  $\max \{|j|, |j_1|, |j_3|\}>\frac N3$ due to the condition 
$|{\mathcal P}(\vec j)|>N$. Hence we conclude by  \eqref{vecjPN} and Proposition \ref{annal} from the following computation
\begin{multline*}
\sum_{\substack{j_1,j_3, j_4,j_5\in \Z_{\leq N}\\ |j_1|>\frac N 3}}  \Big(\sum_{j\in \Z_{\leq N}}  \frac{1}{\langle j_1\rangle\langle j\rangle^2
\langle j_3\rangle^2\langle j_4\rangle^2\langle j_5\rangle^2}\Big)^2
+ \sum_{\substack{j_1,j_3, j_4,j_5\in \Z_{\leq N}\\ |j_3|>\frac N 3}}  \Big(\sum_{j\in \Z_{\leq N}}  \frac{1}{\langle j_1\rangle\langle j\rangle^2
\langle j_3\rangle^2\langle j_4\rangle^2\langle j_5\rangle^2}\Big)^2\\
\sum_{j_1,j_3, j_4,j_5\in \Z_{\leq N}}  \Big(\sum_{\substack{j \in \Z_{\leq N}\\|j|>\frac N3}}  \frac{1}{\langle j_1\rangle\langle j\rangle^2
\langle j_3\rangle^2\langle j_4\rangle^2\langle j_5\rangle^2}\Big)^2=O(N^{-1})
\end{multline*}
In the case 
${\mathcal Z}_N={\mathcal I}^{j_3,j_6}_N$, we have vectors
$$\vec j=(j_1, j_2,j, j_4,j_5, j)$$
and hence we get  $\max \{|j|, |j_1|, |j_2|\}>\frac N3$ due to the condition 
$|{\mathcal P}(\vec j)|>N$. 
Hence we conclude by  \eqref{vecjPN} and Proposition \ref{annal} from the following computation
\begin{multline*}
\sum_{\substack{j_1,j_2, j_4,j_5\in \Z_{\leq N}\\ |j_1|>\frac N 3}}  \Big(\sum_{j\in \Z_{\leq N}}  \frac{1}{\langle j_1\rangle \langle j_2\rangle\langle j\rangle^3
\langle j_4\rangle^2\langle j_5\rangle^2}\Big)^2
+ \sum_{\substack{j_1,j_2, j_4,j_5\in \Z_{\leq N}\\ |j_2|>\frac N 3}}  \Big(\sum_{j\in \Z_{\leq N}}  \frac{1}{\langle j_1\rangle \langle j_2\rangle\langle j\rangle^3
\langle j_4\rangle^2\langle j_5\rangle^2}\Big)^2\\
\sum_{j_1,j_2, j_4,j_5\in \Z_{\leq N}}  \Big(\sum_{\substack{j \in \Z_{\leq N}\\|j|>\frac N3}}  \frac{1}{\langle j_1\rangle \langle j_2\rangle\langle j\rangle^3
\langle j_4\rangle^2\langle j_5\rangle^2}\Big)^2=O(N^{-1}).
\end{multline*}

\end{proof}

\section{Proof of Theorem \ref{probBBB} for $n=2$}\label{PROBarg}

We give the proof of Theorem \ref{probBBB} in the case $n=2$, since Proposition \ref{impo} has been proved in detail for $n=2$. However it will be clear from the argument below, that if we use
Proposition \ref{impo} for a generic $n$, then the argument below extends to a generic $n>2$ {\em mutatis mutandis}.
\\
We point out that, with respect to the original Bourgain's argument, we cannot rely on the exact conservation of the measure along the truncated flow, and we can only use as a partial substitute \eqref{almostconv}. This makes the globalization argument and the construction of the invariant set in Theorem \ref{probBBB} more subtle. 

For every $i, j\in \N$ and $D>0$ and $s>0$ we define 
$$B^{i,j}_{s,D}=\Big \{u\in H^s \hbox{ s.t. } \|u\|_{H^s}\leq D \sqrt{i+j}\Big \},$$
where the constant  $D$ will be fixed later depending only on $s$, 
and 
\begin{equation}\label{defsigma}
\Sigma^{i,j}_{N,s, D}=\bigcap_{h\in \Z\cap [-2^j\langle D\sqrt{i+j}\rangle^\beta c^{-1}, 2^j\langle D\sqrt{i+j}\rangle^\beta c^{-1}]} \Phi_N(h c\langle D\sqrt{i+j}\rangle^{-\beta})(B^{i,j}_{s,D}).\end{equation}
where $\Phi_N(t)$ is the flow associated with \eqref{mKdVN}, and $c, \beta$ are the constant that appear in Proposition \ref{cauchytheory}.
\begin{remark}The elements in $\Sigma^{i,j}_{N,s , D}$ are exactly the initial data 
in $B^{i,j}_{s, D}$ which are mapped in $B^{i,j}_{s, D}$ by $\Phi_N(t)$  
along times equidistributed at distance 
$c \langle D\sqrt{i+j}\rangle^{-\beta}$ in the time interval $[-2^j, 2^j]$.
\end{remark} 
As a byproduct of Proposition \ref{cauchytheory} we obtain  the following useful lemma.
\begin{lemma}\label{i+1}
Let $s>4/3$. There exists $D_0>0$ such that for every $D>D_0$
we have the bound:
\begin{equation}\label{bounddescrite}
\sup_{|\tau|\leq 2^j} \|\Phi_N(\tau) u\|_{H^s}\leq D \sqrt{i+1+j}, \quad \forall u\in \Sigma^{i,j}_{N,s, D}, \quad \forall N\in \N.\end{equation}
\end{lemma}
\begin{remark}
Notice that the bound \eqref{bounddescrite} holds by definition in a stronger form, namely with $D\sqrt{i+j}$ on the right hand side and for every $D>0$, if we restrict the $\sup$ on the left hand side 
to times $h c\langle D\sqrt{i+j}\rangle^{-\beta}$ with $h\in \Z\cap 
[-2^j\langle D\sqrt{i+j}\rangle^\beta c^{-1}, 2^j\langle D\sqrt{i+j}\rangle^\beta c^{-1}]$. The main point in \eqref{bounddescrite} is that the $\sup$ is taken on the full interval
$[-2^j, 2^j]$,
with a small loss on the right hand side namely  $D \sqrt{i+1+j}$, provided that $D$ is large enough.
\end{remark}
\begin{proof}
By splitting the interval $[-2^j, 2^j]$ in subintervals of length $c\langle D\sqrt{i+j}\rangle^{-\beta}$, it is sufficient to prove that 
\begin{multline}\label{giancant}\sup_{\tau\in [h c\langle D\sqrt{i+j}\rangle^{-\beta}, (h+1) c\langle D\sqrt{i+j}\rangle^{-\beta}] } \|\Phi_N(\tau)u\|_{H^s}\leq 
D \sqrt{i+1+j},\\
\forall h\in \Z\cap 
[-2^j\langle D\sqrt{i+j}\rangle^\beta c^{-1}, 2^j\langle D\sqrt{i+j}\rangle^\beta c^{-1}], \quad 
\forall u\in \Sigma^{i,j}_{N,s, D},
\end{multline}
provided that $D$ is large enough.
We shall choose $D$ in such a way that we have the following estimate:
\begin{equation}\label{elemNT}
D \sqrt{i+j} + (D \sqrt{i+j})^{-1}\leq D \sqrt{i+1+j}, \quad \forall i, j \in \N.\end{equation}
It is easy to check that this choice of $D$ is possible, provided that $D>D_0>0$ for a suitable $D_0$.
Now we claim that 
$$\sup_{\tau\in [h c\langle D\sqrt{i+j}\rangle^{-\beta}, (h+1) c\langle D\sqrt{i+j}\rangle^{-\beta}] } \|\Phi_N(\tau)u\|_{H^s}
\leq D \sqrt{i+j} + (D \sqrt{i+j})^{-1}, \quad \forall u\in \Sigma^{i,j}_{N,s, D}.
$$
By translating the time and using the group property of $\Phi_N$ we have that 
$$\sup_{\tau\in [h c\langle D\sqrt{i+j}\rangle^{-\beta}, (h+1) c\langle D\sqrt{i+j}\rangle^{-\beta}] } \|\Phi_N(\tau)u\|_{H^s}=
\sup_{|t|\leq c\langle D\sqrt{i+j}\rangle^{-\beta}}\|\Phi_N(t)(\Phi_N(h c\langle D\sqrt{i+j}\rangle^{-\beta})u)\|_{H^s}.$$
If we set $\Phi_N(h c\langle D\sqrt{i+j}\rangle^{-\beta})u=:u_0$, then by definition of $\Sigma^{i,j}_{N,s, D}$  we have that $u_0\in B^{i,j}_{s,D}$. We can then use 
Proposition \ref{cauchytheory} with $S:=\langle D\sqrt{i+j}\rangle$ to conclude. Then 
from \eqref{elemNT} we get \eqref{giancant}.
\end{proof}

Next we split the proof of Theorem \ref{probBBB} in two parts. First we construct a full measure set in $H^s$ such that items $(i), (ii)$ are fulfilled.
In a second step we refine the construction of the full measure set in such a way that it is invariant along the flow $\Phi(t)$. Moreover the measures $\rho_{2,R}$ 
are invariant  on this set for every $R>0$.

\subsection{Construction of a full measure set $\Sigma^s\subset H^s$ of global solutions to \eqref{data}}\label{subsecglo}

Along this subsection we assume $\frac 43<s<\frac 32$ (which corresponds to the case $n=2$ along Theorem \ref{probBBB}).
\begin{prop}\label{noRiyes}
For every $i\in \N$, $R>0$,  $\frac 43<s<\frac 32$  there exists 
$\Sigma_{i,R}^s\subset H^s$ and 
a constant $D>0$ such that:
\begin{equation}
\Sigma_{i,R}^s \hbox{ is closed in } H^s;
\end{equation}
\begin{equation}\label{rinf}
\rho_{2,R}(H^s\setminus \Sigma_{i,R}^s)\leq 2^{-i};
\end{equation}
\begin{multline}\label{aubg...}
\forall u\in \Sigma_{i,R}^s \quad \exists j_k, L_k\rightarrow \infty, \quad u_k\in H^s \\ \hbox{ s.t. }  \|u_k- u\|_{H^s}\overset{k\rightarrow \infty}\rightarrow 0, \quad
\sup_{|\tau|<2^{j'}} \|\Phi_{L_{k}}(\tau)u_k\|_{H^s}\leq D \sqrt{i+1+j'}, \quad \forall j'\in \{1,\cdots, j_k\}.
\end{multline}
\end{prop}
\begin{proof}
We claim that we can choose $D$ and a sequence
$N_j\rightarrow \infty$ such that
\begin{equation}\label{shirley}
 \rho_{2,R, N_j}(H^s \setminus \Sigma^{i,j'}_{N_j,s,D})
\leq 2^{-i-j'}, \quad \forall j'\in \{1, \dots, j\}. 
\end{equation}
In fact notice that by \eqref{defsigma} we have:
\begin{multline}\label{subadd}
\rho_{2,R,N}(H^s\setminus \Sigma^{i,j}_{N,s,D})=
\rho_{2,R,N}\Big (\bigcup_{h=-[\frac{2^j\langle D\sqrt{i+j}\rangle^\beta}c]}^{[\frac{2^j\langle D\sqrt{i+j}\rangle^\beta}c]}  H^s\setminus  \Phi_{N}(h c\langle D\sqrt{i+j}\rangle^{-\beta})(B^{i,j}_{s,D})\Big)
\\\leq \sum_{h=-[\frac{2^j\langle D\sqrt{i+j}\rangle^\beta}c]}^{[\frac{2^j\langle D\sqrt{i+j}\rangle^\beta}c]} \rho_{2,R,N}\Big(\Phi_N(h c\langle D\sqrt{i+j}\rangle^{-\beta})(H^s\setminus  B^{i,j}_{s,D})\Big),
\end{multline}
where we denoted by $[x]$ the integer part of $x$. Next we recall that by \eqref{almostconv} we get 
for every $\varepsilon>0$, $j\in \N$ the existence of $N_{\varepsilon, j} $  such that
\begin{equation}\label{..man..}
\sup_{t\in [-2^j, 2^j]}\Big |\rho_{2,R,N} \Big(\Phi_{N}(t)(H^s\setminus  B^{i,j}_{s,D})\Big)-
\rho_{2,R,N} \Big (H^s\setminus  B^{i,j}_{s,D})\Big )\Big |\leq \varepsilon, \quad \forall N\geq N_{\varepsilon,j}\end{equation}
and hence we can continue the estimate \eqref{subadd}
as follows
\begin{multline*}
\rho_{2,R,N}(H^s\setminus \Sigma^{i,j}_{N,s, D})\leq 
2 [\frac{2^j\langle D\sqrt{i+j}\rangle^\beta}c]
\big (\rho_{2,R,N} (H^s\setminus  B^{i,j}_{s,D})
+ \varepsilon\big )\\\leq
\frac{K2^{j+1}\langle D\sqrt{i+j}\rangle^\beta}c
e^{-k D^2 (i+j)} + \frac{2^{j+1}\langle D\sqrt{i+j}\rangle^\beta}c\varepsilon,\quad \forall N\geq N_{\varepsilon,j}
 \end{multline*}
where we have used \eqref{gaussianbound}.
Notice that by the same argument as above we have the following bound as well:
\begin{multline}\label{rhonr}
\rho_{2,R,N}(H^s\setminus \Sigma^{i,j'}_{N,s,D})
\leq
\frac{K2^{j'+1}\langle D\sqrt{i+j'}\rangle^\beta}c
e^{-k D^2 (i+j')} + \frac{2^{j'+1}\langle D\sqrt{i+j'}\rangle^\beta}c\varepsilon,\\
\quad \forall j'\in\{0, \dots j\}, \quad N\geq N_{\varepsilon,j}.
\end{multline}
Next we choose
$\varepsilon=2^{-D(i+j)}$ with $D$ large enough in such a way that 
$$
\frac{\langle  D\sqrt{i+j'}\rangle^\beta}c \big (K2^{j'+1}
e^{-k  D^2 (i+j')} + 2^{j'+1}2^{- D(i+j)}\big )
\leq 2^{-i-j'}, \quad \forall j'\in \{1,\dots, j\}.$$
Hence from \eqref{rhonr} we get
\begin{equation*}
\rho_{2,R,N}(H^s\setminus \Sigma^{i,j'}_{N,s, D})
\leq 2^{-i-j'}, \quad \forall j'\in \{1, \dots, j\}, \quad N\geq N_{2^{- D(i+j)},j},
\end{equation*}
which in turn is equivalent to \eqref{shirley} if we choose $N_j=N_{2^{- D(i+j)},j}$.\\
Next we define
$$\tilde \Sigma^{i,j}_{N,s, D}=\bigcap_{j'=0}^j \Sigma^{i,j'}_{N,s, D}$$
and
by \eqref{shirley} we get
\begin{equation}\label{2allamenoi}\rho_{2,R,N_j} \Big (H^s\setminus \tilde \Sigma^{i,j}_{N_j,s,D}\Big )\leq 2^{-i}.\end{equation}
Next we introduce 
 the set $\Sigma_{i,R}^s$ as follows:
\begin{equation}\label{sigmais}
\Sigma_{i,R}^s =\Big \{u\in H^s \hbox{ s.t. } \exists  j_k, L_k \rightarrow \infty, u_{k}\in \tilde \Sigma^{i,j_k}_{L_k,s, D} \hbox{ with }
\|u_{k}-u\|_{H^s} \overset{k\rightarrow \infty} \longrightarrow 0\Big \}.
\end{equation}
Notice that the dependence of the set above from $R$ is related with the fact that indeed $ D$ (that appears in the definition above) depends on $R$, which in turn is related to  the Gaussian bounds used above. Notice that also the bound
\eqref{..man..} depends on $R$, in view of Proposition \ref{impo} which depends on $R$.
We have the following inclusion
\begin{equation}\label{limsupsetth}
\bigcap_{J\in \N} \Big (\bigcup_{j\geq J} \tilde \Sigma^{i,j}_{N_j,s, D}\Big)\subset \Sigma_{i,R}^s,
\end{equation}
where the sequence $N_j$ is provided by \eqref{shirley}.
Next we show that $\Sigma_i^s$ satisfies all the required properties.
It is easy to check by an elementary diagonal argument that  $\Sigma_{i,R}^s$ is a closed set. 
In order to prove \eqref{rinf} first notice that by \eqref{limsupsetth}  we get:
\begin{equation*}
\rho_{2,R} (H^s\setminus \Sigma_{i,R}^s )
\leq \rho_{2,R} \Big (H^s\setminus \bigcap_{J\in \N} \big (\bigcup_{j\geq J} \tilde \Sigma^{i,j}_{N_j,s, D}\big)\Big)
=\lim_{J\rightarrow \infty} \rho_{2,R} \Big(H^s\setminus \bigcup_{j\geq J} \tilde \Sigma^{i,j}_{N_j,s, D}\Big)
\end{equation*} where we used that the family of sets $\bigcup_{j\geq J} \tilde \Sigma^{i,j}_{N_j,s, D}$ is decreasing in $J$. Then by 
\eqref{convrhoNR} we get
\begin{equation*}
\lim_{J\rightarrow \infty} \rho_{2,R} \Big(H^s\setminus \bigcup_{j\geq J} \tilde \Sigma^{i,j}_{N_j,s, D}\Big)
=\lim_{J\rightarrow \infty} \rho_{2,R,N_J} \Big(H^s\setminus \bigcup_{j\geq J} \tilde \Sigma^{i,j}_{N_j,s, D}\Big)
\leq \limsup_{J\rightarrow \infty} \rho_{2,R, N_J} \Big (H^s\setminus \tilde \Sigma^{i,J}_{N_J,s, D}\Big)
\leq 2^{-i}
\end{equation*}
where we have used \eqref{2allamenoi} at the last step.
Finally notice that
\eqref{aubg...} comes from the definition of $\Sigma^s_{i,R}$ (see \eqref{sigmais}) in conjunction with Lemma \ref{i+1}.

\end{proof}

\begin{proof}[Proof of items $(i), (ii)$ in Theorem \ref{probBBB}]
Next we introduce a set $\Sigma^s\subset H^s$ which satisfies items $(i), (ii)$ in Theorem \ref{probBBB}. 
For every $R>0$ we set
\begin{equation}\label{Sigmai}
\Sigma^s_R=\bigcup_{i\in \N} \Sigma^s_{i,R}.\end{equation}
It is clear that it is a $F_\sigma$ subset in $H^s$, moreover due to \eqref{rinf} we get 
\begin{equation}\label{..sl}
\rho_{2,R}(H^s\setminus \Sigma^s_R)=0\end{equation}
Next we select $R=j$ and we can define 
$$\Sigma^s=\bigcup_{j\in \N} \Sigma^s_j.$$
Since we have  $\rho_{2,j}(H^s\setminus \Sigma^s_j)=0$ we conclude by Proposition \ref{luccidi} that 
$\mu_2(H^s\setminus \Sigma^s)=0$ and item $(i)$ of Theorem \ref{probBBB} is proved.
Based on the local Cauchy theory, in order to prove item $(ii)$, it is sufficient to show that for every $u\in \Sigma^s$ we have that the $H^s$ norm cannot blow up in finite time, 
indeed we show a logarithmic upper bound for the $H^s$ norm. By the definition of $\Sigma^s$ it is sufficient to prove this property for initial data belonging to $\Sigma^s_{i,R}$,
with $i, R$ fixed.
Indeed by \eqref{aubg...} for any $u\in \Sigma^s_{i,R}$ 
there exist $j_k, L_k\rightarrow \infty$ and a sequence $u_k$ in $H^s$ such that $u_k\overset{H^s} \rightarrow u$  such that
$$\sup_{|\tau|<2^{j'}} \|\Phi_{L_k}(\tau) u_k\|_{H^s}\leq D \sqrt{i+1+j'}, \quad \forall j'\in\{0, \cdots, j_k\}$$
and hence  from \eqref{convspri} we have the following bound
\begin{equation}\label{j==o}\sup_{|\tau|<2^{j}} \|\Phi(\tau) u\|_{H^{s}}\leq D \sqrt{i+1+j}, \quad \forall j, \quad u\in \Sigma^s_{i,R}.
\end{equation}
By the bound above
we get
$$\sup_{2^l\leq |\tau| < 2^{l+1}} \|\Phi(\tau) u\|_{H^{s}}\leq D \sqrt{i+2+\log_2 2^l}, \quad \forall l, \quad u\in \Sigma^s_{i,R}$$
which in turn implies
$$\sup_{2^l\leq |\tau| < 2^{l+1}}  \|\Phi(\tau) u\|_{H^{s}}\leq D \sqrt{i+2+\log_2 |\tau|}, \quad \forall l, \quad u\in \Sigma^s_{i,R}$$
and hence, since $l$ is arbitrary, we have
$$\|\Phi(\tau) u\|_{H^{s}}\leq D \sqrt{i+2+\log_2 |\tau|}, \quad \forall |\tau|>1, \quad u\in \Sigma^s_{i,R}.$$
On the other hand we also have by \eqref{j==o}, where we choose $j=0$,
$$\sup_{|\tau| <1} \|\Phi(\tau) u\|_{H^{s}}\leq D \sqrt{i+1}$$
and hence gathering together the estimates above we conclude
$$\|\Phi(\tau) u\|_{H^{s}}\leq D \sqrt{i+2+\log_2(1+|\tau|)}, \quad \forall |\tau|>0, \quad u\in \Sigma^s_{i,R}.$$

\end{proof}

\subsection{Construction of the invariant set and invariance of the measure}

Next for $\frac43<s<\frac32$ we construct a full measure set $\Sigma^s\subset H^s$, with respect to $\mu_2$, such that $\Phi(t)(\Sigma^s)=\Sigma^s$ and moreover the measure $\rho_{2,R}$
is invariant once restricted on $\Sigma^s$, for every $R>0$.
We fix $R>0$ and  for $\bar s\in (s, \frac 32)$ we pick a sequence
$$\quad s_l\nearrow \bar s, \quad s_l\in (s, \frac 32).$$
Then we consider the following subset of $H^s$
\begin{equation}\label{slashsi}\Sigma_R=\bigcap_{l\in \N} \Sigma^{s_l}_R\end{equation}
where $\Sigma^{s_l}_R$ are defined in \eqref{Sigmai} for $s=s_l$.
Notice that since $s_l>s$ we have that  $\Sigma_R^{s_l}\subset H^s$ and hence $\Sigma_R\subset H^s$. Moreover we can easily check
$\rho_{2,R}(H^s\setminus \Sigma_R)=0$. In fact we have
$$\rho_{2,R}(H^s\setminus \Sigma_R)\leq \rho_{2,R}\Big (\bigcup_l (H^{s}\setminus \Sigma_R^{s_l})\Big )
\leq \sum_l  \rho_{2,R} (H^{s}\setminus \Sigma_R^{s_l})=\sum_l  \rho_{2,R} (H^{s_l}\setminus \Sigma_R^{s_l})
$$
where in the last step we have used the fact that, due to the property on the support of the Gaussian measure $\mu_2$ (see Section \ref{poBBGA}) and since $s_l>s$, we have $\mu_2(H^s\setminus H^{s_l})=0$. We conclude by recalling \eqref{..sl} with $s=s_l$.

We claim that the set $\Sigma_R$ defined in \eqref{slashsi} is invariant by the flow $\Phi(t)$ for every $t$. Once this fact is proved then it is sufficient to choose
$$\Sigma^s=\bigcup_{j\in \N} \Sigma_j.$$
Notice that by the construction in Subsection \ref{subsecglo} (where we fix $s=s_l$) we have that the flow $\Phi(t)$ is globally well defined on $\Sigma_R$. Therefore we only have to prove
the invariance of $\Sigma_R$ along $\Phi(t)$. In order to do that  it is sufficient to show that  for every given $l$ and $t$ we have the implication:
\begin{equation}\label{establ}u\in \Sigma^{s_{l+1}}_R \Longrightarrow \Phi(t)u\in \Sigma^{s_l}_R.\end{equation} 
In fact once \eqref{establ} is proved then given $u\in \Sigma_R$ and given any $l$ we have by definition that
$u\in \Sigma^{s_{l+1}}_R$ and hence by \eqref{establ} we get 
$\Phi(t)u\in \Sigma^{s_l}_R$. Since $l$ is arbitrary we then conclude that  $\Phi(t)u\in \Sigma_R$.
Notice that in  the definition of $\Sigma^{s}_{i,R}$ there is  the dependence on the constant $D$ from Proposition \ref{cauchytheory}, that in principle depends on $s$ and $R$. Hence,
it could happen that the constant $D$ has to change with $s_l$.
Actually this is not the case since the dependence of $D$ from $s$, at fixed $R$,  in turn is via  the corresponding constants 
$k, K$ in \eqref{gaussianbound}, that  can be chosen uniformly provided that $s\in [4/3, \bar s]$. This last fact follows
from the Gaussian bounds of  the measure $\mu_2$ which are uniform for $s\in [\frac 43, \bar s]$, see Remark \ref{rem1}.
From  now on the constant $D$ involved in the definition of $\Sigma^{s_l}_{i,R}$ is assumed to be independent of $s_l$, but it can depend on $R$.\\
Next we prove \eqref{establ} for  $t>0$ (the same argument works for $t<0$).  We select
\begin{equation}\label{jbar}\bar j= \min \{j\in \N \hbox{ s.t. }  t\leq 2^j\}.\end{equation}
Notice that if $u_0\in \Sigma^{s_{l+1}}_R$ then by definitions \eqref{sigmais} and  \eqref{Sigmai}, and by using Lemma \ref{i+1} we get
\begin{multline*}\exists i\in \N, \quad u_k\in H^{s_{l+1}}, \quad  j_k\rightarrow \infty, \quad L_k\rightarrow \infty \quad \hbox{ s.t. } \\
u_k\overset{H^{s_{l+1}}}\longrightarrow u,
\quad \sup_{|\tau|< 2^{j'}}\|\Phi_{L_k} (\tau)u_k\|_{H^{s_{l+1}}}\leq D\sqrt{i+ j'+1}, \quad \forall j'\in \{0,\dots, j_k\}.\end{multline*}
From this bound and the group property of  $\Phi_N$ we have 
\begin{equation}\label{prevmag}\sup_{|t+\tau|<2^{j'}}\|\Phi_{L_k} (\tau) (\Phi_{L_k}(t)u_k)\|_{H^{s_{l+1}}}\leq D\sqrt{i+ j'+1}, \quad \forall j'\in \{0,\dots, j_k\}.
\end{equation}
Next notice that for every $j'>\bar j$  we have the inclusion
$$[-2^{j'-1}, 2^{j'-1}]\subset [-t-2^{j'}, -t+2^{j'}], \quad  \forall j'>\bar j$$
and hence from the bound \eqref{prevmag} we get
\begin{equation}\label{prevmag2}\sup_{|\tau|<2^{j'-1}}\|\Phi_{L_k} (\tau) (\Phi_{L_k}(t)u_k)\|_{H^{s_{l+1}}}\leq D\sqrt{i+ j'+1}, \quad \forall j'\in \{\bar j+1,\dots, j_k\}.
\end{equation}
On the other hand by \eqref{prevmag2} for $j'=\bar j+1$ we get
\begin{equation*}\sup_{|\tau|<2^{\bar j}}\|\Phi_{L_k} (\tau) (\Phi_{L_k}(t)u_k)\|_{H^{s_{l+1}}}\leq D\sqrt{i+ \bar j+2}
\end{equation*}
which trivially implies 
\begin{equation}\label{prevmag3}\sup_{|\tau|<2^{j'}}\|\Phi_{L_k} (\tau) (\Phi_{L_k}(t)u_k)\|_{H^{s_{l+1}}}\leq D\sqrt{i+ \bar j+2+j'},
\quad \forall  j'\in \{1, \cdots, \bar j\}.
\end{equation} 
By combining \eqref{prevmag2} and \eqref{prevmag3}, and   the fact that $\|\cdot \|_{H^{s_{l}}}\leq \|\cdot\|_{H^{s_{l+1}}}$, we obtain
\begin{equation*}\sup_{|\tau|<2^{j'}}\|\Phi_{L_k} (\tau) (\Phi_{L_k}(t)u_k)\|_{H^{s_{l}}}\leq D\sqrt{i+ \bar j+2+j'},
\quad \forall j'\in \{1, \cdots, j_k-1\}.
\end{equation*}
Moreover, due to \eqref{interpolSS}, we get
$\Phi_{L_k}(t)u_k\overset{H^{s_l}}\longrightarrow \Phi(t)u$. Then we can conclude
$\Phi(t) u\in \Sigma^{s_l}_{i+\bar j+2,R}$ and \eqref{establ} follows by definition of $\Sigma^{s_l}_R$.\\
Once the invariant Borel  set $\Sigma$ has been constructed then the invariance of the measure on $\Sigma$ follows by  exactly the same computations
done for the Benjamin-Ono equation in \cite{TV2} .

\section{Extension to higher order conservation laws}\label{RNU}

As already mentioned in  the introduction, for every integer $n$ 
there exists a conservation law for \eqref{data} with the structure
$$E_{2n+1} (u)=\|\partial^{n} u\|_{L^2}^2 +  R_n(u)$$
where $R_n(u)$ is the integral of a linear combination of densities belonging to the family
\begin{equation*}
{\mathcal D}_n=\Big \{\prod_{i=1}^N \partial^{\alpha_i} v, \quad v\in \{u, \bar u\} \hbox{ s.t. } N\geq 2, \quad \max \alpha_i\leq n-1, \quad \sum_{i=1}^N \alpha_i \leq 2 n-2\Big \}.
\end{equation*}
We also introduce the following subsets of ${\mathcal D}_n$:
\begin{multline*} {\mathcal D}_n^1=\Big \{\prod_{i=1}^N \partial^{\alpha_i} v \in {\mathcal D}_n
\hbox{ s.t. } \#\{i\in \{1, \cdots, N\} \hbox{ s.t. } \alpha_i\neq 0\}\geq 3 \Big\},\\
{\mathcal D}_n^2=\Big \{\prod_{i=1}^N \partial^{\alpha_i} v, \quad v\in \{u, \bar u\} \hbox{ s.t. } N\geq 2, \quad \max \alpha_i\leq n-1, \quad \sum_{i=1}^N \alpha_i< 2 n-2\Big \}.
\end{multline*}
Notice that the densities in  $ {\mathcal D}_n^1$ have at least three factors that appear with a non-trivial derivative, while the densities in 
${\mathcal D}_n^2$ involve in total less than $2n-2$ derivatives.\\
Next we give a more precise structure on the energies $E_{2n+1}$, indeed we isolate all the parts of the energy which, beside the leading quadratic part,
involve terms belonging to ${\mathcal D}_n\setminus ({\mathcal D}_n^1\cup {\mathcal D}_n^2)$.
Next we shall denote for shortness 
$$\tilde {\mathcal D}_n={\mathcal D}_n\setminus ({\mathcal D}_n^1\cup {\mathcal D}_n^2).$$

\subsection{On the structure of $E_{2n+1}$} As recalled in  the introduction, the energies $E_{2n+1}(u)$
can be constructed by an induction argument as follows:
\begin{equation}\label{En2}E_{j}(u)=\Re \int \bar u w_{j}(u) dx, \quad \forall j\in \N,\end{equation}
where
\begin{equation}\label{induction}w_1(u)=u, \quad w_{j+1} (u) = D w_j + \bar u \sum_{k=1}^{j-1} w_k(u) w_{j-k}(u), \quad D=-i\partial_x.\end{equation}
We shall use the following decomposition
$$w_j(u)=w_j^L(u) + w_j^C(u)+ T_j(u)$$
where $w_j^L(u)$ and $w_j^C(u)$ are the homogenous parts of $w_j(u)$ respectively of order one and three, while $T_j(u)$ involves all the terms of $w_j(u)$ 
whose homogeneity is of order larger or equal than four.
Next we extract a more precise structure for $E_{2n+1}(u)$, namely we shall write its expression up to terms belonging to $\tilde {\mathcal D}_n$.
\begin{prop}
For every number $n\geq 2$ we have
\begin{equation}\label{structen2}
E_{2n+1} (u)= \|\partial^{n} u\|_{L^2}^2 + 
 \int (\partial^{n-1} |u|^2)^2 + (4n-2)   \int |\partial^{n-1} u|^2|u|^2 + \tilde R_n(u)
 \end{equation}
where 
$\tilde R_n(u)$ is the integral of a linear combination of densities belonging to $\tilde {\mathcal D}_n$.
\end{prop}
\begin{remark}
The expression \eqref{structen2} should be compared with the explicit expression we have for the conservation law $E_5(u)$ (namely $n=2$)
given in \eqref{consexpl}.
\end{remark}
\begin{proof}
By using \eqref{induction} one can check the identities:
\begin{equation}\label{lIn}w_j^L(u)=D^{j-1} u,\end{equation}
\begin{equation}\label{cUb}w_{j+1}^C(u)= D w_j^C(u)+ \bar u \sum_{k=1}^{j-1} w_k^L(u) w_{j-k}^L(u).
\end{equation}
By combining \eqref{lIn} with \eqref{En2} we get
that the quadratic part of the energy $E_{2n+1}(u)$
is given by
$$\int \bar u w_{2n+1}^L(u)=\int \bar u D^{2n} u=
\int |\partial^n u|^2.$$
Next we focus on the part of the conservation law $E_{2n+1}(u)$ homogenous of order $4$, which in turn by \eqref{En2} depends
on the cubic part of $w_{2n+1}(u)$, namely $w_{2n+1}^C(u)$. By combining \eqref{lIn} and \eqref{cUb} we get by induction 
the following identity:
\begin{equation*}w_{j+1}^C(u)
= D^{j-2} (\bar u \sum_{k=1}^{1} w_1^L(u) w_{1}^L(u))+\bar u \sum_{k=1}^{j-1} w_k^L(u) w_{j-k}^L(u)+ \sum_{i=1}^{j-3} D^{i} (\bar u \sum_{k=1}^{j-i-1} w_k^L(u) w_{j-i-k}^L(u)).
\end{equation*}
Hence by using \eqref{En2} and integration by parts,  we get the following expression for the part of energy $E_{2n+1}(u)$ with homogeneity four:
\begin{multline*}
\int \bar u w_{2n+1}^C(u) dx=\int D^{2n-2} \bar u \bar u u^2 + \int \bar u \bar u \sum_{k=1}^{2n-1} D^{k-1} uD^{2n-k-1}u
\\+ \sum_{i=1}^{2n-3} (-1)^i\int D^{i} \bar u (\bar u \sum_{k=1}^{2n-i-1} D^{k-1}u D^{2n-i-k-1}u)
= \int D^{2n-2} \bar u \bar u u^2 + 2 \int \bar u \bar u u D^{2n-2} u \\+\int \bar u \bar u (\sum_{k=2}^{2n-2} D^{k-1} uD^{2n-k-1}u)
+ 2 \sum_{i=1}^{2n-3} (-1)^i \int D^{i} \bar u \bar u u D^{2n-i-2}u+ \tilde R_n(u)\end{multline*}
where $\tilde R_n(u)$ is the integral of terms in $ \tilde {\mathcal D}_n$ (the explicit expression $\tilde R_n(u)$ in the following computations can change from line to line)
and we can continue by using integration by parts
\begin{multline*}
\cdots = (-1)^{n-1} \int (D^{n-1} \bar u)^2 u^2
+2 (-1)^{n-1} \int |D^{n-1}u|^2 |u|^2+ 2 (-1)^{n-1} \int (D^{n-1} u)^2 \bar u^2\\+4 (-1)^{n-1} 
\int |D^{n-1}u|^2 |u|^2
+ 2 \sum_{k=2}^{n-1} (-1)^{2n-k-1-n+1}\int (D^{n-1} u)^2 \bar u^2 + \int (D^{n-1} u)^2 \bar u^2
\\+(-1)^{n-1} (4n-6) \int |D^{n-1}u|^2 |u|^2+\int \tilde R_n(u),\end{multline*}
which by the elementary identity
$$2(-1)^{n-1} + 2 \sum_{k=2}^{n-1} (-1)^{n-k}+1=(- 1)^{n-1}$$ implies
\begin{multline*}\cdots = (-1)^{n-1} (-i)^{2n-2} \int \big((\partial ^{n-1} \bar u)^2 u^2 + (\partial^{n-1} u)^2 \bar u^2\big) 
\\+4n (-1)^{n-1}  (-i)^{2n-2} \int |\partial ^{n-1}u|^2 |u|^2+\int \tilde R_n(u)
\\=\int \big((\partial ^{n-1} \bar u) u + (\partial^{n-1} u) \bar u\big)^2
-2    \int |\partial^{n-1} u|^2 |u|^2
+4n \int |\partial^{n-1}u|^2 |u|^2+\int \tilde R_n(u)
\\= \int \big(\partial^{n-1} |u|^2 - \sum_{l=1}^{n-2} \binom{n-1}{l}\partial^l u \partial^{n-1-l} \bar u \big)^2
+(4n-2)  \int |\partial^{n-1} u|^2 |u|^2
+\int \tilde R_n(u).\end{multline*}
We conclude the proof since 
$$\big(\partial^{n-1} |u|^2 - \sum_{l=1}^{n-2} \binom{n-1}{l}\partial^l u \partial^{n-1-l} \bar u \big)^2
- (\partial^{n-1} |u|^2)^2\in \tilde {\mathcal D}_n.$$

\end{proof}
\subsection{Computation and estimate of $\frac d{dt} E_{2n+1}(\Pi_N \Phi_N(t)u)_{t=0}$ for  the leading term}
Recall that the hardest part in  the proof of Theorem \ref{probBBB} for $n=2$ was the proof of \eqref{nahm}, which in turn reduces to the estimate \eqref{lucccc}.
Next we isolate the corresponding leading term in the case $n>2$.
Notice that in order to compute
$\frac d{dt} E_{2n+1}(\Pi_N \Phi_N(t)u)_{t=0}$
the most delicate terms to estimate are \begin{equation}\label{PRIN}
\Re  \int \partial^{n-1} (|\Pi_N u|^2) \partial^{n-1} (\Pi_{>N} (|\Pi_N u|^2 \partial \Pi_N u )
 \Pi_N \bar u), 
 \end{equation}
 and
 \begin{multline}\label{PRIN1}
 \Re  \int |\partial^{n-1} (\Pi_N u)|^2 (\Pi_{>N} (|\Pi_N u|^2 \partial \Pi_N u )
 \Pi_N \bar u)
\\+\Re  \int |\Pi_N u|^2 \partial^{n-1} (\Pi_{>N} (|\Pi_N u|^2 \partial (\Pi_N u) ) \partial^{n-1} (\Pi_N \bar u).
\end{multline}
Concerning the term \eqref{PRIN} 
it can be written as follows 
$$\Re  \int \partial^{n-1} (|\Pi_N u|^2) (\Pi_{>N} (|\Pi_N u|^2 \partial^{n}  \Pi_N u )
 \Pi_N \bar u + l.o.t.$$
 where l.o.t. will denote from now on   integrals of any density involving derivatives of  at most order
 $n-1$. Hence we can continue as follows
 \begin{multline*}\cdots=\Re  \int \partial^{n-1} (|\Pi_N u|^2) \partial^{n-1} \Pi_{>N} (|\Pi_N u|^2 \partial  \Pi_N u )
\Pi_N \bar u + l.o.t.\\
=\Re  \int \partial^{n-1} (|\Pi_N u|^2 \Pi_N \bar u ) \partial^{n-1}  \Pi_{>N} (|\Pi_N u|^2 \partial \Pi_N u )
\\-\Re  \int \partial^{n-1} \Pi_N \bar u |\Pi_N u|^2 \partial^{n-1}  \Pi_{>N} (|\Pi_N u|^2 \partial \Pi_N u )
+ l.o.t.\\
= \Re  \int \partial^{n-1} (|\Pi_N u|^2 \Pi_N \bar u ) \partial^{n-1}  \Pi_{>N} (|\Pi_N u|^2 \partial \Pi_N u )
\\-\Re  \int \partial^{n-1} \Pi_N \bar u |\Pi_N u|^2 \Pi_{>N} (|\Pi_N u|^2 \partial^n \Pi_N u )
+l.o.t.\\
=\Re  \int \partial^{n-1} (|\Pi_N u|^2 \Pi_N \bar u ) \partial \Pi_{>N} (|\Pi_N u|^2 \partial^{n-1}  \Pi_N u )
\\-\Re  \int \partial^{n-1} \Pi_N \bar u |\Pi_N u|^2 \partial^{n-1} \Pi_{>N} (|\Pi_N u|^2 \partial \Pi_N u )
+l.o.t.\\=\Re  \int \partial^{n-1} (|\Pi_N u|^2 \Pi_N \bar u ) \partial^{n-1} \Pi_{>N} (|\Pi_N u|^2 \partial  \Pi_N u )
+l.o.t,
\end{multline*}
where we used the fact that
\begin{equation}\label{IPPn}
\Re  \int (\partial^{n-1} \Pi_N \bar u |\Pi_N u|^2) \partial \Pi_{>N} (|\Pi_N u|^2 \partial^{n-1} \Pi_N u )
=\frac 12 \int \partial \Pi_{>N}(|\Pi_N u|^4 |\partial^{n-1} \Pi_N \bar u|^2)=0.\end{equation}
Next we focus on \eqref{PRIN1}. Arguing as above the expression \eqref{PRIN1} can be written as follows:
 \begin{multline*}
 \Re  \int |\Pi_N u|^2 \partial^{n-1} \Pi_{>N} (|\Pi_N u|^2 \partial \Pi_N u ) \partial^{n-1} \Pi_N \bar u+l.o.t.\\
 = \Re  \int |\Pi_N u|^2  \Pi_{>N} (|\Pi_N u|^2 \partial^n \Pi_N u ) \partial^{n-1} \Pi_N \bar u+l.o.t.\\
 = \Re  \int |\Pi_N u|^2  \partial^{n-1}  \Pi_{>N} (|\Pi_N u|^2 \partial \Pi_N u ) \partial^{n-1} \Pi_N \bar u+l.o.t.=l.o.t.
 \end{multline*}
hence the most dangerous term to be treated is 
 $$\Re  \int \partial^{n-1} (|\Pi_N u|^2 \Pi_N \bar u) \partial^{n-1} \Pi_{>N} (|\Pi_N u|^2 \partial \Pi_N u ).
 $$
 Notice that this term has a structure similar to the one of \eqref{nahm} in the case $n=2$.
 By replacing the random vector 
 \eqref{randomizedgen}
 in the expression above we are reduced to showing the limit
 \begin{equation}\label{luccccgen}
\Big \| \Im  \sum_{\substack{j_1,j_2,j_3,j_4,j_5,j_6\in {\Z_{\leq N}}\\ 
j_1-j_2-j_3+ j_4-j_5+j_6=0\\
|j_4-j_5+j_6|>N}} \frac{(j_1-j_2-j_3)^{n-1}(j_4-j_5+j_6)^{n-1} j_6}{\langle j_1^{2n}\rangle \langle j_2^{2n} \rangle
\langle j_3^{2n} \rangle \langle j_4^{2n}\rangle \langle j_5^{2n}\rangle \langle j_6^{2n} \rangle}g_{\vec j}
(\omega)\Big \|_{L^2_\omega}\overset{N\rightarrow \infty} \longrightarrow 0 \end{equation}
which has the same symmetric structure as \eqref{lucccc} up to the exponents $n-1$ and $2n$, and hence can be treated in the same way {\em mutatis mutandis}.


\begin{thebibliography}{1}


\bibitem{Ad2024} Adams, Wellposedness for the NLS hierarchy. Preprint, 
arXiv:2402.07652.

\bibitem{ABITZ24} T. Alazard, N.  Burq, M.  Ifrim, D.Tataru and C. Zuily, Nonlinear interpolation and the flow map for quasilinear equations. Preprint, 
arXiv:2410.06909.

\bibitem{BeCo1984} R. Beals and R. R. Coifman, 
Scattering and inverse scattering for first order systems, 
{\em Communications on Pure and Applied Mathematics}, 37 (1984), 39-90.

\bibitem{BeCo1985} R. Beals and R. R. Coifman,
Inverse scattering and evolution equations, 
{\em Communications on pure and applied mathematics}, 38 (1985), 29-42.

\bibitem{BeCo1987} R. Beals and R. R. Coifman, 
Scattering and inverse scattering for first-order systems. II, 
{\em Inverse Problems}, 3 (1987), 577-593.


\bibitem{BPS96} B. Birnir, G. Ponce and N. Svanstedt, The local ill-posedness of the modified KdV equation, 
{\em Ann. Inst. H. Poincar\'e  Anal. Non Lin\'aire} 
13 (4) (1996), 529--535.

\bibitem{Bog} V. Bogachev, Gaussian Measures, {\it Mathematical Surveys and Monographs}, AMS, 2015.

\bibitem{BS75} J. Bona and R. Smith, The initial-value problem for the Korteweg-de Vries equation, {\it Philosophical Transactions of the Royal Society of London. Series A, Mathematical and Physical Sciences}, 278, (1975), 555--601.
  
\bibitem{B1993} 
J. Bourgain, Fourier transform restriction phenomena for certain lattice subsets and applications to nonlinear evolution equations II. The KdV equation,
{\em Geom. Funct. Anal.} 
3 (3) (1993), 209--262.

\bibitem{Bou2} J. Bourgain, Periodic nonlinear Schr\"odinger equation and invariant measures, {\it Comm. Math. Phys. }166(1), 1--26 (1994).

\bibitem{Bou3} J. Bourgain,  Invariant measures for the 2D-defocusing nonlinear Schr\"odinger equation, {\it Commun.Math. Phys.} 176, 421--445 (1996).

\bibitem{btt} N. Burq, L.Thomann and N. Tzvetkov, Long time dynamics for the one dimensional NLS, 
{\it Annales de l'Institut Fourier}, Volume 63 (2013) no. 6, pp. 2137--2198.




\bibitem{Ch21} A. Chapouto, A remark on the well-posedness of the modified KDV equation in the Fourier-Lebesgue spaces, 
{\em Discrete Contin. Dyn. Syst.} 41(8) (2021), 3915-3950. 

\bibitem{Ch23} A. Chapouto,  A refined well-posedness result for the modified KdV equation in the Fourier-Lebesgue spaces,
{\em J. Dynam. Differential Equations} 35 (2023), no. 3, 2537-2578.


\bibitem{ChF} A. Chapouto and J. Forlano, Invariant measures for periodic KdV and mKdV equations using complete inetgrability. Preprint,  arXiv:2305.14565.

\bibitem{CLO} A. Chapouto, G. Li, and T. Oh,  Deep-water and Shallow-water Limits of Statistical Equilibria for the Intermediate Long Wave Equation. Preprint, arXiv:2409.06905.


\bibitem{CG20} M. Chen and B. Guo, Local well and ill-posedness for the modified KdV equations in subcritical modulation spaces, 
{\em Commun. Math. Sci.} 18 (4) (2020), 909--946.

\bibitem{CCT03} M. Christ, J. Colliander and T. Tao,  Asymptotics, frequency modulation, and 
low regularity ill-posedness for canonical defocusing equations, 
{\em Amer. J. Math.} 125(6) (2003), 1235-1293. 



\bibitem{CKSTT03} J. Colliander, M. Keel, G. Staffilani, H. Takaoka and T. Tao, 
Sharp global well-posedness for KdV and modified 
KdV on $\R$ and $\T$, 
{\em J. Amer. Math. Soc.} 16(3) (2003), 705-749. 

\bibitem{CKSTT04} J. Colliander, M. Keel, G. Staffilani, H. Takaoka and T. Tao, Multilinear estimates for periodic KdV equations, and applications, {\em Journal of Functional Analysis}, 
Volume 211  (2004),  173--218.


\bibitem{DeZh2003} P. Deift and X. Zhou, Long-time asymptotics for solutions of the NLS equation with initial data in a weighted
Sobolev space,  
{\em Commun. Pure Appl. Math.} 56 (2003), 1029--1077.

\bibitem{DNY1} 
Y. Deng, A. Nahmod and H. Yue, Invariant Gibbs measures and global strong solutions for nonlinear Schr\"odinger
equations in dimension two, {\it Annals of Math.} Vol. 200, (2024), to appear.

\bibitem{DNY2} 
Y. Deng, A. Nahmod and H. Yue, Random tensors, propagation of randomness and nonlinear dispersive equations, {\it Inventiones Math},  2022, 539-686.


\bibitem{DTV} Y. Deng, N. Tzvetkov and N. Visciglia, 
Invariant measures and long time behavior for the Benjamin-Ono equation III.
{\it Comm. Math. Phys.}   339 (2015), no. 3, 815--857.

\bibitem{FaTa2007} L. D. Faddeev and L. A. Takhtajan, Hamiltonian methods in the theory of solitons, Classics in Mathematics,
Springer, Berlin, english ed., 2007.

\bibitem{GLV1} G. Genovese, R. Luc\'a, D. Valeri, Invariant measures for the periodic derivative nonlinear Schr\"odinger equation, {\it Math. Ann.} 374 (2019), no. 3-4, 1075--1138.

\bibitem{GLV2} G. Genovese, R. Luc\'a, D. Valeri,  Gibbs measures associated to the integrals of motion of the periodic derivative nonlinear Schr\"odinger equation, 
{\it Selecta Math. }(N.S.) 22 (2016), no. 3, 1663--1702.

\bibitem{GK} P. Gerard, T. Kappeler and P. Topalov, 
Sharp well-posedness results of the Benjamin-Ono equation in $H^s(T,R)$ and qualitative properties of its solutions.
{\it Acta Math. }231 (2023), no. 1, 31--88.

\bibitem{GK} B. Gr\'ebert and T. Kappeler, The defocusing NLS equation and its normal form, {\it EMS}, 2014.

\bibitem{G04}A. Gr\"unrock, An improved local well-posedness result 
for the modified KdV equation,
{\em Int. Math. Res. Not.} 2004 (61) (2004), 3287--3308.

\bibitem{GV09} A. Gr\"unrock and L. Vega, Local well-posedness for the modified KdV equation
in almost critical $H^{s,r}$-spaces,
{\em Trans. Amer. Math. Soc.} 361 (11) (2009), 5681--5694.

\bibitem{G09} Z. Guo, Global well-posedness of Korteweg-de Vries equation in 
$H^{-\frac 34}$,
{\em J. Math. Pures Appl.}, (9) 91 (6) (2009), 583--597.


\bibitem{G13} Z. Guo, S. Kwon and T. Oh, Poincar\'e-Dulac normal form reduction for 
unconditional well-posedness of the periodic cubic  NLS, 
{\em Comm. Math. Phys.} 322(1) (2013), 19-48. 

\bibitem{HKVZ} S. Haque, R. Killip, M. Visan, and Y. Zhang, Global well-posedness and equicontinuity for mKdV in modulation spaces. Preprint, arXiv:2411.05300. 

\bibitem{IK2007} A. Ionescu and C. Kenig. Local and Global  Wellposedness of  the Periodic KP-I Equations,  {\em Mathematical Aspects of Nonlinear Dispersive Equations}, 
{\em Annals of Math. Studies} 163 (2007) 181-211.

\bibitem{HKVZ24} S. Haque, R. Killip, M. Visan, Y. Zhang, Global well-posedness and equicontinuity for mKdV in modulation spaces. Preprint, arXiv:2411.05300.

\bibitem{HGKV24} B. Harrop-Griffiths, R. Killip and M. Visan
Sharp well-posedness for the cubic NLS and mKdV in $H^{s}(\R)$,
{\em Forum of Mathematics, Pi} (2024), Vol. 12:e6 1-86.

\bibitem {KM17} T. Kappeler and J.-C. Molnar, On the well-posedness of the defocusing mKdV 
equation below $L^2$, 
{\em Siam J. Math. Anal.} 49 (3) (2017), 2191--2219.


\bibitem{KST} T. Kappeler, B. Schaad, P. Topalov, Scattering-like phenomena of the periodic defocusing NLS equation, {\it Math. Res. Lett.} 24 (2017), no. 3, 803--826.

\bibitem {KT05}T. Kappeler and P. Topalov, Global well-posedness of mKdV in $L^2(\T,\R)$, 
{\em Commun. Partial Differ. Equ.} 30 (1-3) (2005), 435--449.

\bibitem{Ka83} T. Kato, On the Cauchy problem for the (generalized) 
Korteweg - de Vries equation, in Studies in Applied Mathematics, 
{ \it Advances in Mathematics Supplement Studies}, volume 8 (Academic Press, New York, 1983), 93-128.


\bibitem{Ka1979} T. Kato, On the Korteweg-De Vries equations, 
{\em Manuscripta Math} 28, 89--99, 1979. 

\bibitem{K2004} C. Kenig, On the local and global well-posedness theory for the KP-I equation.
{\em Annales de lInst. Henri Poincar\'e} 827-837, 2004.


\bibitem{KePi2016} C.~E. Kenig and D. Pilod, Local well-posedness for the KdV hierarchy at high regularity, 
{\em Adv. Differential Equations} {\bf 21} (2016), no.9-10, 801-836. 


\bibitem{KPV01}C. Kenig, G. Ponce and L. Vega,
On the ill-posedness of some canonical dispersive equations,
{\em Duke Math. J.} 106 (2001), no. 3, 617--633.

\bibitem{KPV93}C. Kenig, G. Ponce and L. Vega,
 Well-posedness and scattering results for the generalized Korteweg-de Vries equation via 
 the contraction principle, 
{\em Comm. Pure Appl. Math.} 46 (1993), no. 4, 527 --620.

\bibitem{KPV93-1}C. Kenig, G. Ponce and L. Vega
The Cauchy problem for the Korteweg-de Vries equation in Sobolev spaces of negative indices,
{\em Duke Math. J.} 71 (1993), no. 1, 1-21.


\bibitem{KZ} C. Kenig and S. Ziesler, Local well-posedness for the modified 
Kadomtsev-Petviashvili equations, 
{\em Differential and integral Equations}, 
Vol. 18, No 10 1111-1146, 2005.

\bibitem{Ki09} N. Kishimoto, 
Well-posedness of the Cauchy problem for the KdV equation at the critical regularity,
{\em Differ. Integral Equ.} 22 (5-6) (2009), 447--464.

\bibitem{KVZ18} 
R. Killip, M. Visan and X. Zhang, Low regularity conservation laws for integrable PDE,
{\em Geom. Funct. Anal.} 28 (4) (2018), 1062--1090.

\bibitem{KKL23} F. Klaus, H. Koch, B. Liu, Well-posedness for the KdV hierarchy. Preprint, arXiv:2309.12773.


\bibitem{KT18} H. Koch and D. Tataru, Conserved energies for the cubic nonlinear Schr\"odinger 
equation in one dimension,  
{\em Duke Math. J.} 167(17) (2018), 3207--3313. 

\bibitem{KT03} H. Koch and N. Tzvetkov, On the local well-posedness of the Benjamin-Ono equation in $H^s(\R)$, {\it IMRN}  26, (2003), Pages 1449--1464.

\bibitem{KOY20}S. Kwon, T. Oh and H. Yoon, Normal form approach to unconditional well-posedness of nonlinear dispersive PDEs on the real line, 
{\em Ann. Fac. Sci. Toulouse Math.} (6) 29(3) (2020), 649--720. 

\bibitem{LRS88} J. Lebowitz, H. Rose and E. Speer, Statistical mechanics of the nonlinear Schr\"odinger equation,
{J. Stat. Phys}, 1988, 657-687.

\bibitem{LP15} F. Linares and G. Ponce, Introduction to Nonlinear Dispersive Equations, {\it Springer, Universitext (UTX)},
2015.

\bibitem{M12} L. Molinet, Sharp ill-posedness results for the KdV and mKdV equations on the torus, 
{\em Adv. Math.} 230 (4--6) (2012), 1895--1930.

\bibitem{MPV} L. Molinet, D. Pilod, S. Vento, On unconditional well-posedness for the periodic modified Korteweg de Vries equation, 
J. Math. Soc. Japan 71 (2019), no. 1, 147--201.

\bibitem{KO} S. Kwon, T. Oh, On Unconditional Well-Posedness of Modified KdV, {\it IMRN}, Volume 2012, Issue 15, 2012, 3509--3534.


\bibitem{OW21} T. Oh and Y. Wang, On global well-posedness of the modified KdV 
equation in modulation spaces, 
{\em Discrete Cont. Dyn. Syst.} 41 (6) (2021), 2971--2992.

\bibitem{ONRS} A. Nahmod, T. Oh, L. Rey-Bellet and G. Staffilani, Invariant weighted Wiener measures and almost sure global well-posedness for the periodic derivative NLS,  {\it JEMS}, 14(4), 1275--1330 (2012).

\bibitem{Pa1997} R. Palais, The symmetries of solitons, 
{\em Bulletin of the American Mathematical Society} 34 (1997), 
339-403.


\bibitem{SaTe1976} J. C. Saut and R. Temam, Remarks on the Korteweg-De Vries equation. {\em Israel Journal of Mathematics} Vol. 24, No. 1, 1976.  



\bibitem{TeUh1998} C.-L. Terng and K. Uhlenbeck, 
Poisson actions and scattering theory for integrable systems, 
{\em Surveys in Differential Geometry}, 4 (1998), 315-402.


\bibitem{Ts81} M. Tsutsumi, 
Weighted Sobolev spaces and rapidly decreasing solutions of some nonlinear dispersive wave equations,
{\em J. Diff. Equ.} 42 (2) (1981), 260--281.

\bibitem{Tz} N. Tzvetkov, Invariant measures for the defocusing Nonlinear Schr\"odinger equation, {\it Annales de l'Institut Fourier}, Volume 58 (2008) no. 7, pp. 2543--2604.

\bibitem{TV1} N. Tzvetkov, N. Visciglia, Gaussian measures associated to the higher order conservation laws of the Benjamin-Ono equation, 
{\it Ann. Sci. \'Ec. Norm. Sup\'er.} (4) 46 (2013), no. 2, 249--299.

\bibitem{TV2}  N. Tzvetkov, N. Visciglia, Invariant measures and long-time behavior for the Benjamin-Ono equation, {\it 
Int. Math. Res. Not. IMRN}, 2014, no. 17, 4679--4714.

\bibitem{TV3} N. Tzvetkov, N. Visciglia,  Invariant measures and long time behavior for the Benjamin-Ono equation II, 
{\it J. Math. Pures Appl.} (9) 103 (2015), no. 1, 102--141.

\bibitem{TV13} N. Tzvetkov and N. Visciglia, Gaussian measures associated to the higher order conservation laws of the Benjamin-Ono equation, {\it  Ann. Scient. \'Ec. Norm. Sup.}, 
4 e s\'erie, t. 46, 2013, p. 249--299.


\bibitem{ZaSh71} V. E. Zakharov and A. B. Shabat, Exact theory of two-dimensional self-focusing and one-dimensional selfmodulation
of waves in nonlinear media, 
{\em \'{E}ksper. Teoret. Fiz.} 61 (1971), 118-134.


\bibitem{Zh2001} P.E. Zhidkov, On an infinite sequence of invariant measures for the cubic nonlinear Schr\"odinger equation, 
{\em Int. J. Math. Math. Sci.} 28 (7) (2001), 375--394.

\bibitem{ZhLN} P.E. Zhidkov, Korteweg-de Vries and Nonlinear Schr\"odinger Equations: Qualitative Theory. {\it Lecture Notes in Mathematics}, (LNM, volume 1756).


\bibitem{Zh1989} X. Zhou, 
Direct and inverse scattering transforms with arbitrary spectral singularities, 
{\em Comm. Pure Appl. Math.} 42 (1989), 895-938.


\bibitem{Zh1998} X. Zhou, 
L2-Sobolev space bijectivity of the scattering and inverse scattering transforms, 
{\em Comm. Pure Appl. Math.} 51 (1998), 697-731.








\end{thebibliography}
\end{document}